\documentclass{amsart}
\usepackage{amsmath,amssymb,amsthm,amsfonts,amscd,mathrsfs,mathtools}
\usepackage{tikz-cd}
\usepackage[hyphens]{url}
\usepackage{caption}
\usepackage{hyperref}
\usepackage{overpic}
\usetikzlibrary{matrix,calc,positioning,arrows,decorations.pathreplacing,decorations.markings,patterns}
\tikzset{->-/.style={decoration={
			markings,
			mark=at position #1 with {\arrow{>}}},postaction={decorate}}}
\tikzset{-<-/.style={decoration={
					markings,
					mark=at position #1 with {\arrow{<}}},postaction={decorate}}}

\theoremstyle{plain}
    \newtheorem{thm}{Theorem}[section]
    \newtheorem{lem}[thm]   {Lemma}
    \newtheorem{cor}[thm]   {Corollary}
    \newtheorem{prop}[thm]  {Proposition}
    \newtheorem{que}[thm] {Question}

    \newtheorem{athm}{Theorem}
    
    \newtheorem{acor}{Corollary}[athm]

\theoremstyle{definition}
    \newtheorem{defn}[thm]  {Definition}
    
    \newtheorem{ex}[thm]{Example}
    \newtheorem{rem}[thm]{Remark}
    \newtheorem{nota}[thm]{Notation}

\makeatletter
\def\paragraph{\@startsection{paragraph}{4}%
  \z@\z@{-\fontdimen2\font}%
  {\normalfont\bfseries}}
\makeatother

\newcommand{\cA}{\mathcal{A}} 
\newcommand{\ua}{\underline{a}} 
\newcommand{\cR}{\mathcal{R}} 
\newcommand{\fb}{\mathfrak{b}} 
\newcommand{\cB}{\mathcal{B}} 
\newcommand{\reB}{\bar{B}} 
\newcommand{\C}{\mathbb{C}} 
\newcommand{\bfc}{\mathbf{c}} 
\newcommand{\bfC}{\mathbf{C}} 
\newcommand{\bu}{\mathbf{u}} 
\newcommand{\ud}{\underline{d}} 
\newcommand{\bd}{\mathbf{d}} 
\newcommand{\be}{\mathbf{e}} 
\newcommand{\fe}{\mathfrak{e}} 
\newcommand{\bF}{\mathbb{F}} 
\newcommand{\cF}{\mathcal{F}} 
\newcommand{\G}{G} 
\newcommand{\cH}{\mathcal{H}} 
\newcommand{\cJ}{\mathcal{K}} 
\newcommand{\bK}{\mathbf{K}} 
\newcommand{\cM}{\mathcal{M}} 
\newcommand{\bM}{\mathbb{M}} 
\newcommand{\uP}{\underline{P}} 
\newcommand{\fp}{\mathfrak{p}} 
\newcommand{\Q}{\mathbb{Q}} 
\newcommand{\R}{\mathbb{R}} 
\newcommand{\bfR}{\mathbf{R}} 
\newcommand{\fS}{\mathfrak{S}} 
\newcommand{\ft}{\mathfrak{t}} 
\newcommand{\cT}{\mathcal{T}} 
\newcommand{\fv}{\mathfrak{v}} 
\newcommand{\fy}{\mathfrak{y}} 
\newcommand{\Z}{\mathbb{Z}} 

\newcommand*{\conf}[2][\bullet]{{C}_{#1}(#2)^{\infty}}
\newcommand*{\confred}[2][\bullet]{\widetilde{C}_{#1}(#2)^{\infty}}
\newcommand*{\redchains}[2][*]{\widetilde{\Ch}_{#1}\pa{#2}}
\newcommand*{\redcochains}[2][*]{\widetilde{\Ch}^{#1}\pa{#2}}

\newcommand{\sgn}{\mathrm{sgn}}
\newcommand{\UMor}{\mathrm{UMor}}
\newcommand{\uzero}{\underline{0}}
\newcommand{\Ss}{\mathbf{ss}}

\newcommand{\ab}{\mathrm{ab}} 
\newcommand{\Aut}{\mathrm{Aut}} 
\newcommand{\id}{\mathrm{id}} 

\newcommand{\sgone}{\Sigma_{g,1}}
\newcommand{\ggone}{\Gamma_{g,1}}
\newcommand{\Homeo}{\mathrm{Homeo}}
\newcommand{\ttau}{\tilde\theta} 

\newcommand{\Fp}{\mathbb{F}_p}
\newcommand{\twoB}{2\! B} 

\newcommand{\del}{\partial}
\newcommand{\sth}{\textsuperscript{th} }

\newcommand{\sort}{\underline{\mathrm{sort}}}
\newcommand{\Shuf}{\mathfrak{Shuf}} 

\newcommand{\hor}{\mathrm{hor}}
\newcommand{\ver}{\mathrm{ver}}

\newcommand{\Ch}{\mathrm{Ch}} 
\newcommand{\Tor}{\mathrm{Tor}}
\newcommand{\Ext}{\mathrm{Ext}}
\newcommand{\Hom}{\mathrm{Hom}}
\newcommand{\Imm}{\mathrm{Im}}

\newcommand{\pa}[1]{\mathopen{}\left(#1\right)\mathclose{}}     
\newcommand{\sca}[1]{\left<#1\right>}
\newcommand{\set}[1]{\mathopen{}\left\{#1\right\}\mathclose{}}  
\newcommand{\floor}[1]{\lfloor #1 \rfloor}

\title{Homology of configuration spaces of surfaces modulo an odd prime}

\author{Andrea Bianchi}
\address{\parbox{.9\linewidth}{Max Planck Institute for Mathematics,\\
Vivatsgasse 7, Bonn, Germany}
}
\email{bianchi@mpim-bonn.mpg.de}

\author{Andreas Stavrou}
\address{\parbox{.9\linewidth}{Department of Mathematics,
University of Chicago,\\
5734 S. University Avenue,
Chicago, IL, 60637, United States of America}}
\email{andreasstavrou@uchicago.edu}


\begin{document}

\date{\today}
\begin{abstract}
For a compact orientable surface $\Sigma_{g,1}$ of genus $g$ with one boundary component and for an odd prime number $p$, we study the homology of the unordered configuration spaces $C_\bullet(\Sigma_{g,1}):=\coprod_{n\ge0}C_n(\Sigma_{g,1})$ with coefficients in $\mathbb{F}_p$.
We describe $H_*(C_\bullet(\Sigma_{g,1});\mathbb{F}_p)$ as a bigraded module over the Pontryagin ring $H_*(C_\bullet(D);\mathbb{F}_p)$, where $D$ is a disc, and compute in particular the bigraded dimension over $\mathbb{F}_p$.
We also consider the action of the mapping class group $\Gamma_{g,1}$,
and prove that the mod-$p$ Johnson kernel $\mathcal{K}_{g,1}(p)\subseteq\Gamma_{g,1}$ is the kernel of the action on $H_*(C_\bullet(\Sigma_{g,1};\mathbb{F}_p))$.
\end{abstract}

\maketitle

\section{Introduction and statement of results}
We are interested in the homology of configuration spaces of surfaces, as plain abelian groups and as representations of the mapping class group.
For a topological space $X$ and $n\ge0$, the $n$\sth \emph{unordered} configuration space of $X$, denoted $C_n(X)$, is defined as the quotient
\[
 C_n(X)=\set{(p_1,\dots,p_n)\in X^n\,|\,p_i\neq p_j\mbox{ forall }  i\neq j}\,/\,\fS_n,
\]
where $\fS_n$ acts on the space of $n$-tuples of distinct points by relabeling them.

We denote by $\sgone$ a compact, connected, orientable surface of genus $g\ge0$ with one boundary curve. The group $\Homeo(\sgone,\del\sgone)$ 
of homeomorphisms of the surface $\sgone$ fixing the boundary pointwise acts on $C_n(\sgone)$, and hence on the homology $H_*(C_n(\sgone);R)$ with coefficients in a commutative ring $R$; the latter action descends to an action of the mapping class group $\ggone:=\pi_0\Homeo(\sgone,\del\sgone)$ on $H_*(C_n(\sgone);R)$.
We are interested in the following two questions.
\begin{que}
 What $R$-module is $H_i(C_n(\sgone);R)$? Specifically, if $R$ is a field, what is the dimension of $H_i(C_n(\sgone);R)$ over $R$?
\end{que}
\begin{que}
 How does the mapping class group $\ggone$ act on $H_i(C_n(\sgone);R)$? What is the kernel of this action?
\end{que}
For $R$ being $\bF_2$ or $\Q$, the first question is answered in \cite{BCT} and \cite{BC,DCK}, respectively, and an answer to both questions can be found in \cite{Bianchi} and \cite{Stavrou}, respectively.
In this paper we will mainly focus on the cases $R=\Z$ and $R=\Fp$, for $p$ an odd prime.

\subsection{Mapping class group action}
Our first main result is the following.
\begin{athm}\label{thm:A}
Let $n\ge2$ and let $R$ be a commutative ring.
\begin{enumerate}
    \item The Johnson kernel $\cJ_{g,1}\subseteq\Gamma_{g,1}$, i.e. the subgroup generated by separating Dehn twists, acts trivially on $H_*(C_n(\sgone);R)$.
    \item For $R=\Z$ and $R=\Q$, the kernel of the action of $\ggone$ on $H_*(C_n(\sgone);R)$ is the Johnson kernel.
    \item For $p$ an odd prime, the kernel of the action of $\ggone$ on $H_*(C_n(\sgone);\Fp)$ is the mod-$p$ Johnson kernel $\cJ_{g,1}(p)\subseteq\Gamma_{g,1}$, i.e. the subgroup generated by separating Dehn twists and $p$\sth powers of all Dehn twists.
\end{enumerate}
\end{athm}
For $R=\bF_2$, the kernel of the action of $\ggone$ on $H_*(C_n(\sgone);\bF_2)$ is $\cT_{g,1}(2)$ \cite{Bianchi}. The fact that $\cJ_{g,1}$ acts trivially on $H_*(C_n(\sgone);\Q)$ already appears in \cite{Stavrou}.

\subsection{Homology of configurations as Ext groups}
Let $R$ be a commutative ring.
Our second, main result, Theorem \ref{thm:B}, describes the direct sum
\[
H_*(C_\bullet(\sgone);R)\cong\bigoplus_{n\ge0}H_*(C_n(\sgone);R)
\]
as a tensor product of simpler terms, and in particular as a module over the Pontryagin ring
\[
H_*(C_\bullet(D);R)\cong\bigoplus_{n\ge0}H_*(C_n(D);R),
\]
where $D\cong\Sigma_{0,1}$ is a disc. Here we use that $C_\bullet(D)=\coprod_{n\ge0}C_n(D)$ has an $E_2$-algebra structure induced by embedding several copies of $D$ inside $D$, and that $C_\bullet(\sgone)=\coprod_{n\ge0}C_n(\sgone)$ has a structure of $E_1$-module over $C_\bullet(D)$ induced by identifying $\sgone\cong\sgone\cup_{\frac12\partial}D$ and by embedding several copies of $D$ in $D\subset\sgone$.

In Theorem \ref{thm:B} we make use of the bigrading $(\bullet,\star)$ given by the two gradings $\bullet$ (called \emph{weight}) and $\star$ (called \emph{bar-degree}); the grading $*$ can be recovered as $\bullet+\star$.

We work in the abelian category $\mathrm{Mod}_R^{\Z}$ of weighted $R$-modules, i.e. of $\Z$-indexed families of $R$-modules, endowed with the tensor product given by Day convolution with respect to sum in $\Z$ and tensor product of $R$-modules.
We let $\Gamma_R(y)$ denote the free divided power $R$-algebra generated by a variable $y$ of weight $-2$ (the negative weight of $y$ ensures that the terms in Theorem \ref{thm:B} are in non-negative weight).
We let $\cH_g^R$ denote a free $R$-module of rank $2g$ concentrated in bigrading $(\bullet,\star)=(2,0)$, and we let $R[\cH_g^R]$ denote the free symmetric $R$-algebra generated by $\cH_g^R$; we consider $R[\cH_g^R]$ just as a bigraded $R$-module.
Finally, in Definition \ref{defn:bMg} we introduce, for $g\ge0$, an explicit $\Gamma_R(y)$-module $\bM_g^R$: we remark here that $\bM_g^R$ is concentrated in non-positive weights and is finitely generated and free over $R$.

\begin{athm}\label{thm:B}
Let $R$ be a commutative ring, and set $*=\bullet+\star$.
\begin{enumerate}
 \item There is an isomorphism of bigraded rings
 \[
 H_*(C_\bullet(D);R)\cong R[\epsilon]\otimes_R \Ext^\star_{\Gamma_R(y)}(R,R).
 \]
Here $R[\epsilon]$ is the polynomial ring in one variable $\epsilon$ in bigrading $(1,-1)$, and we endow $\Ext^\star_{\Gamma_R(y)}(R,R)$ with the Yoneda product.
\item There is an isomorphism of bigraded
$R[\epsilon]\otimes_R\Ext^\star_{\Gamma_R(y)}(R,R)$-modules
\[
 H_*(C_\bullet(\sgone);R)\cong R[\epsilon]\otimes_R\Ext^\star_{\Gamma_R(y)}(\bM_g^R,R)\otimes_R R[\cH_g^R].
\]
Here $R[\epsilon]$ acts on itself by ring multiplication, whereas $\Ext^\star_{\Gamma_R(y)}(R,R)$ acts by the Yoneda product on $\Ext^\star_{\Gamma_R(y)}(\bM_g^R,R)$.
\end{enumerate}
\end{athm}
Part (1) of Theorem \ref{thm:B}, in the stated form, was first proved by Markaryan \cite{Markaryan}, and is preceded by classical computations of $H_*(C_\bullet(D);R)$ as bigraded vector spaces for $R$ being a field of characteristic zero \cite{Arnold}, two \cite{Fuchs} or an odd prime $p$ \cite{Weinstein}, and in all characteristics a description of the ring structure as well as of the action of the dual Steenrod algebra is given in \cite[Chapter III]{CLM}. Similar results comparing twisted homology of braid groups and cohomology of augmented algebras can be found in \cite{Callegaro, ETW, Hoang}.
For $R=\bF_2$ or $\Q$, Theorem \ref{thm:B} recovers computations of $H_*(C_\bullet(\sgone);R)$ already present in the literature \cite{BCT,Bianchi,BC,DCK,Stavrou}.

\begin{rem}
\label{rem:Extnegative}
 In Theorem \ref{thm:B}, and later in Theorem \ref{thm:C}, $\Ext^\star_{\Gamma_R(y)}(-,R)$ is the collection of the right derived functors of $\Hom_{\Gamma_R(y)}(-,R)$. The latter is
 part of the internal hom-functor $\Hom_{\Gamma_R(y)}(-,-)$ of the abelian category $\mathrm{Mod}_{\Gamma_R(y)}(\mathrm{Mod}_R^\Z)$ of weighted $\Gamma_R(y)$-modules in $\mathrm{Mod}_R^\Z$, endowed with the monoidal structure of weighted tensor product over $\Gamma_R(y)$.
 We adopt the convention that $\Ext$-groups are concentrated in bar-degree $\star\le0$. So we denote by $\Ext^0$ the functor $\Hom$, by $\Ext^{-1}$ the first right derived functor of $\Hom$, and so on.
\end{rem}

Theorem \ref{thm:B} recovers the following results due to \cite{McDuff, Segal} (see also \cite{ORW}), and due to \cite{KMT} (also implicit in \cite{BC,DCK,Stavrou}).
\begin{acor}[Split homological stability and rational extremal stability]
\label{cor:splitstabextremalstab}
The following statements hold.
\begin{enumerate}
 \item The map $\epsilon\cdot -\colon H_i(C_n(\sgone);R)\to H_i(C_{n+1}(\sgone);R)$ is split injective and is an isomorphism for $i\le \frac 12n$.
 \item For $i\ge0$, there are two polynomials $P_{g,i}^1(s),P_{g,i}^{-1}(s)\in\Q[s]$ of degree $2g-1$ such that for all $n\ge 2g+i+1$ there is an equality
 \[
 \dim_{\Q}H_{n-i}(C_n(\sgone);\Q)= P_{g,i}^{(-1)^n}(n).
 \]
\end{enumerate}
\end{acor}
For an original corollary of Theorem \ref{thm:B}, note that for $g,g'\ge0$ we can embed the disjoint union $\sgone\sqcup\Sigma_{g',1}$ inside $\Sigma_{g+g',1}$ as the complement of a pair of pants, and for all $n,n'\ge0$ we obtain an inclusion 
\[
C_n(\sgone)\times C_{n'}(\Sigma_{g',1})\hookrightarrow C_{n+n'}(\Sigma_{g+g',1}).
\]
We now set $g'=n'=1$.
\begin{acor}[Filled genus stabilisation]\label{cor:filledgenus}
Let $g,n\ge0$. Then the embedding construction induces a split injective map
\[
 H_*(C_n(\sgone);R)\otimes_R H_1(\Sigma_{1,1};R)\hookrightarrow H_{*+1}(C_{n+1}(\Sigma_{g+1,1});R).
\]
\end{acor}
We remark that, instead, the plain genus stabilisation map 
\[
H_*(C_n(\sgone);R)\to H_{*}(C_n(\Sigma_{g+1,1});R)
\]
is not injective in general; for instance one can identify the map
\[
H_1(C_2(\Sigma_{0,1});\Z)\to H_1(C_2(\Sigma_{1,1});\Z)
\]
with the composite map $\Z\twoheadrightarrow\bF_2\hookrightarrow\bF_2\oplus\Z^2$.

\subsection{Computation of the Ext groups}
We next move to our third, main result of the paper. Theorem \ref{thm:B}
reduces the computation of $H_*(C_\bullet(\sgone);R)$ to the computation of
the ring $\Ext^\star_{\Gamma_R(y)}(R,R)$ and of its modules 
$\Ext^\star_{\Gamma_R(y)}(\bM_g^R,R)$.

Let now $R=\Fp$ with $p$ an odd prime.
Setting $\fy_i=y^{[i]}=y^i/i!$, we have that $\Gamma_{\Fp}(y)$ is isomorphic as a weighted algebra to the truncated polynomial algebra
\[
\Fp[\fy_0,\fy_1,\fy_2,\dots]/(\fy_0^p,\fy_1^p,\fy_2^p\dots),
\]
and consequently we have an isomorphism of bigraded algebras
\[
 \Ext^\star_{\Gamma_{\Fp}(y)}(\Fp,\Fp)\cong \bigotimes_{i=0}^\infty \Ext^\star_{\Fp[\fy_i]/(\fy_i^p)}(\Fp,\Fp)\cong \bigotimes_{i=0}^\infty\Fp[\alpha_i,\beta_i],
\]
where
\[
\alpha_i\in \Ext^1_{\Fp[\fy_i]/(\fy_i^p)}(\Fp,\Fp)\ \mbox{ and }\ \beta_i\in \Ext^2_{\Fp[\fy_i]/(\fy_i^p)}(\Fp,\Fp)
\]
are generators, and $\bigotimes_{i=0}^\infty\Fp[\alpha_i,\beta_i]=\Fp[\alpha_0,\beta_0,\alpha_1,\beta_1,\dots]$
denotes the free graded-commutative algebra generated by the classes $\alpha_0,\beta_0,\alpha_1,\beta_1,\dots$.

We postpone the statement of Theorem \ref{thm:C} to the body of the paper, due to its technical nature, but we briefly discuss its content here.
In Definition \ref{defn:twoB} we introduce for all $u\ge0$ a certain weighted $\Gamma_{\Fp}(y)$-module $B^{\Fp}_u$, and in Proposition \ref{prop:bMgsplitting} we decompose the $\Gamma_{\Fp}(y)$-module $\bM_g^{\Fp}$ as a finite direct sum of weighted shifts of $\Gamma_{\Fp}(y)$-modules of the form $B^{\Fp}_u$. We thus reduce the computation of $\Ext^\star_{\Gamma_{\Fp}(y)}(\bM_g^{\Fp},\Fp)$ to the computation of $\Ext^\star_{\Gamma_{\Fp}(y)}(B^{\Fp}_u,\Fp)$, and the latter is achieved in Theorem \ref{thm:C}.

The following corollaries follow from Theorem \ref{thm:C} and its argument of proof.

\addtocounter{athm}{1}
\setcounter{acor}{0}
\begin{acor}[Qualitative module structure]
\label{acor:qualitative}
Let $p$ be an odd prime, and let $\Fp[\cH_g^{\Fp}]$ be as in Theorem \ref{thm:B}; then $H_*(C_\bullet(\sgone);\Fp)$ is isomorphic, as a module over $H_*(C_\bullet(D);\Fp)\cong\Fp[\epsilon,\alpha_0,\beta_0,\dots]$, to a finite direct sum of modules of the following forms, where $i\ge0$ (tensor products are taken here over $\Fp$):
\[
 \Fp[\epsilon,\alpha_0,\beta_0,\dots]/(\alpha_0,\dots,\beta_{i-1})\otimes\Fp[\cH_g^{\Fp}];\quad
 \Fp[\epsilon,\alpha_0,\beta_0,\dots]/(\alpha_0,\dots,\alpha_{i})\otimes\Fp[\cH_g^{\Fp}].
\]
\end{acor}

\begin{acor}[Near top degree generators]\label{acor:neartop}
Let $g\ge0$ and let $p$ be an odd prime; then $H_*(C_\bullet(\sgone);\Fp)$ is generated as a $H_*(C_\bullet(D);\Fp)$-module by the direct sum
\[
\bigoplus_{n\ge0}H_{n-1}(C_n(\sgone);\Fp)\oplus\bigoplus_{n\ge0}H_{n}(C_n(\sgone);\Fp).
\]
\end{acor}
The same result holds with coefficients in $\bF_2$ or $\Q$; in fact $H_*(C_\bullet(\sgone);\bF_2)$ is a free $H_*(C_\bullet(D);\bF_2)$-module generated by the vector space $\bigoplus_{n\ge0}H_{n}(C_n(\sgone);\bF_2)$, see \cite{Bianchi}.
The terminology ``near top'' alludes to the vanishing of $H_i(C_n(\sgone);\Fp)$ for $i>n$, as $C_n(\sgone)$ is homotopy equivalent to a complex of dimension at most $n$.

\begin{acor}[Genus versus prime]\label{acor:genusvsprime}
 Let $g\ge0$, $\lambda\ge0$, and let $p$ be an odd prime.
 \begin{enumerate}
  \item if $g<p^\lambda$, then $H_*(C_\bullet(\sgone);\Fp)$ is free as a module over the subring
  \[
  \Fp[\epsilon,\alpha_\lambda,\beta_\lambda,\alpha_{\lambda+1},\beta_{\lambda+1},\dots]\otimes \Fp[\cH_g]
  \]
  of the ring $H_*(C_\bullet(D);\Fp)\otimes_{\Fp}\Fp[\cH_g]$;
  \item if $g<(p-1)p^{\lambda-1}$, then $H_*(C_\bullet(\sgone);\Fp)$ is free as a module over the subring
  \[
  \Fp[\epsilon,\beta_{\lambda-1},\alpha_\lambda,\beta_\lambda,\alpha_{\lambda+1},\dots]\otimes_{\Fp} \Fp[\cH_g].
  \]
 \end{enumerate}
\end{acor}

\begin{acor}[Top and top minus 1 homology]\label{acor:topminusone}
 Let $g\ge0$ and $n\ge0$; then $H_n(C_n(\sgone);\Z)$ is a free abelian group, whereas $H_{n-1}(C_n(\sgone);\Z)$ has only $2$-power torsion.
As a consequence, for $p$ an odd prime we have
 \[
  \dim_\Q(H_n(C_n(\sgone);\Q)) = 
  \dim_{\Fp}(H_n(C_n(\sgone);\Fp)).
 \]
\end{acor}

\begin{acor}[$p$-power torsion]\label{acor:ppowertorsion}
 Let $p$ be an odd prime and $n\ge0$; then the $p$-power torsion part of $H_*(C_n(\sgone);\Z)$ coincides with the image of the Bockstein homomorphism
 \[
 H_{*+1}(C_n(\sgone);\Fp)\to H_*(C_n(\sgone);\Z)
 \]
 induced by the short exact sequence $\Z\overset{\cdot p}{\to}\Z\twoheadrightarrow\Fp$; in particular all $p$-power torsion classes are in fact $p$-torsion.
\end{acor}

\begin{acor}[Fake basechange]\label{acor:basechange}
 For a bigraded $\Q$-vector space $V$ and for an odd prime $p$ let $V^\Q_{\Fp}$ be a bigraded $\Fp$-vector space having in each bigrading the same dimension (over $\Fp$) as $V$ (over $\Q$). Then for $g\le p-2$ there is an isomorphism of bigraded $\Fp$-vector spaces
 \[
  H_*(C_\bullet(\sgone);\Fp)\cong \pa{H_*(C_\bullet(\sgone);\Q)}^\Q_{\Fp}\otimes_{\Fp}\Fp[\beta_0,\alpha_1,\beta_1,\alpha_2,\beta_2\dots].
 \]
\end{acor}

\begin{acor}[Low dimensional Betti numbers]\label{acor:lowdimbetti}
 Let $p$ be an odd prime and $g,n\ge0$. Then for $i< \max\set{2p-2,g+p}$, $H_i(C_n(\sgone);\Z)$ has no $p$-torsion. Equivalently, for $i<\max\set{2p-2,g+p}$ we have
 \[
  \dim_\Q(H_i(C_n(\sgone);\Q) = 
  \dim_{\Fp}(H_i(C_n(\sgone);\Fp)).
 \]
\end{acor}
Corollaries \ref{acor:topminusone} and \ref{acor:lowdimbetti} imply that $H_*(C_n(\sgone);\Z)$ has no $p$-torsion for $n<\max\set{2p,g+p+2}$. This partially confirms the prediction of \cite[Remark 3.7]{ChenZhang}.

\begin{acor}[Extremal asymptotic growth]\label{acor:extremalquasistab}
Let $p$ be an odd prime, and let $g\ge1$ and $i\ge0$. Then there exist constants $0<\bfc_{g,i}<\bfC_{g,i}$ such that, for $n\ge i$,
\[
 \bfc_{g,i} \log_p(n) ^i n^{2g-1}< \dim_{\Fp}H_{n-i}(C_n(\sgone);\Fp) < \bfC_{g,i} \log_p(n) ^i n^{2g-1}.
\]
\end{acor}

\subsection{Related work}
Configuration spaces of manifolds $M$ and their homology are a classical topic in algebraic topology. Arnold \cite{Arnold}, Fuchs \cite{Fuchs} and Weinstein \cite{Weinstein} computed $H_*(C_n(\R^2);R)$ for $R$ being $\Q$, $\bF_2$ and $\bF_p$ respectively. The computation of $H_*(C_n(\R^d);R)$ for any ring $R$ is due to Cohen \cite[Chapter III]{CLM}.

B\"odigheimer--Cohen--Taylor \cite{BCT} compute explicitly $H_*(C_n(M);R)$, as a bigraded vector space, for $d\ge3$ odd and any field $R$, and for $d\ge2$ and  $R=\bF_2$: in the case $M=\sgone$, their result only settles the case $R=\bF_2$. The first computation of $H_*(C_n(\sgone);\Q)$ as a graded vector space is due to B\"odigheimer--Cohen \cite{BC}; a computation of the rational cohomology of $C_n(M)$, for $M$ any compact, connected, possibly non-orientable surface, is given in \cite{DCK}. A description of $H_*(C_n(M);\Q)$ for a generic even-dimensional manifold $M$ was given by F\'elix--Thomas \cite{FelixThomas} and by Knudsen \cite{Knudsen:BettiNumbers}, and a detailed computation of rational Betti numbers in the case of surfaces was done by Drummond-Cole--Knudsen \cite{DCK}.

For a generic even-dimensional manifold $M$ and an odd prime $p$, very little is known about $H_*(C_n(M);\Fp)$. To the best of our knowledge, this article contains the first complete computation for an orientable, even-dimensional manifold $M$ which is not an open subset of another manifold of the form $M'\times \R$.

We mention that Brantner--Hahn--Knudsen \cite[Corollary 1.11]{BHK} have computed $H_*(C_p(\sgone);\Fp)$ as a bigraded vector space, obtaining the special case of Corollary \ref{acor:lowdimbetti} for $n=p$. Their strategy is to first consider the Morava $E$-theory of $p$ at height $h\ge0$, denoted $E_h$, and compute $E^*_h(C_p(\sgone))$; then deduce a computation of the Morava $K$-theory $K(h)^*(C_p(\sgone))$; and finally deduce a computation of $H^*(C_p(\sgone);\Fp)$ by observing that, for $h$ large enough, the $K(h)$-based Atiyah--Hirzebruch spectral sequence collapses. A simplified argument was then given by Chen--Zhang \cite{ChenZhang}, building on \cite{Zhang, Knudsen:HEA}.
The methods used in this paper are more elementary, and whether they can be applied to the computation of other cohomology theories on configuration spaces of surfaces seems to be a difficult question.

The action of the mapping class group on the homology of \emph{ordered} configuration spaces was treated in \cite{BMW}, building on work of Moriyama \cite{Moriyama}; in the unordered case, the problem has been studied with $\bF_2$-coefficients by the first author \cite{Bianchi}, and with $\Q$-coefficients by Looijenga \cite{Looijenga} and by the second author \cite{Stavrou}.
\subsection{Acknowledgments}
The first author thanks Lorenzo Guerra for a discussion about divided power algebras, and Jeremy Miller for a discussions about homology operations on configuration spaces. Both authors thank Oscar Randal-Williams for several discussions on the topic and for his supervision of the second author. Both authors thank Zachary Himes, Jeremy Miller, Jan Steinebrunner, Nathalie Wahl, Adela Zhang and the referee for comments on a first draft of the article.

The article was prepared while A.B. was affiliated with the University of Copenhagen and A.S. with the University of Cambridge.
A. Bianchi was supported by the European Research Council under the European Union’s Horizon2020 research and innovation programme (grant agreement No. 772960), and by the Danish National Research Foundation through the Copenhagen Centre for Geometry and Topology (DNRF151).
A. Stavrou was funded by a studentship of the Engineering and Physical Sciences Research Council(project reference: 2261124).

\tableofcontents
\part{The action of the Johnson filtration}\label{part:1}
\section{Partially compactified configurations}
In this paper, we work in the category of compactly generated Hausdorff spaces.
Let $(X,X')$ be a pair of spaces with $X'\subset X$ closed. For $n\ge0$ we define the following two closed subspaces of $X^n$:
\begin{align*}
    \Delta_n(X)&:=\{(x_1,...,x_n)\in X^n\,|\, x_i=x_j\text{ for some } i\neq j\},\\
    A_n(X,X')&:=\{(x_1,...,x_n)\in X^n\,|\,x_i\in X' \text{ for some } i\}.
\end{align*}
Any continuous map $f:X\to Y$ induces a map $f^n\colon X^n\to Y^n$ that restricts to a map 
\[
\Delta_n(f)\colon\Delta_n(X)\to \Delta_n(Y).
\]
If $f$ is moreover a map of pairs $(X,X')\to (Y,Y')$, then $f^n$ also restricts to a map
\[
A_n(f)\colon A_n(X,X')\to A_n(Y,Y').
\]
Finally, the permutation action of $\mathfrak{S}_n$ on $X^n$ preserves the subsets $\Delta_n(X)$ and $A_n(X,X')$.
\begin{defn}\label{defn:CnXinfty}
    For $n\ge 1$ and a pair of spaces $(X,X')$, we define the quotient
$$\conf[n]{X,X'}:=\big(X^n/(\Delta_n(X)\cup A_n(X,X'))\big)/\fS_n.$$ 
 We denote by $\infty$ the image of $\Delta_n(X)\cup A_n(X,X')$ in the quotient,
which serves as basepoint for $\conf[n]{X,X'}$. For $n=0$, we define $\conf[0]{X,X'}$ to be the discrete space $\{\infty,\emptyset\}$ based at $\infty$. We write the wedge of pointed spaces $\conf{X,X'}:=\bigvee_{n\ge 0}\conf[n]{X,X'}$.

The image of $(x_1,...,x_n)$ in $\conf[n]{X,X'}$ is denoted by $\left[x_1,...,x_n\right]$; we often denote by $s$ a generic point in $\conf[n]{X,X'}$.

In the case of a pointed space $(X,x_0)$, we also write $\confred[n]{X}:=\conf[n]{X,\{x_0\}}$, and $\confred{X}=\conf{X,\{x_0\}}$.
\end{defn}
Observe that $\conf[n]{X,X'}$ contains $C_n(X\setminus X')$ and, if $X$ is compact, then $\conf[n]{X,X'}$ is also compact; this justifies the terminology of Definition \ref{defn:CnXinfty}.
Furthermore, 
$\conf[n]{X,X'}$ is homeomorphic to $\confred[n]{X/X'}$, where the basepoint of $X/X'$ is the image of $X'$. A map of pairs $f:(X,X')\to (Y,Y')$ induces a map $\conf[n]{f}:\conf[n]{X,X'}\to \conf[n]{Y,Y'}$.

\subsection{Connection to \texorpdfstring{$C_n(M)$}{CnM}}\label{sec:Poincareduality}
We let $M$ be a compact manifold with boundary throughout the subsection.
\begin{nota}\label{nota:confnM}
We will use the shorthand $\conf[n]{M}:=\conf[n]{M,\partial M}$, which coincides with the one-point compactification of $C_n(\mathring{M})$, where $\mathring{M}\subseteq M$ is the interior of $M$. 
\end{nota}
The group $\Homeo(M, \partial M)$ of orientation-preserving homeomorphisms of $M$ fixing $\partial M$ pointwise acts on both $C_n(\mathring{M})$ and $\conf[n]{M}$ by homeomorphisms. 
\begin{prop}
\label{prop:equivariantPD}
  Suppose $M$ is oriented and of even dimension $2d$. Then for every $n,i\ge 0$ and every commutative ring $R$, there is a $\Homeo(M,\partial M)$-equivariant\footnote{Strictly speaking, we assume that $\Homeo(M, \partial M)$ acts via the adjoint action on one of the two $R$-modules.} isomorphism $H_i(C_n(M);R)\cong \widetilde{H}^{dn-i}(\conf[n]{M};R)$.
\end{prop}
\begin{proof} A standard application of Poincar\'e-Lefschetz duality.
\end{proof}

Guided by this proposition, in Sections \ref{sec:confbouquet} and \ref{sec:confsurf} we compute a cellular chain complex for $\confred{\Sigma_{g,1}}$, admitting an action of the mapping class group $\Gamma_{g,1}:=\pi_0\Homeo(\Sigma_{g,1},\partial\Sigma_{g,1})$ on it.

\subsection{Superposition of configurations}
\label{subsec:superposition}
\begin{defn}
\label{defn:muproduct}
For $m,n\ge 0$ and for a pair of spaces $(X,X')$ we define the map
\begin{align*}
    \mu_{m,n}:\conf[m]{X,X'}\times \conf[n]{X,X'}&\to \conf[m+n]{X,X'}\\
    (\left[x_1,\dots,x_m\right],\left[y_1,\dots,y_n\right])&\mapsto \left[x_1,\dots,x_m,y_1,\dots,y_m\right].
\end{align*}
Observe that $\mu_{m,n}(s,s')=\infty$ if either $s$ or $s'$ is $\infty$, and that $\mu_{n,0}(s,\emptyset)=\mu_{0,n}(\emptyset,s)=s$ and $\mu_{n,0}(s,\infty)=\mu_{0,n}(\infty,s)=\infty$.

We let $\mu\colon \conf{X,X'}\times\conf{X,X'}\to\conf{X,X'}$ be the map restricting to $\mu_{m,n}$ on $\conf[m]{X,X'}\times \conf[n]{X,X'}$.
\end{defn}
\begin{rem}
   It follows that the pair $\big(\conf{X,X'},\mu\big)$ is an abelian topological monoid with unit $\emptyset$ and absorbing element $\infty$. 
\end{rem}

\section{Configurations on bouquets of circles}
\label{sec:confbouquet}
We introduce a model for bouquets of circles. The setup of this section is analogous to \cite{Moriyama}.

\begin{defn}\label{defn:bouquets}
For $k\ge 1$, construct the bouquet of $k$ circles as the quotient $$V_k:=\coprod_{i=1}^k[0,1]\times \{i\}/\sim$$ where $\sim$ collapses the endpoints of all intervals to a single point $*$. Picking the standard orientation on $[0,1]$, we orient each circle in $V_k\cong\bigvee_{i=1}^kS^1$. We let $\gamma_1,\dots,\gamma_k\in \pi_1(V_k,*)$ be the standard generators, where $\gamma_i$ is represented by the loops $t\mapsto [t,i]$. We use the elements $\gamma_i$ to identify $\pi_1(V_k,*)$ with the standard, free group $\Z^{*k}$ on $k$ generators.
\end{defn}

\subsection{A cell stratification for \texorpdfstring{$\confred{V_k}$}{tCbullet(Vk)infty}}
\label{subsec:cellstratbouquet}

For $n\ge 0$, our model for the standard simplex $\Delta^n$ is
\[
\Delta^n=\{(t_1,\dots,t_n)\in \R^n:0\le t_1\le t_2\le \dots\le t_n\le 1 \},
\]
oriented as a subspace of $\R^n$.

\begin{defn}
Let $k\ge 1$. For a vector $\fv=(v_1,\dots,v_k)\in (\Z_{\ge 0})^{\times k}$, we define
\begin{itemize}\setlength\itemsep{.1cm}
    \item $d(\fv)=\sum_{i=1}^k v_i$, the \emph{dimension} of $\fv$,
    \item the polysimplex $\Delta^{\fv}:=\prod_{i=1}^k \Delta^{v_i}$,
    \item and the map $\Phi_{\fv}:\Delta^{\fv}\to \confred[d(\fv)]{V_k}$, with
     \begin{align*}
        \Big( (t^{(1)}_{1},&\dots,t^{(1)}_{v_1}),(t^{(2)}_1,\dots,t^{(2)}_{v_2}),\dots, (t^{(k)}_1,\dots,t^{(k)}_{v_k})\Big) \\
        &\mapsto \left[(t^{(1)}_1,1),\dots,(t^{(1)}_{v_1},1), (t^{(2)}_1,2),\dots,(t^{(2)}_{v_2},2),\dots,(t^{(k)}_1,k),\dots,(t^{(k)}_{v_k},k)\right].
    \end{align*}
\end{itemize}
\end{defn}
The image of $\Phi_{\fv}$ is the space of configurations on $V_k$ with precisely $v_i$ points on the $i$\sth circle, for $1\le i\le k$: more precisely, $\Phi_{\fv}$ restricts to a topological embedding of the interior of $\Delta^{\fv}$, with image the subspace $\prod_{i=1}^kC_{v_i}((0,1)\times\set{i})\subset\confred[d(\fv)]{V_k}$, and it maps $\del\Delta^\fv$ constantly to the point $\infty\in\confred[d(\fv)]{V_k}$.

Conversely, for all $n\ge 0$, each configuration $s\in \confred[n]{V_k}\setminus\set{\infty}$ lies in the image of $\Phi_{\fv}$ for exactly one vector $\fv$ with $d(\fv)=n$.
We obtain the following proposition.

\begin{prop}\label{prop:thecellularstructureforbouquets}
    The space $\confred{V_k}$ admits a cellular decomposition with
    \begin{itemize}
    \item a $0$-cell given by the point $\infty$;
    \item for each $\fv\in (\Z_{\ge 0})^{\times k}$, a $d(\fv)$-dimensional open cell $e_\fv:=\Phi_{\fv}(\mathring{\Delta}^\fv)$, whose boundary is attached to the $0$-cell $\infty$. For $\fv=0$ we get a further $0$-cell.
    \end{itemize}
In particular, for $n\ge 0$, the subspace $\confred[n]{V_k}\subset\confred{V_k}$ is homeomorphic to the wedge of $\binom{n+k-1}{k-1}$
spheres of dimension $n$. 
\end{prop}
We next consider the cellular chain complex of $\confred{V_k}$.
\begin{defn}
 \label{defn:UMor}
We denote the reduced\footnote{that is, relative to the $0$-cell corresponding to $\infty$.} cellular chain complex corresponding to the cellular decomposition from Proposition \ref{prop:thecellularstructureforbouquets} by
$$
\UMor_\bullet(k):=\redchains{\confred{V_k}}\cong\bigoplus_{n\ge0}\redchains{\confred[n]{V_k}};
$$
we view $\UMor_\bullet(k)$ as a \emph{weighted} chain complex, and in particular as a weighted abelian group, in the following way: for $n\ge0$ we define the summand $\UMor_{-n}(k)$ as $\redchains{\confred[n]{V_k}}$, which in fact only consists of $\redchains[n]{\confred[n]{V_k}}$; the summand $\UMor_{-n}(k)$ is in non-positive weight $\bullet=-n$.
\end{defn}
The previous definition is a version of a construction of Moriyama \cite{Moriyama}, in the context of unordered configuration spaces: this explains the notation ``$\UMor$''.
A direct consequence of Proposition \ref{prop:thecellularstructureforbouquets} is that $\UMor_\bullet(k)$ has zero differential and is a free abelian group generated by the cells $e_{\fv}$, for varying
$\fv\in (\Z_{\ge 0})^{\times k}$.
Therefore $\UMor_\bullet(k)$ is also isomorphic to
\[
\bigoplus_{n\ge 0}\widetilde{H}_*(\confred[n]{V_k})\cong\bigoplus_{n\ge 0}\widetilde{H}_n(\confred[n]{V_k}).
\]
The reason why we put $\UMor_\bullet(2k)$ in non-positive weights will become clear when approaching the proof of Theorem \ref{thm:B}. For the moment, observe that by Poincar\'e duality the homology group
$\widetilde{H}_n(\confred[n]{V_k})$ is isomorphic to the \emph{cohomology} group $H^0(C_n(\set{1,\dots,k}\times(0,1)))$: this is the dual of $H_0(C_n(\set{1,\dots,k}\times(0,1)))$, and our convention puts the latter in weight $n$, and the former in weight $-n$.

For $1\le i\le k$ let now $\fe_i=(0,\dots,0,1,0,\dots,0)\in (\Z_{\ge 0})^{\times k}$
be the unit vector in the $i$\sth direction,
and write $e(i,n):=e_{n\fe_i}$ for the cell corresponding to the vector $n\fe_i$, parametrising configurations of $n$ points, all lying on $(0,1)\times\set{i}$.
\begin{defn}\label{defn:shuffle}\label{defn:Sscoefficient}
 For $m,n\ge0$, a \emph{shuffle} of type $(m,n)$ is a permutation $\sigma\in\fS_{m+n}$ that is increasing on the sets $\{1,...,m\}$ and $\{m+1,...,m+n\}$. We denote by $\Shuf_{m,n}\subset\fS_{m+n}$ the set of all shuffles of type $(m,n)$.

For $m,n\ge 0$ the \textit{signed shuffle} coefficient is
\[
\Ss(m,n)=\left\{ 
\begin{array}{cl} 0 &\text{if both } m,n \text{ are odd,}\\\binom{\lfloor\frac{m}2\rfloor+\lfloor\frac{n}2\rfloor}{ \lfloor\frac{m}2\rfloor}=\binom{\lfloor\frac{m+n}2\rfloor}{\lfloor\frac m2\rfloor} & \text{otherwise.}
\end{array}
\right.
\]
\end{defn}
\begin{lem}
\label{lem:Sssumsigns}
 For $m,n\ge0$ we have that $\Ss(m,n)=\sum_{\sigma\in\Shuf_{m,n}}\sgn(\sigma)$.
\end{lem}
\begin{proof}
This is a double induction on $m,n\ge 0$ and is left to the reader.
\end{proof}

\begin{defn}\label{defn:sortmap}
For $m,n\ge0$ we define a map of pairs 
\[
\begin{split}
\sort_{m,n}\colon (\Delta^m\times\Delta^n,\del(\Delta^m\times\Delta^n))&\to(\Delta^{m+n},\del\Delta^{m+n})\\
((s_1,\dots,s_m),(t_{1},\dots,t_{n}))&\mapsto (p_1,\dots,p_{m+n}),
\end{split}
\]
where $(p_1,\dots,p_{m+n})$ is the weakly increasing sorting of the sequence of real numbers $s_1,\dots,s_{m},t_1,\dots,t_{n}$, repeated with multiplicity. 
\end{defn}
\begin{lem}
\label{lem:degreesort}
 For $m,n\ge0$, the degree of $\sort_{m,n}$ is equal to $\Ss(m,n)$.
\end{lem}
\begin{proof} This is a combinatorial exercise left to the reader.
\end{proof}

\begin{prop}\label{prop:UMoralgebra}
The product $\mu$ from Definition \ref{defn:muproduct} endows $\UMor_\bullet(k)$ with a ring structure.
By setting $x_i:=e(i,1)$ and $y_i:=e(i,2)$, we have the isomorphism of weighted rings
$$\UMor_\bullet(k)\cong \Lambda_\Z(x_1,\dots,x_k)\otimes \Gamma_{\Z}(y_1,\dots,y_k)$$
where $\Lambda_\Z$ and $\Gamma_{\Z}$ denote, respectively, the free exterior algebra and the free divided power algebra, where the elements $x_i$ have weight $\bullet=-1$ and the elements $y_i$ have weight $\bullet=-2$.
In particular, a homomorphism of rings with source $\UMor_\bullet(k)$ is characterised by its values on the $x_i$ and $y_i$, for $i=1,\dots,k$, if the target ring is torsion-free as an abelian group. 
\end{prop}
\begin{proof}
We check that for $m,n\ge0$ the map $\mu_{m,n}$ from Definition \ref{defn:muproduct} is cellular. The cells of $\confred[n+m]{V_k}$ are in dimension $0$ and $(m+n)$, and $\confred[m]{V_k}\times\confred[n]{V_k}$ has product cells of dimension $0,m,n$ and $(m+n)$. The union of all $0$-cells, $m$-cells and $n$-cells in the product  $\confred[m]{V_k}\times\confred[n]{V_k}$ consists of all pairs $(s_1,s_2)$ in which at least one between $s_1$ and $s_2$ is $\infty$; for such pairs we have $\mu_{m,n}(s_1,s_2)=\infty$, which is a $0$-cell. So $\mu_{m,n}$ preserves skeleta, i.e. it is cellular. Thus $\mu$ is cellular and induces a map $\mu_*$ that makes $\redchains{\conf{V_k}}$ into a bigraded ring, and similarly, making $\UMor_\bullet(k)$ into a weighted ring (concentrated in non-positive weights).

    By the definition of the map $\Phi_{\fv}$, all points in the image of $e_{\fv}\times e_{\fv'}$ along $\mu_{m,n}$ are contained in the union $\set{\infty}\cup e_{\fv+\fv'}$, thus $e_{\fv}\cdot e_{\fv'}=\lambda(\fv,\fv')e_{\fv+\fv'}$ for some $\lambda(\fv,\fv')\in \Z$. Furthermore, if $\fv$ and $\fv'$ have disjoint supports, then up to a permutation of the $k$ coordinates in $(\Z_{\ge0})^{\times k}$ the polysimplex $\Delta^{\fv+\fv'}$ can be identified with $\Delta^{\fv}\times \Delta^{\fv'}$, and $\Phi^{\fv+\fv'}$ coincides with $\mu\circ(\Phi^{\fv}\times \Phi^{\fv'})$. Thus $e_{\fv}\cdot e_{\fv'}=\pm e_{\fv+\fv'}$; the sign depends on whether or not the identification $\Delta^{\fv+\fv'}\cong\Delta^{\fv}\times\Delta^{\fv'}$ is orientation-preserving. We deduce that $\UMor_\bullet(k)$ is generated as a ring by the elements $e(i,n)$ for $i=1,\dots,k$ and $n\ge 1$.
    
    A particular case of the previous discussion is when $\fv=m\fe_i$ and $\fv'=n\fe_j$ for some $1\le i,j\le k$ with $i\neq j$: then the sign $\epsilon\in \{\pm 1\}$ such that $e(i,m)\cdot e(j,n)=\epsilon \hspace {2pt} e(j,n)\cdot e(i,m)$ is equal to the degree of the map $\Delta^m\times \Delta ^n\to \Delta^n\times \Delta ^m$ swapping the two coordinates, i.e. $\epsilon=(-1)^{mn}$.
    
    We next study the value of the coefficient $\lambda(m\fe_i,n\fe_i)$ for a fixed $i\in \{1,\dots,k\}$, and show that $\lambda(m\fe_i,n\fe_i)$ is equal to $\Ss(m,n)$.
    For $l\ge0$, write $\Phi^l$ for the map
    \[
    \Phi^{l\fe_i}/\del\Delta^l\colon \Delta^l/\del\Delta^l\overset{\cong}{\to}\confred[l]{S^1\times\set{i}}\subset \confred[l]{V_k},
    \]
    induced by $\Phi^{l\fe_i}$ on the quotient: this map sends the interior of $\Delta^l$ homeomorphically onto $C_l((0,1)\times\set{i})\subset\confred[l]{V_k}$.

The following diagram of maps of pairs is commutative
\[
 \begin{tikzcd}[column sep=40pt]
  (\Delta^m\times \Delta^n, \partial(\Delta^m\times \Delta^n))\ar[r,"\Phi^{m}\times \Phi^{n}"]\ar[d,"\sort_{m,n}"]& (\confred[m]{V_k},\infty)\times(\confred[n]{V_k},\infty)\ar[d,"\mu_{m,n}"]\\
  (\Delta^{m+n}, \partial\Delta^{m+n})\ar[r,"\Phi^{n+m}"]&(\confred[m+n]{V_k},\infty);
 \end{tikzcd}
\]
since the bottom map is the characteristic map of a cell, we have that $\lambda(m\fe_i,n\fe_i)$ is equal to the degree of $\sort_{m,n}$, which is $\Ss(m,n)$ by Lemma \ref{lem:degreesort}.

We conclude that $\UMor_\bullet(k)$ is generated as a ring by the elements $e(i,n)$ for $i=1,\dots,k$ and $n\ge 1$ (the elements $e(i,0)$ are equal to each other and give the unit of the ring); since $\Ss(2m,1)=1$ for all $m\ge0$, we can in fact factor $e(i,2m+1)=e(i,2m)e(2i,1)$, so that the elements $e(i,1)$ and $e(i,2m)$ also generate $\UMor_\bullet(k)$.

The above discussion shows that the elements $e(i,1)$ anticommute, the elements $e(i,2m)$ are central, and $e(i,2m)\cdot e(i,2n)=\binom{m+n}{m}e(i,2(m+n))$. Therefore $\UMor_\bullet(k)$ receives a natural homomorphism of rings with source the tensor product of $\Z$-algebras
\[
\Lambda_\Z(x_1,\dots,x_k)\otimes \Gamma_{\Z}(y_1,\dots,y_k),
\]
mapping $x_i\mapsto e(i,1)$ and $y_i\mapsto e(i,2)$, and the above discussion shows that this is in fact a bijection of rings.
\end{proof}
\subsection{Contents}
Let $\G_k:=\pi_1(V_k,*)=\langle \gamma_1,\dots,\gamma_k\rangle\cong \Z^{*k}$ as in Definition \ref{defn:bouquets}. Write $H_k:=(G_k)^{\ab}\cong \Z^{k}$ and $[-]:\G_k\to H_k$ for the abelianisation map. Observe that in the exterior algebra $\Lambda H_k$ the element  $1+[\gamma_i]$ has multiplicative inverse $1-[\gamma_i]$.

\begin{defn}
\label{defn:content}
    The \textit{content} is the ring homomorphism $c:\Z[\G_k]\longrightarrow \Lambda H_k$ defined on the ring generators as $\gamma_i\mapsto 1+[\gamma_i]$ and $\gamma_i^{-1}\mapsto 1-[\gamma_i]$. Write $c_i$ for the component of $c$ in the $i$\sth exterior power $\Lambda^iH_k$; note that $c_i\colon\Z[\G_k]\to\Lambda^iH_k$ is a homomorphism of abelian groups. (See \cite{Stavrou} for a similar construction.)
\end{defn}

\begin{lem}\label{lem:uniquenessofcontent}
The unique function of sets $C:\G_k\to \Lambda^2H_k$ satisfying that
\begin{enumerate}
    \item $C$ vanishes on the generators $\gamma_1,\dots,\gamma_k$, and
    \item $C(w_1w_2)=C(w_1)+C(w_2)+\left[w_1\right]\wedge \left[w_2\right]$ for all $w_1,w_2\in G_k$,
\end{enumerate}
is the restriction of $c_2$ to $G_k\subset\Z[G_k]$.
\end{lem}
\begin{proof}
    The extension of $C$ to the function
    \[
    \bar{C}:=1+[-]+C:\G_k\to \Lambda^*H_k/\Lambda^{\ge 3}H_k
    \]
    is multiplicative by property (2). As a result, $\bar{C}$ is defined by its values on the generators of $\G_k$, on which, by property (1), it takes the value $\bar{c}(\gamma_i)=1+[\gamma_i]$. Therefore $\bar{C}$ agrees with $c$ modulo $\Lambda^{\ge 3}H_k$ and therefore $C=c_2$.
\end{proof}

\subsection{Induced maps on \texorpdfstring{$\UMor_\bullet$}{Umor*}}
A based map $f:(V_k,*)\to (V_l,*)$ induces a map $\UMor_\bullet(f):\UMor_\bullet(k)\to \UMor_\bullet(l)$ which depends on $f$ only up to based homotopy. So, by abuse of notation, for a homomorphism $\phi:G_k\to G_l$, we write $\UMor_\bullet(\phi):=\UMor_\bullet(f)$, where $f\colon V_k\to V_l$ is any based map inducing $\phi$ on $\pi_1$. Furthermore, $\UMor_*(-)$ is \textit{functorial}, that is for composable homomorphisms $\phi:G_k\to G_l$ and $\psi:G_l\to G_m$, we have
\[
\UMor_\bullet(\psi\circ \phi)=\UMor_*(\psi)\circ \UMor_*(\phi).
\]
Note that $\UMor_\bullet(\phi)$ is a \emph{ring homomorphism}, as the superposition product $\mu$, making $\confred{V_k}$ into an abelian topological monoid, is natural with respect to pointed maps $f\colon V_k\to V_l$.

\begin{nota}
We denote by $(-)_x\colon H_k\to\bigoplus_{i=1}^k\Z x_i$ the isomorphism of abelian groups with $[\gamma_i]\mapsto x_i$, and use the same notation for the map induced between exterior powers of $H_k$ and $\bigoplus_{i=1}^k\Z x_i$. Similarly $(-)_y\colon H_k\to\bigoplus_{i=1}^k\Z y_i$ is the isomorphism such that $[\gamma_i]\mapsto y_i$.
\end{nota}

\begin{thm}\label{thm:UMormaps}
For a group homomorphism $\phi\colon G_k\to G_l$,
the induced ring homomorphism $\UMor_\bullet(\phi):\UMor_\bullet(k)\to \UMor_\bullet(l)$ is given on the ring generators from Proposition \ref{prop:UMoralgebra} by
\begin{align}
    x_i&\mapsto \left[\phi(\gamma_i)\right]_x \label{eq:assertedactionofUMor1}\\
y_i&\mapsto \left[\phi(\gamma_i)\right]_y+\left[c_2(\phi(\gamma_i))\right]_x \label{eq:assertedactionofUMor2}.
\end{align}
\end{thm}
\begin{proof}
We observe that there is a natural inclusion $V_k\cong \confred[1]{V_k}\hookrightarrow\confred{V_k}$, inducing on first homology the map $(-)_x\colon H_k\to\UMor_\bullet(k)$ with $[\gamma_i]\mapsto x_i$. This justifies \eqref{eq:assertedactionofUMor1}.
We prove \eqref{eq:assertedactionofUMor2} in four steps, dealing with different types of maps $\phi$.
\vspace{.5em}
\paragraph{Step 1: Simple maps}
Given a function $g:\{1,\dots,k\}\to \{0\}\sqcup \{1,\dots,l\}$, consider the group homomorphism $\phi_g:G_k\to G_l$ with $\gamma_i\mapsto\gamma_{g(i)}$ if $g(i)\ge1$ and $\gamma_i\mapsto 1$ if $g(i)=0$. We can realise $\phi_g$ by the based map of bouquets $f_g:V_k\to V_l$ with $*\mapsto *$ and
\[
f_g:(t,i)\in (0,1)\times \{i\}\subset V_k\mapsto \left\{
\begin{array}{cl}
(t,g(i)) & \text{if } g(i)\ge1, \\
* &  \text{otherwise.}
\end{array}
\right.
\]
We observe that $\confred[n]{f_g}$ maps the cell $e(i,n)$ of $\confred[n]{V_k}$ to the cell $e(g(i),n)$ of $\confred[n]{V_l}$ via an oriented homeomorphism, and collapses it to the $0$-cell $\infty$ otherwise. In particular, $\UMor_\bullet(f_g)$ sends
$y_i\mapsto y_{g(i)}$ if $g(i)\ge1$, and it sends
$y_i\mapsto 0$ otherwise.
This agrees with assertion \eqref{eq:assertedactionofUMor2}, since $c_2(\gamma_i)=0$.
\vspace{.5em}
\paragraph{Step 2: Pinch map} The pinch map $P:V_1\to V_2$, given by
\[
P (t,1)= \left\{
\begin{array}{cl} (2t,1) & \text{if } 0\le t\le \frac{1}{2}\\
(2t-1,2) & \text{if } \frac{1}{2} \le t \le 1,
\end{array}\right.
\]
realises the group homomorphism $p:G_1\to G_2$ with $\gamma_1\mapsto \gamma_1\gamma_2$. According to \eqref{eq:assertedactionofUMor2}, we must show that $\UMor_\bullet(P) $ sends 
$y_1\mapsto y_1+y_2+x_1\wedge x_2$. We compute $\UMor_\bullet(P)(y_1)$. The composition of the characteristic map $\Phi^{(2)}$ relative to the cell $e(1,2)\subset\confred[2]{V_1}$ with the map $\confred[2]{P}$ is given explicitly as the map $\confred[2]{P}\circ \Phi^{(2)}:\Delta^2\to\confred[2]{V_2}$ with
\[
(t_1,t_2)\mapsto\left\{\begin{array}{cl}\left[(2t_1,1), (2t_2,1)\right]&\text{if } t_1,t_2\le \frac{1}{2},\\[.1cm]
\left[(2t_1,1),(2t_2-1,2)\right]& \text{if } t_1\le \frac{1}{2}\le t_2,\\[.1cm]
\left[(2t_1-1,2), (2t_2-1,2)\right]& \text{if }\frac{1}{2}\le t_1,t_2.\end{array}\right.
\]
These three cases partition the domain $\Delta^2$ into three regions $A^{(2,0)}$, $A^{(1,1)}$, $A^{(0,2)}$ which are parametrised by $\Delta^2$, $\Delta^1\times \Delta^1$, $\Delta^2$, respectively, via the oriented affine transformations $a_{(2,0)}: (t_1,t_2)\mapsto (2t_1,2t_2)$, $a_{(1,1)}:(t_1,t_2)\mapsto (2t_1,2t_2-1)$, $a_{(0,2)}:(t_1,t_2)\mapsto (2t_1,2t_2-1)$, as in Figure \ref{fig:splitofcell}. For each $\fv\in\{(2,0),(1,1),(0,2)\}$, we have an equality of maps
\[
\Phi^{\fv}\circ a_{\fv}^{-1}=\confred[2]{P}\circ \Phi^{(2)}: A^{\fv}\to \confred[2]{V_2}.
\]
We conclude that $\UMor_\bullet(P):y_1=e_{(2)}\mapsto e_{(2,0)}+e_{(0,2)}+e_{(1,1)}=y_1+y_2+x_1\wedge x_2$.
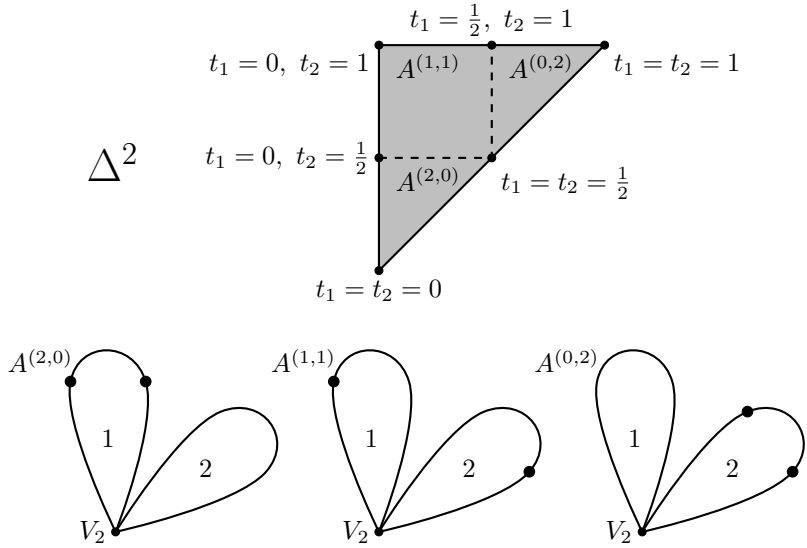
\begin{figure}[ht]
    \centering
    \begin{tikzpicture} 
        \begin{huge}
        \draw (-3,1) node[anchor=south east]{$\Delta^2$};
        \end{huge}
        \filldraw[thick, color=black, fill=lightgray] (0,0) -- (3,3) -- (0,3) -- cycle;
        \draw[dashed, thick] (0,1.5) -- (1.5,1.5) -- (1.5,3);
        \filldraw[black] (3,3) circle (1.5pt) node[anchor=north west]{$t_1=t_2=1$};
        \filldraw[black] (1.5,1.5) circle (1.5pt) node[anchor=north west]{$t_1=t_2=\frac12$};
        \filldraw[black] (0,0) circle (1.5pt) node[anchor=north]{$t_1=t_2=0$};
        \filldraw[black] (0,3) circle (1.5pt) node[anchor=north east]{$t_1=0,\ t_2=1$};
        \filldraw[black] (1.5,3) circle (1.5pt) node[anchor=south]{$t_1=\frac12,\ t_2=1$};
        \filldraw[black] (0,1.5) circle (1.5pt) node[anchor=east]{$t_1=0,\ t_2=\frac 12$};
        \draw (1.6,3) node[anchor=north west]{$A^{(0,2)}$};
        \draw (.1,1.5) node[anchor=north west]{$A^{(2,0)}$};
        \draw (.1,3) node[anchor=north west]{$A^{(1,1)}$};
    \end{tikzpicture}

    \vspace{10pt}
    \begin{tikzpicture}
        
    \begin{scope}[shift={(-3.5,0)}]
            \draw (-1,2) node[anchor=south]{$A^{(2,0)}$};
           \draw[thick] plot [smooth, tension=1] coordinates {(0,0) (-0.6,2) (0.4,2) (0,0)};
            \draw (-0.1,1) node[anchor=south]{1};
            
            \draw[thick] plot [smooth, tension=1] coordinates {(0,0) (1.4,1.6) (2,0.8) (0,0)};
            \draw (1.2,0.6) node[anchor=south]{2};
            
            \filldraw[black] (-0.6,2) circle (2pt);
            \filldraw[black] (0.4,2) circle (2pt);
            \filldraw[black] (0,0) circle (1.5pt) node[anchor=east]{$V_2$};
    \end{scope}
    \begin{scope}[shift={(0,0)}]
            \draw (-1,2) node[anchor=south]{$A^{(1,1)}$};
            \draw[thick] plot [smooth, tension=1] coordinates {(0,0) (-0.6,2) (0.4,2) (0,0)};
            \draw (-0.1,1) node[anchor=south]{1};
            
            \draw[thick] plot [smooth, tension=1] coordinates {(0,0) (1.4,1.6) (2,0.8) (0,0)};
            \draw (1.2,0.6) node[anchor=south]{2};
            
            \filldraw[black] (-0.6,2) circle (2pt);
            \filldraw[black] (2,0.8) circle (2pt);
            \filldraw[black] (0,0) circle (1.5pt) node[anchor=east]{$V_2$};
    \end{scope}
    \begin{scope}[shift={(3.5,0)}]
            \draw (-1,2) node[anchor=south]{$A^{(0,2)}$};
           \draw[thick] plot [smooth, tension=1] coordinates {(0,0) (-0.6,2) (0.4,2) (0,0)};
            \draw (-0.1,1) node[anchor=south]{1};
            
            \draw[thick] plot [smooth, tension=1] coordinates {(0,0) (1.4,1.6) (2,0.8) (0,0)};
            \draw (1.2,0.6) node[anchor=south]{2};
            
            \filldraw[black] (1.4,1.6) circle (2pt);
            \filldraw[black] (2,0.8) circle (2pt);
            \filldraw[black] (0,0) circle (1.5pt) node[anchor=east]{$V_2$};
    \end{scope}
    \end{tikzpicture}
    \caption{The 2-simplex $\Delta^2$ (above) maps under $\confred[2]{P}\circ \Phi^{(2)}$ to configurations of $2$ points in $V_2$. Region $A^{(2,0)}$ corresponds (below) to configurations of two points on the first loop; region $A^{(1,1)}$ to configurations of one point on each loop; region $A^{(0,2)}$ to configurations of two points on the second loop.}
    \label{fig:splitofcell}
\end{figure}

\vspace{.5em}
\paragraph{Step 3: Maps from $G_1$} For $w\in G_l$, let $\psi_w:G_1\to G_l$ be the group homomorphism with $\gamma_1\mapsto w\in G_l$. Given $w,u\in G_l$, the homomorphism $\psi_{wu}$ factors through
\begin{equation}\label{eq:composition}
\begin{tikzcd}[column sep=40pt]
G_1\rar["P"] & G_2\cong G_1 *G_1 \rar["\psi_w * \psi_u"] & \rar G_l* G_l\rar["\phi_g"] & G_l
\end{tikzcd}
\end{equation}
where $g:\{1,\dots,l,1',\dots,l'\}\to \{0,1,\dots,l\}$ sends $i,i'\mapsto i$.

We can now apply $\UMor_\bullet$ to composition \eqref{eq:composition}, and evaluate at $y_1$ leveraging Steps 1 and 2, to obtain
\begin{equation}\label{eq:umory}
\UMor_\bullet(\psi_{wu})(y_1)=\UMor_\bullet(\psi_w)(y_1)+\UMor_\bullet(\psi_u)(y_1)+[w]_x\wedge [u]_x. 
\end{equation}
By decomposing $\UMor_\bullet(\psi_w)(y_2)\in\UMor_{-2}(l)$ as the sum of components $Y(w)\in \bigoplus_{i=1}^l\Z y_i$ and $C(w)\in \Lambda^2\pa{\bigoplus_{i=1}^l\Z x_i}$, equality \eqref{eq:umory} gives
\begin{equation}\label{eq:propertyofc}
\begin{split}
Y(wu)&=Y(w)+Y(u);\\
C(wu)&=C(w)+C(u)+[w]_x\wedge [u]_x.
\end{split}
\end{equation}
Thus $Y:G_l\to \bigoplus_{i=1}^l\Z y_i$ is a group homomorphism, and it coincides with $[-]_y$ on the generators $\gamma_1,\dots,\gamma_l$ by Step 1, and so $Y(-)=[-]_y$. Finally, $C$ vanishes on words of length $1$ by Step 1 and satisfies \eqref{eq:propertyofc}, so $C=c_2$ by Lemma \ref{lem:uniquenessofcontent}.
\vspace{.5em}
\paragraph{Step 4: Any map $\phi:G_k\to G_l$}
If $\phi\colon G_k\to G_l$ is any group homomorphism, by Step 1 we have that the generator $y_i\in \UMor_{-2}(k)$ is equal to the image of $y_1\in \UMor_{-2}(1)$ along the map $\phi_{g_i}\colon G_1\to G_k$, where $g_i: \{1\}\to \{0,...,k\}$ takes $1$ to $i$; we therefore compute
\[
\UMor_\bullet(\phi)(y_i)=\UMor_\bullet(\phi\circ\phi_{g_i})(y_1)
\]
which, by Step 3, agrees with assertion \eqref{eq:assertedactionofUMor2} .
\end{proof}

\section{Cellular chains for configurations of surfaces}
\label{sec:confsurf}
The setup of this section is analogous to \cite{Bianchi,BMW}.
\subsection{A model for the surface \texorpdfstring{$\Sigma_{g,1}$}{Sigmag1}}
Let $g\ge 0$ be a fixed integer throughout the section.
In the rectangle $\bfR=[0,2]\times [0,1]$, we decompose the side $\{2\}\times [0,1]$ into $4g$ consecutive closed intervals of equal length called $J_1,\dots,J_{4g}$ ordered and oriented with increasing second coordinate. Denote by $\tilde p_0\subset\{2\}\times [0,1]$ the set of their endpoints. 
\begin{defn}
\label{defn:cM}
We let $\cM$ be the quotient of $\bfR$ obtained by identifying $J_{4i+1}$ and $J_{4i+2}$ with $J_{4i+3}$ and $J_{4i+4}$, respectively, via their unique \emph{orientation-reversing} isometry, for all $0\le i\le g-1$. We denote by $\fp\colon \bfR\to\cM$ the quotient map. See Figure \ref{fig:Mopen}.
\end{defn}

The space $\cM$ is homeomorphic to $\Sigma_{g,1}$ and admits a cell decomposition consisting of the following open cells: 
\begin{itemize}
    \item one $0$-cell $p_0=\fp(\tilde p_0)$;
    \item $(2g+1)$-many $1$-cells attached to $p_0$:
    \begin{itemize}
        \item the images $I_{2i+j}=\fp(\mathring{J}_{4i+j})$, for $0\le i\le g-1$ and $j=1,2$;
        \item the image $I=\fp(\partial \bfR-\{2\}\times [0,1])$, coinciding with $\partial \cM\setminus\set{p_0}$; 
    \end{itemize}
    \item one $2$-cell, the image $\fp(\mathring{\bfR})$.
\end{itemize}

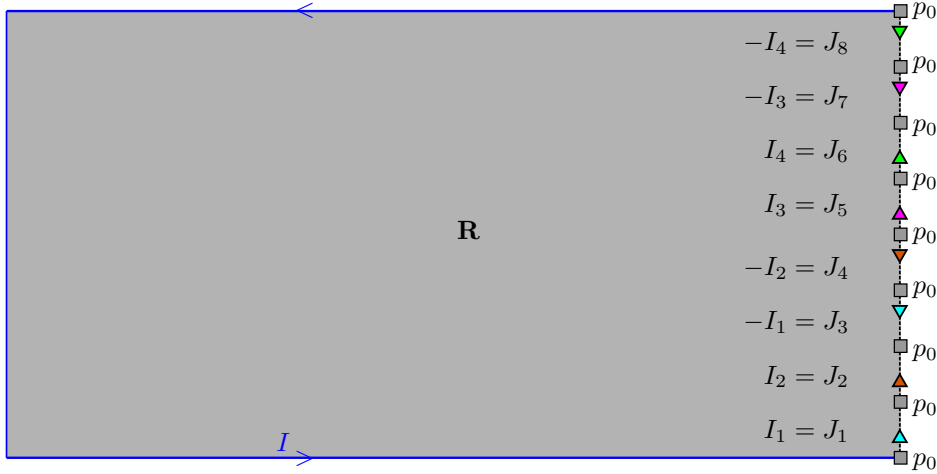
\begin{figure}[ht]
 \centering
 \begin{overpic}[width=12cm]{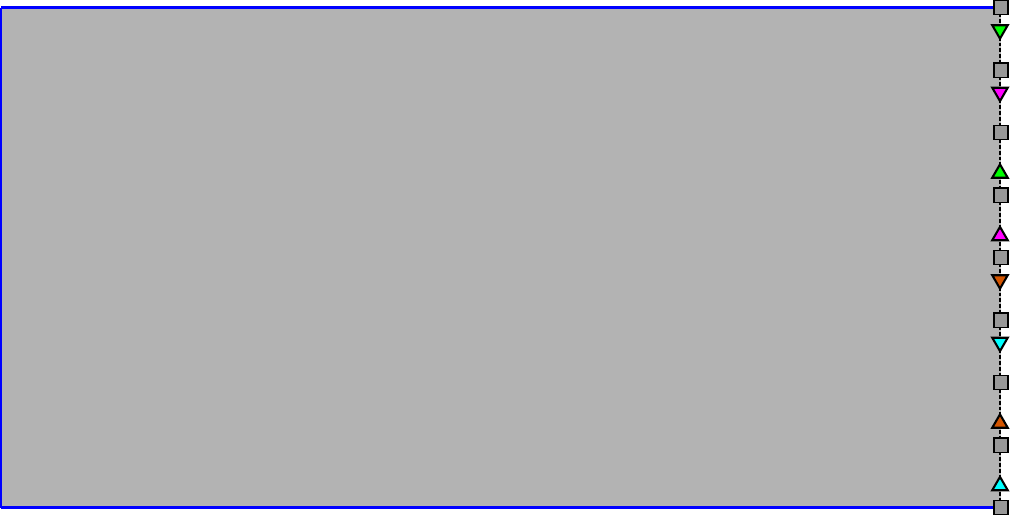}
 \put(50,25){$\bfR$} \put(84,3){$I_1=J_1$}  \put(84,9){$I_2=J_2$}  \put(81.75,15){$-I_1=J_3$} \put(81.75,21){$-I_2=J_4$} \put(84,28){$I_3=J_5$} \put(84,34){$I_4=J_6$} \put(81.75,40){$-I_3=J_7$} \put(81.75,46){$-I_4=J_8$}
 \put(100.5,0){$p_0$}\put(100.5,6){$p_0$}\put(100.5,12){$p_0$}\put(100.5,19){$p_0$}\put(100.5,25){$p_0$}\put(100.5,31){$p_0$}\put(100.5,37){$p_0$}\put(100.5,44){$p_0$}\put(100.5,50){$p_0$}
 \put(30,1.5){$\textcolor{blue}{I}$}
 \put(32,0){$\textcolor{blue}{>}$}
 \put(32,49.6){$\textcolor{blue}{<}$}
 \end{overpic}
 \caption{The cell decomposition of $\cM$ in the case $g=2$.}
 \label{fig:Mopen}
\end{figure}

\begin{defn}\label{defn:iotaembedding}
We denote by $\iota: (V_{2g},*)\hookrightarrow (\cM, p_0)$ the embedding induced by the map $\coprod_{i=1}^{2g} [0,1]\times\set{i}\to\bfR$ sending $[0,1]\times\set{2j+\epsilon}$ onto $J_{4j+\epsilon}$ via the unique linear and orientation-preserving homeomorphism, for $0\le j\le g-1$ and $\epsilon=1,2$.
\end{defn}
The map $\iota$ identifies the $i$\sth circle of $V_{2g}$ with the closure of $I_i$, for $1\le i\le 2g$. In particular, $\iota$ is a homeomorphism between $V_{2g}$ and $\fp(\set{2}\times[0,1])\subset\cM$.

\subsection{Cellular decomposition of \texorpdfstring{$\conf[n]{\cM}$}{confn(M)}}
Recall from Notation \ref{nota:confnM} the shorthand $\conf[n]{\cM}:=\conf[n]{\cM,\partial \cM}$. 

\begin{defn}\label{defn:record}
For $n\ge 0$, an $n$-\textit{record}, denoted generically $\ft$, is a choice of the following set of data:
\begin{itemize}
    \item an integer $b\ge 0$;
    \item a sequence $\uP=(P_1,\dots,P_b)$ of positive integers $P_i\in \Z_{\ge 1}$;
    \item a sequence $\fv=(v_1,\dots,v_{2g})$ of non-negative integers $v_i\in \Z_{\ge 0}$,
\end{itemize}
satisfying the equality $P_1+\dots+P_b+v_1+\dots+v_{2g}=n$. We generically write $\ft=(b,\uP,\fv)$, $n(\ft):=n$ and $b(\ft):=b$. The \emph{dimension} of $\ft$ is $d(\ft):=n+b$. When $g=0$, we usually write a record as $(b,\uP)$.
\end{defn}

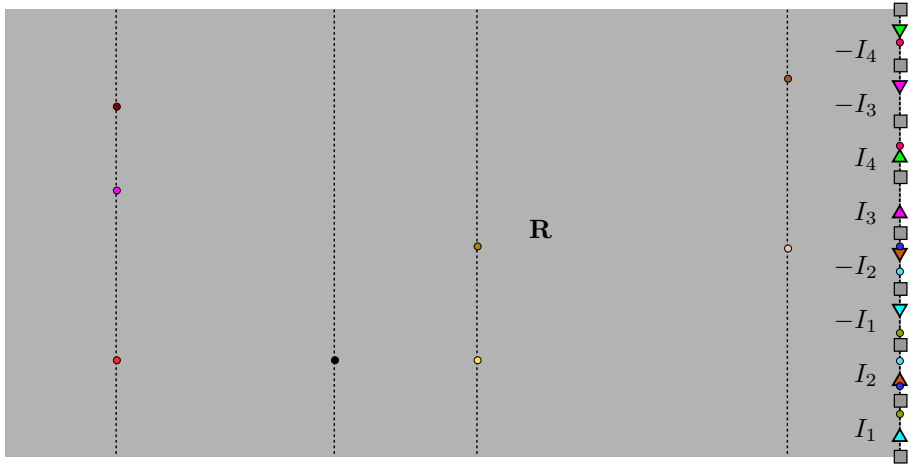
\begin{figure}[ht]
 \centering
 \begin{overpic}[width=12cm]{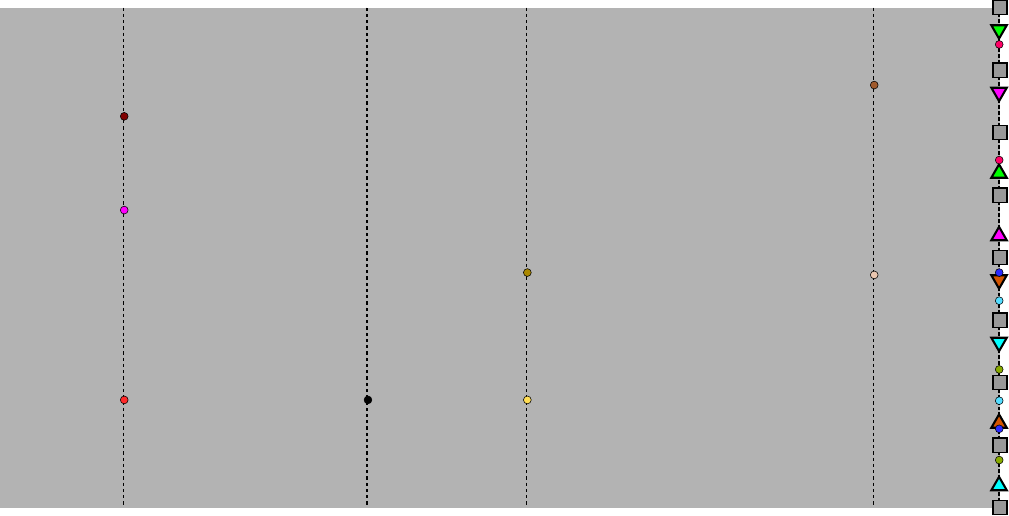}
 \put(58,25){$\bfR$} \put(94,3){$I_1$} \put(94,9){$I_2$} \put(91.75,15){$-I_1$} \put(91.75,21){$-I_2$} \put(94,27){$I_3$} \put(94,33){$I_4$} \put(91.75,39){$-I_3$} \put(91.75,45){$-I_4$}
 \end{overpic}
 \caption{A configuration lying in the open cell $e_\ft\subset C_{12}(\cM)$, where $\ft=(4,(3,1,2,2), (1,2,0,1))$. Observe that on the right side, each point is drawn twice.}
 \label{fig:cell}
\end{figure}

For each $n$-record $\ft$, let $e_{\ft}$ be the subspace of configurations $s\in C_n(\mathring{\cM})\subset \conf[n]{\cM}$ that satisfy the following conditions:
\begin{enumerate}
    \item for all $1\le i\le 2g$, exactly $v_i$ points of the configuration $s$ lie on $I_i$;
    \item there are exactly $b$ real numbers $0<x_1<\dots<x_b<2$ such that $s$ admits at least one point lying in $\mathring{\bfR}\cong\fp(\mathring{\bfR})\subset\mathring{\cM}$ having $x_i$ as first coordinate;
    \item for all $1\le i\le b$, exactly $P_i$ points of $s\cap \mathring{\bfR}$ have first coordinate equal to $x_i$. In other words, $s\cap \mathring{\bfR}$ consists of $b$ vertical ``bars'' with $P_i$ points lying on the $i$\sth bar from the left.
\end{enumerate}
Each $s\in C_n(\mathring{\cM})\subset\conf[n]{\cM}$ lies in a unique subspace $e_{\ft}$. For a given $n$-record $\ft$, the space $e_{\ft}$ is homeomorphic to an open disc of dimension $d(\ft)$. Specifically, consider the polysimplex
$$\Delta^{\ft}:=\Delta^{b}\times \prod_{i=1}^{b}\Delta^{P_i} \times \prod_{i=1}^{2g}\Delta^{v_i}$$
and the map $\Phi^{\ft}:\Delta^{\ft}\to \conf[d(\ft)]{\cM}$ sending the sequence of simplicial coordinates
\begin{equation*}
   \big((x_1,\dots,x_b), (s^{(1)}_1,\dots,s^{(1)}_{P_1}),\dots, (s^{(b)}_1,\dots,s^{(b)}_{P_b}),(t_1^{(1)},\dots, t_{v_1}^{(1)}),\dots, (t_{1}^{(2g)},\dots, t_{v_{2g}}^{(2g)})\big)
\end{equation*}
to the following configuration, where $\cdot$ denotes the superposition product:
\begin{equation*}
    \left[ \fp(2x_i,s^{(i)}_{j})\,|\,1\le i\le b, 1\le j\le P_i\right]\cdot\left[
    \iota((t_{j}^{(i)},i))\,|\,1\le i\le 2g, 1\le j\le v_i\right].
\end{equation*}
Each map $\Phi^{\ft}$ restricts to a homeomorphism $\mathring{\Delta}^{\ft}\to e_{\ft}$, and maps the boundary $\partial\Delta^{\ft}$ to the union of $\set{\infty}$ and the subspaces $e^{\ft'}$ for varying $\ft'$ among the $n$-records satisfying $d(\ft')<d(\ft)$. This leads to the following, which we state as a proposition for future reference.
\begin{prop}\label{prop:celldecompsurface}
    The space $\conf{\cM}$ admits a cell decomposition with a $0$-cell $\infty$, and  a $d(\ft)$-cell $e_{\ft}$, for each $n$-record $\ft$, $n\ge 0$. The orientations of the cells and the attaching maps are $\Phi^{\ft}$.
\end{prop}

\subsection{Cellular chain ring}
In this subsection we will put a bigraded differential ring structure on the reduced cellular chain complex $$\redchains{\conf{\cM}}\cong\bigoplus_{\ft}\Z e_{\ft}$$
(with the differential on the right hand side of the isomorphism determined below.) The most natural choice of a bigrading would put the generator $e_\ft$ in weight $n(\ft)$ and homological degree $d(\ft)$; we will use the following alternative, which will become convenient in the proof of Theorem \ref{thm:B} later.
\begin{nota}
 \label{nota:star}\label{nota:bulletstarofft}
We introduce a new grading with values in $\Z$, called the \emph{bar-degree} and denoted $\star$. It is the difference between the \emph{homological degree}, denoted $*$, and the \emph{weight}, denoted $\bullet$. A bigraded abelian group will usually be endowed with the two gradings $(\bullet,\star)$; the \emph{homological degree}, denoted $*$, can be recovered as $*:=\bullet+\star$.

For a record $\ft$, we put the generator $e_\ft\in\redchains{\conf{\cM}}$ in weight $\bullet=-n(\ft)$ and in homological degree $*=b(\ft)-n(\ft)$; as a result, the bar-degree $\star$ of $e_\ft$ will be
\[
*-\bullet=b(\ft)-n(\ft)-(-n(\ft))=b(\ft).
\]
\end{nota}
We justify this convention in a way similar to what we did in the discussion after Definition \ref{defn:UMor}: the generator $e_\ft$ will contribute to the chain group $\redchains[d(\ft)]{\conf[n(\ft)]{\cM}}$, and hence to the homology group $\widetilde{H}_{d(\ft)}(\conf[n(\ft)]{\cM})$, which is isomorphic to the \emph{cohomology group} $H^{2n(\ft)-d(\ft)}(C_{n(\ft)}(\cM))$; regarding cohomology as the bigraded dual of homology, and willing to consider $H_i(C_n(\cM))$ as living in homological degree $*=i$ and in weight $\bullet=n$, we are thus led to put $e_{\ft}$ in weight $-n(\ft)$ and in homological degree $-(2n(\ft)-d(\ft))=b(\ft)-n(\ft)$.
Note that the differential $\del$ on $\redchains{\conf{\cM}}$ has bidegree $(\bullet,\star)=(0,-1)$.

\begin{prop}\label{prop:muiscellular}
The superposition product $\mu\colon\conf{\cM}\times\conf{\cM}\to\conf{\cM}$ is a cellular map, and thus it makes $\redchains{\conf{\cM}}$ into a commutative differential bigraded ring. Graded commutativity and the Leibniz rule hold with respect to the homological degree $*=\bullet+\star$.     
In particular, for records $\ft,\ft'$, we have $e_{\ft}\cdot e_{\ft'}=(-1)^{d(\ft)d(\ft')}e_{\ft'}\cdot e_{\ft}$. Furthermore, for every record $(b,\uP,\fv)$ we have a factorisation $e_{(b,\uP,\fv)}=e_{(b,\uP,\uzero)}\cdot e_{(0,0,\fv)}$.
\end{prop}
\begin{proof}
For a configuration $s\in \conf{\cM}\setminus\set{\infty}$, denote by $n(s)$, $d(s)$ and $b(s)$ the values $n(\ft)$, $d(\ft)$ and $b(\ft)$, respectively, where $e_{\ft}$ is the unique cell containing $s$ in its interior, and we use the notation from Definition \ref{defn:record}; we also set $d(\infty)=0$, as $\infty$ is a 0-cell in $\conf{\cM}$. To show that $\mu$ is cellular with the product cellular structure on $\conf{\cM}\times\conf{\cM}$ it suffices that 
\[
d(\mu(s,s'))\le d(s)+d(s')\ \mbox{ for all }\ s,s'\in\conf{\cM}.
\]
If $\mu(s,s')=\infty$, this is obvious, as we have $d(\mu(s,s'))=0$. Otherwise both $s,s'\neq\infty$, we have the inequality $b(\mu(s, s'))\le b(s)+b(s')$ and the equality $n(\mu(s,s'))=n(s)+n(s')$, and by using $d(s)=b(s)+n(s)$ we are done.

The graded Leibniz rule for the differential $\del$ on $\redchains{\conf{\cM}}$ holds with respect to $d(\ft)$, and similarly the commutativity of $\mu$ implies that the induced product on $\redchains{\conf{\cM}}$ satisfies $e_{\ft}\cdot e_{\ft'}=(-1)^{d(\ft)d(\ft')}e_{\ft'}\cdot e_{\ft}$. But the homological degree $*=b(\ft)-n(\ft)$ from Notation \ref{nota:bulletstarofft} and $d(\ft)=b(\ft)+n(\ft)$ have the same parity. Therefore the Leibniz rule and graded commutativity also hold with respect to $*$.
    
Finally, for a given record $(b,\uP,\fv)$ we have a canonical identification
\[
\Delta^{(b,\uP,\fv)}=\Delta^{(b,\uP,\uzero)}\times \Delta^{(0,\uzero,\fv)},
\]
and the characteristic map $\Phi^{(b,\uP,\fv)}\colon\Delta^{(b,\uP,\fv)}\to\conf{\cM}$ coincides with the composition
    \[
     \begin{tikzcd}[column sep=60pt]
      \Delta^{(b,\uP,\uzero)}\times \Delta^{(0,\uzero,\fv)}\ar[r,"\Phi^{(b,\uP,\uzero)}\times\Phi^{(0,\uzero,\fv)}"]&\conf{\cM}\times\conf{\cM}\ar[r,"\mu"]&\conf{\cM}
     \end{tikzcd}
    \]
proving the equality $e_{(b,\uP,\fv)}=e_{(b,\uP,\uzero)}\cdot e_{(0,0,\fv)}$.
\end{proof}
\begin{defn}
 Let $n\ge0$, $\sigma\in\fS_n$, $\uP=(P_1,\dots,P_n)\in\Z^n$ and $N=\sum (1+P_iP_j)$, where the sum is extended over all pairs of indices $1\le i<j\le n$ satisfying $\sigma(i)>\sigma(j)$. The \emph{$\uP$-twisted sign} of $\sigma$, denoted $\sgn_{\uP}(\sigma)$, is $(-1)^N\in\set{\pm1}$. (Note that if $2$ divides all $P_i$, then $\sgn_{\uP}(\sigma)=\sgn(\sigma)$.)
\end{defn}
The following is immediate.
\begin{lem}
\label{lem:productChB}
Let $(b,\uP,\uzero)$ and $(b,\uP',\uzero)$ be records with vanishing third vector. Write $\uP=(P_1,\dots,P_b)$ and $\uP'=(P_{b+1},\dots,P_{b+b'})$; denote by
\[
\uP*\uP'=(P_1,\dots,P_{b+b'})
\]
the concatenation of $\uP$ and $\uP'$, and for a shuffle $\sigma$ of type $(b,b')$ (see Definition \ref{defn:shuffle}) let $\uP*_{\sigma}\uP'=(P_{\sigma(1)},\dots,P_{\sigma(b+b')})$.
Then we have
\[
 e_{(b,\uP,\uzero)}\cdot e_{(b',\uP',\uzero)}=(-1)^{b'(P_1+\dots+P_b)}\sum_{\sigma} \sgn_{\uP*\uP'}(\sigma) e_{(b+b',\uP*_\sigma\uP',\uzero)}
 \]
where the sum is taken over all shuffles $\sigma$ of type $(b,b')$.
\end{lem}

Recall now the embedding $\iota$ from Definition \ref{defn:iotaembedding}. The induced map
\[
\conf{\iota}:\confred{V_{2g}}\to \conf{\cM}
\]
is cellular, and in fact it exhibits $\confred{V_{2g}}$ as a subcomplex of $\conf{\cM}$, spanned by $\infty$ and all cells $e_{0,\uzero,\fv}$. Moreover, since superposition of configurations is natural in the space, we have that $\conf{\iota}$ is a map of commutative monoids. Passing to reduced cellular chain complexes, the map $\redchains{\conf{\iota}}:\UMor_\bullet(2g)\to \bigoplus_{\ft}\Z e_{\ft}$ is the inclusion of differential algebras given by $e_{\fv}\mapsto e_{(0,\uzero,\fv)}$. Notice that the inclusion $\redchains{\conf{\iota}}$ respects the weight $\bullet$ (which by our conventions attains non-positive values on both sides), and if we put $\UMor_\bullet(2g)$ in bar-degree $\star=0$, then 
$\redchains{\conf{\iota}}$ also respects the bar-degree. Combining this and Proposition \ref{prop:thecellularstructureforbouquets}, we obtain the following proposition.
\begin{prop}
\label{prop:decompositionintobarandUMor}
There is a tensor product decomposition of bigraded rings
\[
\redchains{\conf{\cM}}\cong\widetilde{\Ch}^{B}_*(\cM)\otimes \UMor_\bullet(2g)
\]
where $\widetilde{\Ch}^{B}_*(\cM)$ is generated as a free abelian group by the cells of the form $e_{(b,\uP,\uzero)}$.
\end{prop}
Both factors are subrings of $\redchains{\conf{\cM}}$: the first by Lemma \ref{lem:productChB}, the second since it is the image of the ring homomorphism induced by $\conf{\iota}$. Only $\UMor_\bullet(2g)$, on which the differential is trivial, is a differential subring.

\begin{rem}
We used the superscript $B$ for the subring $\widetilde{\Ch}^{B}_*(\cM)$ because we think of it as bar resolution for the homology of configurations on a disc. In fact, the factor $\widetilde{\Ch}^{B}_*(\cM)$ is responsible for the bar-degree of the elements in 
$\redchains{\conf{\cM}}$.
\end{rem}

\subsection{The differential on \texorpdfstring{$\redchains{\conf{\cM}}$}{tildeCh(Cbullet(M)infty}}
Since $\redchains{\conf{\cM}}$ is a differential bigraded ring, its differential $\del$ is uniquely determined via the graded Leibniz rule by its behaviour on a set of ring generators, such as the elements $e_{(b,\uP,\uzero)}$ and $e_{(0,\uzero,\fv)}$ for varying $(b,\uP)$ and $\fv$; this uses Proposition \ref{prop:muiscellular}. On generators of the second kind the differential $\del$ vanishes, so we will focus on computing $\del(e_{(b,\uP,\uzero)})$.
\begin{defn}\label{defn:omega}
Let $\omega$ denote the element
\[
\omega:=x_1\wedge x_2+x_3\wedge x_4+\dots +x_{2g-1}\wedge x_{2g}
\in\Lambda^2_\Z(x_1,\dots,x_{2g})\subset \UMor_{-2}(2g),
\]
where we use the notation from Proposition \ref{prop:UMoralgebra}.
For $k\ge 1$ define $\Omega_{2k-1}:=0$ and define $\Omega_{2k}\in \Lambda^{2k}_\Z(x_1,\dots,x_{2g})\subset\UMor_{-2k}(2g)$ as the element
\[
\Omega_{2k}:=\frac{(2\omega)^k}{k!}=2^k\sum_{0\le i_1<\dots<i_k\le g-1}x_{2i_1+1}\wedge x_{2i_1+2}\wedge\dots\wedge x_{2i_k+1}\wedge x_{2i_k+2}.
\]
\end{defn}
\begin{defn}
 \label{defn:zetag}
We denote by $\zeta_g\in\pi_1(V_{2g},*)\cong\Z^{*2g}$ the element
$$\zeta_g= \gamma_1\gamma_2\gamma_1^{-1}\gamma_2^{-1}\dots\gamma_{2g-1}\gamma_{2g}\gamma_{2g-1}^{-1}\gamma_{2g}^{-1}.$$ 
\end{defn}

\begin{lem}\label{lem:contentomega}
 Recall Definition \ref{defn:content}. Then $c_2(\zeta_g)=2\omega\in \Lambda^2 H_{2g}$.
\end{lem}
\begin{proof}
 By property (2) from Lemma \ref{lem:uniquenessofcontent}, and using that
 \[
 [\gamma_{2i}\gamma_{2i+1}\gamma_{2i}^{-1}\gamma_{2i+1}^{-1}]=0\in H_{2g}\ \mbox{ for all }\ 0\le i\le g-1,
 \]
 we have $c_2(\zeta_g)=\sum_{i=0}^{g-1}c_2(\gamma_{2i}\gamma_{2i+1}\gamma_{2i}^{-1}\gamma_{2i+1}^{-1})$.
 Since $c_2$ vanishes on each generator $\gamma_j$ and $\gamma_j^{-1}$, applying property (2) we compute $c_2(\gamma_{2i}\gamma_{2i+1}\gamma_{2i}^{-1}\gamma_{2i+1}^{-1})=2[\gamma_{2i}]\wedge[\gamma_{2i+1}]$.
 \end{proof}

\begin{prop}\label{prop:differential}
The differential $\del$ restricted to $\widetilde{\Ch}_*^{B}(\cM)$ is the sum of two components 
$\del_B+\del_M$
given on a generator $e_{(b,\uP,\uzero)}$ by
 \[
  \begin{split}
\del_B(e_{(b,\uP,\uzero)})&=\sum_{i=1}^{b-1} (-1)^{i} \Ss(P_i,P_{i+1})  \  e_{(b-1, (P_1,\dots,P_{i-1},P_i+P_{i+1},P_{i+2},\dots, P_b),\uzero)};\\
\del_M(e_{(b,\uP,\uzero)})&=(-1)^b e_{(b-1,(P_1,\dots,P_{b-1})),\uzero)}\cdot \Omega_{P_b}.
  \end{split}
 \]
\end{prop}
\begin{proof}
Fix a record $\ft=(b,\uP,\uzero)$, and endow the polysimplex
\[
\Delta^\ft=\Delta^b\times \Delta^{P_1}\times\dots\times \Delta^{P_b}
\]
with the product cell structure of the simplicial cell structures on the factors.

We claim that the map $\Phi^\ft\colon\Delta^\ft\to\conf{\cM}$ is cellular, i.e. each open $k$-cell of $\Delta^\ft$ is sent along $\Phi^\ft$ inside the $k$-skeleton of $\conf{\cM}$; we will prove this claim by induction on $b\ge0$. For the top cell this is tautological, as $\Phi^\ft$ is a characteristic map of a cell. We write $\del\Delta^\ft$ as a union of subcomplexes
\[
\begin{array}{cl}
\del_j^{\hor}\Delta^\ft:=\del_i\Delta^b\times\Delta^{P_1}\times\dots\times \Delta^{P_b},&\mbox{for }0\le j\le b,\\ \del_j^{\ver,i}\Delta^\ft:=\Delta^b\times\Delta^{P_1}\times\dots\times\del_j\Delta^{P_i}\times\dots\times \Delta^{P_b},&\mbox{for }1\le i\le b\mbox{ and }0\le j\le P_b,
\end{array}
\]
and check that $\Phi^\ft$ is cellular on each of these subcomplexes. 

The map $\Phi^\ft$ restricts to the constant map to $\infty$ on each subcomplex $\del_j^{\ver,i}\Delta^\ft$. Geometrically, this corresponds to the situation in which
either two points of the $i$\sth bar collide (for $1\le j\le P_i-1$), or one point of the $i$\sth bar goes to $\partial M$ (for $j=0,P_i$). For $b=0$, this concludes the proof that $\Phi^\ft$ is cellular.
    
We now assume $b\ge1$ and analyse the restriction of $\Phi^\ft$ on the subcomplexes $\del_j^\hor\Delta^\ft$.
For $0\le j\le b$ let $\iota_j\colon\Delta^{b-1}\to\Delta^b$ be the standard inclusion of a face,
and denote by $\hat\iota_j\colon\Delta^{b-1}\times\prod_{i=1}^b\Delta^{P_i}\to\Delta^\ft$ the product map of $\iota_j$ with the identity of $\prod_{i=1}^b\Delta^{P_i}$. Notice that $\hat\iota_j$ is a cellular isomorphism between $\Delta^{b-1}\times\prod_{i=1}^b\Delta^{P_i}$ and $\del_j^\hor\Delta^\ft$.

The restriction of $\Phi^\ft$ on $\del_0^\hor\Delta^\ft$ is again the constant map to $\infty$. Geometrically, this corresponds to the first bar going all the way to left and colliding onto $\partial \cM$.

For $1\le j\le b-1$, the restriction of $\Phi^\ft$ on $\del_j^{\hor}\Delta^\ft$ corresponds geometrically to letting the $j$\sth and $(j+1)$\textsuperscript{st} bars collide;
the map $\Phi^\ft\circ\hat\iota_j$ is equal to the composition
\[
\begin{tikzcd}[column sep=50pt]
\Delta^{b-1}\times \Delta^{P_1}\times\dots\times \Delta^{P_j}\times \Delta^{P_{j+1}}\times \dots \times \Delta^{P_b}\dar["\id\times\sort_{P_j,P_{j+1}}\times \id"']& \\
\Delta^{b-1}
\times \Delta^{P_1}\times\dots\times \Delta^{P_j+P_{j+1}}\times \dots \times 
\Delta^{P_b}\ar[r,"\Phi^{(b-1,\del_j\uP,0)}"] &
\conf{\cM},
\end{tikzcd}
\]
where $\del_j\uP=(P_1,\dots,P_{j-1},P_j+P_{j+1},P_{j+2},\dots,P_b)$ and $\sort_{P_j,P_{j+1}}$ is the map from Definition \ref{defn:sortmap}.
Thus the restriction of $\Phi^\ft$ on $\del_j^\hor\Delta^\ft$ is cellular, as it is the composite of cellular maps: by inductive hypothesis we know that $\Phi^{(b-1,\del_j\uP,0)}$ is cellular.

Finally, we have to treat the case $j=b$, which corresponds geometrically to the right-most bar colliding with $\iota(V_{2g})$. 
Consider the element $\zeta_g\in\pi_1(V_{2g},*)=\Z^{*2g}$ from Definition \ref{defn:zetag},
and let $f_{\zeta_g}:(V_1,*)\to (V_{2g},*)$ be the based map realising the group homomorphism $\gamma_1\mapsto \zeta_g$ ``with constant speed'': concretely, recalling that $V_1=[0,1]/\set{0,1}$ and that $\cM$ is a quotient of the rectangle $\bfR=[0,2]\times[0,1]$, then the map $[0,1]\to\bfR$ with $t\mapsto (2,t)$ descends to a map $f_{\zeta_g}\colon V_1\to \cM$ whose image is contained in the subspace $\iota(V_{2g})\subset\cM$.
Then $\Phi^\ft\circ\hat\iota_b$ is equal to the composition
\[
\begin{tikzcd}[column sep=40pt]
\Delta^{b-1}\times \Delta^{P_1}\times\dots\times\Delta^{P_{b-1}}\times\Delta^{P_b}\dar["\id\times\Phi_{P_b}"]& \\
\Delta^{b-1}
\times \Delta^{P_1}\times\dots\times\Delta^{P_{b-1}}\times
\confred[P_b]{V_1}\ar[d,"{\Phi^{(b-1,\del_b\uP,0)}\times\conf[P_b]{f_{\zeta_g}}}"] &\\
\conf{\cM}\times\conf{\cM}\ar[r,"\mu"'] &\conf{\cM},
\end{tikzcd}
\]
where we use the following:
\begin{itemize}
 \item the characteristic map $\Phi_{P_b}\colon \Delta^{P_{b-1}}\to\confred[P_b]{V_1}$ in the cell decomposition of $\confred[P_b]{V_1}$ corresponds to the vector $\fv=(P_b)\in\Z_{\ge0}^1$, i.e. to the cell $e(1,P_b)$, as in Subsection \ref{subsec:cellstratbouquet};
 \item $\del_b\uP$ denotes the sequence $(P_1,\dots,P_{b-1})$;
 \item the map 
 \[
 \conf[P_b]{f_{\zeta_g}}\colon\confred[P_b]{V_1}=\conf[P_b]{V_1,*}\to\conf[P_b]{\cM}=\conf[P_b]{\cM,\del\cM}
 \]
 is induced by the pointed map $f_{\zeta_g}$, and we include $\conf[P_b]{\cM,\del\cM}$ inside $\conf{\cM,\del\cM}$;
\end{itemize}
In particular $\Phi^\ft\circ\hat\iota_b$ is exhibited as a composition of cellular maps, again using induction on $b$. This concludes the proof that $\Phi^\ft$ is cellular.

We can in fact use the previous analysis to compute the differential $\del e_\ft\in\widetilde{\Ch}_*^{B}(\cM)$ as a sum of contributions coming from the different faces of $\Delta^\ft$. All faces $\del^{\ver,i}_j\Delta^\ft$ and the face $\del_0^\hor\Delta^\ft$ will contribute trivially. For $1\le j\le b-1$, the face $\del_j^\hor\Delta$ will contribute by a multiple of $e_{(b-1,\del_j\uP,\uzero)}$, with coefficient being $(-1)^j\cdot\deg(\sort_{P_j,P_{j+1}})$, which by Lemma \ref{lem:degreesort} gives $(-1)^j\cdot\Ss(P_j,P_{j+1})$. These contributions give the component $\del_B(e_\ft)$.
Finally, the face $\del_b^\hor\Delta^\ft$ contributes by the product
\[
e_{(b-1,\del_b\uP,\uzero)}\cdot \UMor(f_{\zeta_g})(e(1,P_b))\in  \redchains{\conf{\cM}},
\]
using the tensor product decomposition from Proposition \ref{prop:decompositionintobarandUMor}.

By Theorem \ref{thm:UMormaps}, if $P_b=2k+1$, then 
\[
\UMor(f_{\zeta_g})(e(1,P_b))=\UMor(f_{\zeta_g})(x\frac{y^k}{k!})=[\zeta_g]_x\frac{([c_2(\zeta_g)]_x)^k}{k!}=0,
\]
since $\zeta_g$ is a commutator and hence $[\zeta_g]_x=0$; similarly, if $P_b=2k$, then
\[
\UMor(f_{\zeta_g})(e(1,P_b))=\UMor(f_{\zeta_g})(\frac{y^k}{k!})=\frac{([c_2(\zeta_g)]_x)^k}{k!}=\frac{(2\omega)^k}{k!};
\]
in the second to last equality we use that $\UMor_\bullet(2g)$ is torsion-free as an abelian group, and in the last equality we use Lemma \ref{lem:contentomega}.
In both cases
\[
\UMor(f_{\zeta_g})(e(1,P_b))=\Omega_{P_b}.
\]
\end{proof}
\subsection{Cellular approximation of diffeomorphisms}
We let $D:=[0,1]\times [0,1]\subset \cM$ and identify the mapping class group $\Gamma_{g,1}$ with $\pi_0(\Homeo(\cM,D\cup \partial\cM))$.
\begin{defn}
 \label{defn:tau}
 For $1\le t\le 2$ we define the self-map $\theta_t$ of the pair $(\cM,\partial M)$ as the map with
\[ 
\fp(x,y)\mapsto \left\{
\begin{array}{cl}
\fp(tx,y) &\text{if } x\le \frac2t,\\[.2cm]
\fp(2,y) &\text{if } x\ge \frac2t.
\end{array}
\right.
\]
For $t=2$, we also write $\theta=\theta_2$. See Figure \ref{fig:tau}.
\end{defn}
The maps $\theta_t$ assemble into a homotopy of self-maps of the pair $(\cM,\del\cM)$ starting from the identity and ending with $\theta$.
\begin{defn}
 We say that a self-map $f$ of the pair $(\cM,\partial\cM)$ is \emph{vertical} if for all $0\le x\le 2$ there is $0\le x'\le 2$ such that $f(\fp(\set{x}\times[0,1]))\subseteq\fp(\set{x'}\times[0,1])$.
\end{defn}
Each map $\theta_t$ is vertical, and for any homeomorphism $f\in\Homeo(\cM,D\cup \partial\cM)$ the composition $\theta\circ f$  (which is no longer a homeomorphism) is vertical. Compositions of vertical maps are vertical.

\begin{figure}[ht]
 \centering
 \begin{overpic}[width=12cm]{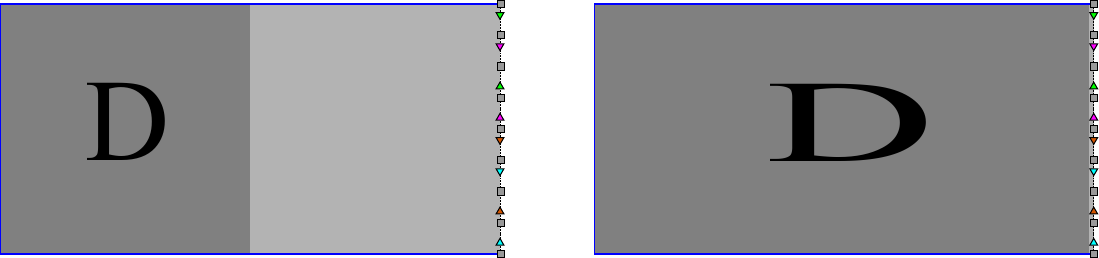}
 \put(49,10){$\overset{\theta}{\to}$}
 \end{overpic}
 \caption{The map $\theta$ expands $D$ and contracts horizontally the rest of $\cM$.}
 \label{fig:tau}
\end{figure}

\begin{prop}\label{prop:tauiscellular}
Let $f$ be a vertical self-map of the pair $(\cM,\del\cM)$; then the induced map $\conf{f}\colon\conf{\cM}\to\conf{\cM}$ is cellular.
\end{prop}
\begin{proof}
Let $n\ge0$ be fixed; in the notation of the proof of Proposition \ref{prop:muiscellular}, it suffices to show the inequality $d(\conf[n]{f}(s))\le d(s)$, for all $s\in \conf[n]{\cM}$. If $\conf[n]{f}(s)=\infty$, this holds immediately. Otherwise, we have $n(\conf[n]{f}(s))=n(s)=n$; moreover, since $f$ is vertical, we have $b(\conf[n]{f}(s))\le b(s)$. Using the equality $d(-)=n(-)+b(-)$, the conclusion follows.
\end{proof}
In particular every vertical map of $(\cM,\del\cM)$ induces an endomorphism of the chain complex $\redchains{\conf{\cM}}$. The homotopy of vertical maps $\theta_t$ induces a homotopy of cellular self-maps of $\conf{\cM}$, which in turn induces a \emph{homotopy through self-chain maps} of the chain complex $\redchains{\conf{\cM}}$; clearly the space of self-chain maps of a chain complex is discrete and, furthermore, the mapping degree for maps between spheres, the quantity which determines chain maps, is locally constant; we conclude that $\id_{\cM}$ induces \emph{the same} self-chain map on $\redchains{\conf{\cM}}$ as $\theta$; in other words, $\theta$ acts as the identity on $\redchains{\conf{\cM}}$.
More generally, if $f\simeq g$ are isotopic homeomorphisms in $\Homeo(\cM,D\cup \partial \cM)$, then $f\circ\theta$ is homotopic to $g\circ\theta$ through vertical maps, hence
$\conf{\theta\circ f}$ is homotopic to $C_n(\theta\circ g)$ through cellular maps, and hence $\theta\circ f$ and $\theta\circ g$ induce the same self-chain map of $\redchains{\conf{\cM}}$.

Finally, if $f,g\in\Homeo(\cM,D\cup \partial \cM)$, then $\theta\circ f\circ\theta\circ g$ is homotopic to $\theta\circ f\circ g$ through vertical maps, so that again $\theta\circ f\circ\theta\circ g$ and $\theta\circ f\circ g$ induce the same 
self-chain map of $\redchains{\conf{\cM}}$. This leads to the following proposition.

\begin{prop}
 The cellular chain complex $\redchains{\conf{\cM}}$ admits an action of $\Gamma_{g,1}$ by self-chain maps, defined by letting the class of a homeomorphism $f\in \Homeo(\cM,D\cup \cM)$ act by the chain map induced by the cellular map $\conf{\theta\circ f}$.
\end{prop}

\begin{prop}\label{prop:UMorissubrepandCHbistrivial}
Under the decomposition of Proposition \ref{prop:decompositionintobarandUMor}, the subrings
\[
\widetilde{\Ch}^{B}_*(\cM)\ \mbox{ and }\ \UMor_\bullet(2g)
\]
of $\redchains{\conf{\cM}}$ are $\Gamma_{g,1}$-subrepresentations. More precisely, the $\Gamma_{g,1}$-action on $\widetilde{\Ch}^{B}_*(\cM)$ is trivial, and $f\in \Homeo(\cM,D\cup \cM)$ acts on $\UMor_\bullet(2g)$ via
\[
C_{\bullet}(\theta\circ f|_{\iota(V_{2g})})_*=\UMor(\theta\circ f|_{\iota(V_{2g})})=\UMor(\phi),
\]
where $\phi\in \Aut(\Z^{*2g})$ is the automorphism of $\pi_1(\Sigma_{g,1})\cong \pi_1(V_{2g},*)$ induced by $f$.
\end{prop}
\begin{proof}
For any $f\in \Homeo(\cM,D\cup \partial \cM)$, 
the self-map $\conf{\theta\circ f}$ of $\conf{\cM}$ is induced by a self-map of $(\cM,\del\cM)$, so it respects the product $\mu$; it follows that $\UMor(\theta\circ f|_{\iota(V_{2g})})$ is a ring automorphism.
Let now $e_\ft$ be a generator in $\widetilde{\Ch}^{B}_*(\cM)$, for some record $\ft$ of the form $(b,\uP,\uzero)$ with $b\ge1$. We want to study the behaviour of $\conf{\theta\circ f}$ on the open cell $e_\ft\subseteq\conf{\cM}$. Let $e^{1/2}_\ft$ be the subspace of $e_\ft$ containing configurations $s$ which are supported in the left, open half of $\mathring{\bfR}$, i.e. for each point $\fp(x,y)\in s$ we have $0<x<1$. Then $e^{1/2}_\ft$ is open in $e_\ft$; moreover $\theta\circ f$ maps $e^{1/2}_\ft$ homeomorphically onto $e_\ft$, and maps the complement of $e^{1/2}_\ft$ to configurations $s'$ satisfying $b(s')< b(s)=b$. It follows that the generator $e_\ft\in\redchains{\conf{\cM}}$ is fixed by the chain map $\conf{\theta\circ f}_*$; this proves that $\widetilde{\Ch}^{B}_*(\cM)$ is invariant under the action of $\Gamma_{g,1}$, and the restricted action is trivial.

The fact that $\UMor_\bullet(2g)$ is also invariant under the action of $\Gamma_{g,1}$, and that the restricted action coincides with the one described in Theorem \ref{thm:UMormaps}, follows from the fact that $\UMor_\bullet(2g)$ is the image of the ring homomorphism $\conf{\iota}_*$.
\end{proof}

This proposition combined with Proposition \ref{prop:equivariantPD} reduces the problem of determining the $\Gamma_{g,1}$-representation $H_*(C_n(\Sigma_{g,1}))$ to studying the action of $\Gamma_{g,1}$ $\UMor_{\bullet}(2g)$. This is done in the next section.

\section{The \texorpdfstring{$\Gamma_{g,1}$}{Gammag1}-action on the homology of configurations}
We fix a prime number $p$ and $g\ge0$ throughout the section. We abbreviate $G=G_{2g}=\Z^{*2g}\cong\pi_1(\Sigma_{g,1})$, with the isomororphism specified in Section \ref{sec:confsurf}, and $H=G^{\ab}\cong H_1(\Sigma_{g,1})$.

\begin{figure}[h]
    \centering
    \begin{tikzpicture}
    
    \draw
    (0,0) ellipse (0.3 and 1) ;
    \draw (-0.3,0.5) node[left]{\tiny$\partial\Sigma_{g,1}$};
    \draw (0,1) -- (8,1)  .. controls (9,1) and (10,1)  .. (10,0)  .. controls (10,-1) and (9,-1) .. (8,-1)  -- (0,-1);
    \draw (10,-0.7) node[below]{$\Sigma_{g,1}$};
    \draw ({0.3*cos(30)},-0.5) to (8.3,-0.47);
    
    \filldraw[black] ({0.3*cos(30)},-0.5) circle (1pt) node[anchor=east]{\tiny$p_0$};

    \begin{scope}[shift={(2,-0.5)}, scale=0.5]
        \draw[thick, looseness=0.8, postaction={decoration={markings,mark=at position 0.4 with {\arrow{>}}},decorate}] (0,0) to[out=80,in=-90]  (1.7,1) to [out=90,in=0] node[above]{\small$\gamma_1$} (0,2) to[out=180, in=90] (-2.3,1)  to[out=-90,in=100] (0,0);
        \draw[thick, looseness=0.8, postaction={decoration={markings,mark=at position 0.4 with {\arrow{>}}},decorate}] (-0.3,-1) to[out=30,in=-90] node[right]{\small$\gamma_2$} (0,0) to[out=90,in=-30] (-0.3,1);
        
        \draw[thick, looseness=0.8,dashed] (-0.3,-1) to[out=150,in=-90] (-0.6,0) to[out=90,in=-150] (-0.3,1);
        
        \begin{scope}[shift={(-0.3,1)}]
        \draw (-1,0.5) to[out=-60,in=180] (0,0) to[out=0,in=-120] (1,0.5);
        \draw (-0.6,0.12) to[out=60,in=180] (0,0.4) to[out=0,in=120] (0.6,0.12);
        \end{scope}
    \end{scope}
    
    \begin{scope}[shift={(4.6,-0.5)}, scale=0.5]

        \draw[thick, looseness=0.8, postaction={decoration={markings,mark=at position 0.4 with {\arrow{>}}},decorate}] (0,0) to[out=80,in=-90]  (1.7,1) to [out=90,in=0] node[above]{\small$\gamma_3$} (0,2) to[out=180, in=90] (-2.3,1)  to[out=-90,in=100] (0,0);
        \draw[thick, looseness=0.8, postaction={decoration={markings,mark=at position 0.4 with {\arrow{>}}},decorate}] (-0.3,-1) to[out=30,in=-90] node[right]{\small$\gamma_4$} (0,0) to[out=90,in=-30] (-0.3,1);
        \draw[thick, looseness=0.8,dashed] (-0.3,-1) to[out=150,in=-90] (-0.6,0) to[out=90,in=-150] (-0.3,1);
        
        \begin{scope}[shift={(-0.3,1)}]
        \draw (-1,0.5) to[out=-60,in=180] (0,0) to[out=0,in=-120] (1,0.5);
        \draw (-0.6,0.12) to[out=60,in=180] (0,0.4) to[out=0,in=120] (0.6,0.12);
        \end{scope}
    \end{scope}
    
    \draw (6,0) node[anchor=west]{$\dots$};
    \begin{scope}[shift={(8.3,-0.5)}, scale=0.5]
        \draw[red, thick, looseness=0.8] (0.1,-1) to[out=30,in=-90] node[right]{\small$a$}  (0.4,0) to[out=90,in=-30] (0.1,1.05);
        \draw[red, looseness=0.8,dashed] (0.1,-1) to[out=150,in=-90] (-0.2,0) to[out=90,in=-150] (0.1,1.05);
        
        \draw[thick, looseness=0.8, postaction={decoration={markings,mark=at position 0.45 with {\arrow{>}}},decorate}] (0,0) to[out=80,in=-90]  (1.7,1) to [out=90,in=0] node[above]{\small$\gamma_{2g-1}$} (0,2) to[out=180, in=90] (-2.3,1)  to[out=-90,in=100] (0,0);
        
        \draw[thick, looseness=0.8, postaction={decoration={markings,mark=at position 0.4 with {\arrow{>}}},decorate}] (-0.3,-1) to[out=30,in=-90] (0,0) to[out=90,in=-30] (-0.3,1);
        
        \draw[thick, looseness=0.8,dashed] (-0.3,-1) to[out=150,in=-90] node[left]{\small$\gamma_{2g}$}(-0.6,0) to[out=90,in=-150] (-0.3,1);
        
        \begin{scope}[shift={(-0.3,1)}]
        \draw (-1,0.5) to[out=-60,in=180] (0,0) to[out=0,in=-120] (1,0.5);
        \draw (-0.6,0.12) to[out=60,in=180] (0,0.4) to[out=0,in=120] (0.6,0.12);
        \end{scope}
    \end{scope}

    \end{tikzpicture}
    \caption{The standard generators of $G=\pi_1(\Sigma_{g,1},p_0)$ and the curve $a$ from the proof of Lemma \ref{lem:pthpowersarekilled}.}
    \label{fig:SymplecticBasis}
\end{figure}
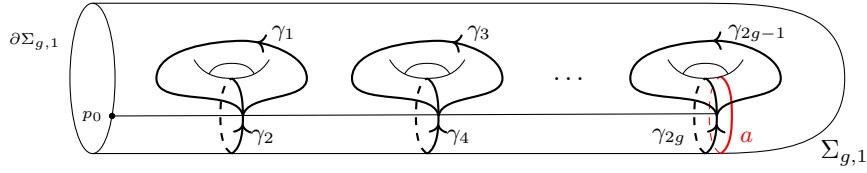

\subsection{Recollections on Torelli groups and Johnson homomorphisms}

\begin{defn}
The \emph{Torelli group}, denoted $\cT_{g,1}$, is the kernel of the $\Gamma_{g,1}$-action on $H_1(\Sigma_{g,1};\Z)$. The \emph{Torelli group modulo $p$}, denoted $\cT_{g,1}(p)$, is the kernel of the $\Gamma_{g,1}$ action on $H_1(\Sigma_{g,1};\Fp)$.
\end{defn}

\begin{thm}[Cooper \cite{Cooper}, Perron \cite{Perron}]\label{thm:torellimoduloPgeneration}
For a prime $p$, the group $\cT_{g,1}(p)$ is generated by $\cT_{g,1}$ and the $p$\sth powers of all Dehn twists.
\end{thm}

In \cite{JohnsonI},  Johnson proved that the quotient $[G,G]/[G,[G,G]]$ of lower central subgroups of $G$ is a free abelian group, isomorphic to $\Lambda^2H$ as $\Gamma_{g,1}$-representation. The explicit isomorphism $j\colon[G,G]/[G,[G,G]]\to\Lambda^2H$ is given by sending the class of $w_1w_2w_1^{-1}w_2^{-1}$ in $[G,G]/[G,[G,G]]$ to $[w_1]\wedge [w_2]\in \Lambda^2H$, for $w_1,w_2\in G$. 
\begin{defn}[Johnson \cite{JohnsonI}]
The \emph{Johnson homomorphism} is the group homomorphism
$\tau_{g,1}:\cT_{g,1}\to \Hom(H,\Lambda^2H)$
sending $\phi\in \cT_{g,1}$ to the homomorphism of abelian groups $[w]\mapsto j([\phi(w)w^{-1}]_{[G,[G,G]]})$.
\end{defn}
Johnson showed that $\tau_{g,1}$ is a well defined group homomorphism, and computed its image using the identification $\Hom(H,\Lambda^2H)\cong H^{\vee}\otimes \Lambda^2H\cong H\otimes \Lambda^2H$, where the self duality $H\cong H^{\vee}$ comes from the symplectic form on $H=H_1(\Sigma_{g,1})$. 
\begin{thm}[Johnson \cite{JohnsonI}]\label{thm:Johnsonimage}
    The image of $\tau_{g,1}$ is the subgroup $\Lambda^3H\subset H\otimes \Lambda^2H$, where the inclusion is given by sending the triple wedge $a\wedge b\wedge c\in \Lambda^3H$ to the sum
    \[
    a\otimes b\wedge c+b\otimes c\wedge a+c\otimes a\wedge b\in H\otimes \Lambda^2H.
    \]
\end{thm} 
We observe that $H\otimes \Lambda^2H$ is a free abelian group of rank $2g\binom{2g}{2}$, and $\Lambda^3H$ is a split free abelian subgroup of rank $\binom{2g}{3}$.

The kernel of $\tau_{g,1}$ is the so-called \emph{Johnson kernel} $\cJ_{g,1}\subseteq \cT_{g,1}$.
\begin{thm}[Johnson \cite{JohnsonI}]\label{thm:septwistgenerateJ}
    If $g\ge 0$, the group $\cJ_{g,1}$ is generated by all Dehn twists around separating simple closed curves in $\Sigma_{g,1}$.
\end{thm}

\subsection{The maps \texorpdfstring{$\xi$}{xi} and \texorpdfstring{$\xi^p$}{xip}}
\begin{nota}
\label{nota:UMor2splitting}
Recall Proposition \ref{prop:UMoralgebra} and Theorem \ref{thm:UMormaps}; we identify the abelian group $\UMor_{-2}(2g)$ with
\[
 \UMor_{-2}(2g)=\bigoplus_{i=1}^{2g}\Z y_i\oplus\Lambda^2\pa{\bigoplus_{i=1}^{2g}\Z x_i}\cong H\oplus\Lambda^2 H.
\]
\end{nota}
We observe that $\UMor_{-2}(2g)$ has a na\"ive $\Gamma_{g,1}$-action as the direct sum of the representations $H$ and $\Lambda^2H$. In light of Proposition \ref{prop:UMorissubrepandCHbistrivial}, it also inherits an action via $\UMor_{-2}(\theta\circ-)$.
For both actions of $\Gamma_{g,1}$ on $\UMor_{-2}(2g)$ we have a short exact sequence of $\Gamma_{g,1}$-representations
\[
 \Lambda^2H\to\UMor_{-2}(2g)\to H,
\]
and for the na\"ive action there is a $\Gamma_{g,1}$-equivariant splitting $H\to\UMor_{-2}(2g)$.
\begin{nota}
For $\phi\in \Gamma_{g,1}$ denote the na\"ive action by $\phi*_H-$ and the action via $\UMor_\bullet(\theta\circ-)$ by $\phi*-$. 
\end{nota}
\begin{defn}
    We define a function of sets $\xi:\Gamma_{g,1}\to \Hom(H,\Lambda^2H)$ by sending $\phi\in\Gamma_{g,1}$ to the homomorphism of abelian groups sending the basis element $[\gamma_i]\in H$ to the element $c_2(\phi(\gamma_i))\in \Lambda^2H$.
    Composing with the modulo $p$ reduction, we similarly define $\xi^p:\Gamma_{g,1}\to \Hom(H,\Lambda^2H)\otimes \Fp$.
    
    We let $\xi_{\tau}$ (resp. by $\xi^p_{\tau}$) be the restriction of $\xi$ (resp. $\xi^p$) on $\cT_{g,1}$ (resp. on $\cT_{g,1}(p)$).
\end{defn}
The function $\xi$ enjoys a special relationship with the Johnson homomorphism.
\begin{prop}\label{prop:xiequalsJohnson}
    $\xi_{\tau}=2\tau_{g,1}$.
\end{prop}
\begin{proof}
    This is contained in the proof of \cite[Proposition 6.7]{Stavrou}. 
\end{proof}

\begin{prop}\label{prop:actiononUMorviaxi}
    Let $a\in H\subset\UMor_{-2}(2g)$ and $b\in \Lambda^2H\subset\UMor_{-2}(2g)$; then for $\phi\in\Gamma_{g,1}$ we have
    $\phi*(a\oplus b)=\phi*_H a\oplus (\phi*_H b + \xi(\phi)(a))$.
\end{prop}
\begin{proof}
    Applying Theorem \ref{thm:UMormaps} proves the equality on the basis of $\UMor_{-2}(2g)$ given by the elements $y_i$ and $x_j\wedge x_k$, for $1\le i,j,k\le 2g$.
\end{proof}
\begin{prop}\label{prop:xionproducts}
    For $\phi,\psi\in \Gamma_{g,1}$ and $a\in H\subset\UMor_{-2}(2g)$, we have
    \[
    \xi(\phi\psi)(a)=\phi*_H\xi(\psi)(a)+\xi(\phi)(\psi*_H a).
    \]
\end{prop}
\begin{proof} We leave the proof to the reader.
\end{proof}

\begin{cor}
\label{cor:xigrouphomomorphism}
The maps $\xi_{\tau}$ and $\xi^p_{\tau}$ are group homomorphisms.
\end{cor}
\begin{proof}
    By definition of $\cT_{g,1}$, for any $\phi\in \cT_{g,1}$, the map $\phi*_H-$ is the identity map. Then, for $\phi,\psi\in \cT_{g,1}$, Proposition \ref{prop:xionproducts} yields
    \[
    \xi_{\tau}(\phi\psi)(a)=\xi_{\tau}(\phi)(a)+\xi_{\tau}(\psi)(a)\ \mbox{ for all }\ a\in H.
    \]
    Thus $\xi_{\tau}$ is a group homomorphism. 
    The same argument modulo $p$ works for $\xi_{\tau}^p$, using that the na\"ive action of $\cT_{g,1}(p)$ on $H\otimes\Fp$ is trivial.
\end{proof}

\begin{rem}
It can be shown that $\xi^p_{\tau}$ is equal to $2\tau^{Z}_1(p)$, where
\[
\tau^{Z}_1(p)\colon \cT_{g,1}(p)\to\Hom(H,\Lambda^2H)\otimes\Fp
\]
is the Johnson-Zassenhaus homomorphism from  \cite{Cooper, Zassenhaus}. However the definition of $\tau^{Z}_1(p)$ is a bit complicated, hence we prefer to avoid comparing $\xi^p_{\tau}$ to $\tau^{Z}_1(p)$ and directly prove the needed properties of $\xi^p_{\tau}$.
\end{rem}

The normal subgroup $\cT_{g,1}\subset \Gamma_{g,1}$ admits an action of $\Gamma_{g,1}$ by conjugation, and the abelian group $\Hom(H,\Lambda^2H)$ has a $\Gamma_{g,1}$-action defined by
\[
f\mapsto (\phi*_H-)\circ f\circ (\phi^{-1}*_H-)\ \mbox{ for }\ \phi\in \Gamma_{g,1}\ \mbox{ and }\ f\in \Hom(H,\Lambda^2H).
\]
Similarly, $\Gamma_{g,1}$ acts on $\cT_{g,1}(p)$ and $\Hom(H,\Lambda^2H)\otimes\Fp$.
\begin{lem}\label{lem:xigammag1equivariant}
    Under the above actions, $\xi_{\tau}$ and $\xi_{\tau}^p$ are $\Gamma_{g,1}$-equivariant.
\end{lem}
\begin{proof}
    
    Applying Proposition \ref{prop:xionproducts} twice, for $a\in H\subset\UMor_{-2}(2g)$ and $\phi,\psi\in\cT_{g,1}$, we get a chain of equalities
    \begin{align*}
        &\xi_{\tau}(\phi\psi\phi^{-1})(a)= (\phi\psi)*_H\xi_{\tau}(\phi^{-1})(a)+\xi_{\tau}(\phi\psi)(\phi^{-1}*_Ha)\\
        &=\phi*_H\psi*_H\xi_{\tau}(\phi^{-1})(a)+\pa{\phi*_H\xi_{\tau}(\psi)(\phi^{-1}*_Ha)+\xi_{\tau}(\phi)(\psi*_H\phi^{-1}*_Ha)}.
    \end{align*}
    But, as $\psi\in \cT_{g,1}$, the sum of the first and third summand vanishes:
    \begin{equation*}
        \phi*_H\xi_{\tau}(\phi^{-1})(a)+\xi_{\tau}(\phi)(\phi^{-1}*_Ha)     =\xi_{\tau}(\phi\phi^{-1})(a)
        = \xi_{\tau}(\id)(a)
        =0.
    \end{equation*}
    Again, we used Proposition \ref{prop:xionproducts}. The conclusion follows. The argument for $\xi_{\tau}^p$ is analogous, considering the above chains of equalities modulo $p$.
\end{proof}
\begin{rem}
    The definition of the maps of sets $\xi$ and $\xi^p$ really depends on a choice of generators for the free group $G$; Lemma \ref{lem:xigammag1equivariant} however proves that the restrictions $\xi_{\tau}$ and $\xi_{\tau}^p$ are independent of such choice.
\end{rem}
Recall from Theorem \ref{thm:septwistgenerateJ} that all separating Dehn twists lie in $\cT_{g,1}\subseteq\cT_{g,1}(p)$, and from Theorem \ref{thm:torellimoduloPgeneration} that all $p$\sth powers of Dehn twists lie in $\cT_{g,1}(p)$.
\begin{lem}\label{lem:pthpowersarekilled}
All $p$\sth powers of non-separating Dehn twists are in the kernel of $\xi^p_{\tau}$.
All separating Dehn twists are in the kernel of $\xi_{\tau}$.
\end{lem}
\begin{proof}
All non-separating Dehn twists in $\Gamma_{g,1}$ are conjugate to each other, so it suffices to show that the $p$\sth power of a specific Dehn twist is in the kernel of $\xi^p$.
Let $a$ be a simple closed curve freely homotopic to $\gamma_{2g}$, as in Figure \ref{fig:SymplecticBasis}, and let $D_a\in\Gamma_{g,1}$ be the associated Dehn twist. Then $D_a(\gamma_{2g})=\gamma_{2g-1}\gamma_{2g}$, whereas $D_a(\gamma_i)=\gamma_i$ for all $1\le i\le2g-1$. 
It follows that $D_a^p$ sends $\gamma_{2g}\mapsto\gamma_{2g-1}^p\gamma_{2g}$ and fixes the other generators $\gamma_i$; therefore the group homomorphism
$\xi(D_a^p)\colon H\to\Lambda^2H$ sends 
\[
[\gamma_{2g}]\mapsto c_2(D_a^p(\gamma_{2g}))=\binom{p}{2}[\gamma_{2g-1}]\wedge[\gamma_{2g-1}]+ p[\gamma_{2g-1}]\wedge [\gamma_{2g}]=p[\gamma_{2g-1}]\wedge [\gamma_{2g}]
\]
and sends $\gamma_i\mapsto c_2(D_a^p(\gamma_i))=c_2(\gamma_i)=0$ for $i\neq 2g$.
Modulo $p$, we obtain the zero group homomorphism, i.e. $\xi^p(D^p_a)=0$.

To prove that separating Dehn twists lie in the kernel of $\xi_{\tau}$, we use the equality $\xi_{\tau}=2\tau_{g,1}$ together with the fact that by Theorem \ref{thm:septwistgenerateJ} the kernel of $\tau_{g,1}$ is precisely the subgroup of $\Gamma_{g,1}$ generated by all separating Dehn twists.
\end{proof}

\begin{prop}\label{prop:Fpimageofxi}
Let $p$ be an odd prime. Then the image of $\xi^p_{\tau}$ is isomorphic to $\Lambda^3H\otimes\Fp\subset H\otimes\Lambda^2H\otimes\Fp$, and the kernel of $\xi^p_{\tau}$ is generated by the kernel of $\xi_{\tau}$ and $p$\sth powers of Dehn twists.
\end{prop}
\begin{proof}
By Lemma \ref{lem:pthpowersarekilled} all $p$\sth powers of Dehn twists are in the kernel of $\xi^p_{\tau}$, and by Theorem \ref{thm:torellimoduloPgeneration} they generate $\cT_{g,1}(p)$ together with $\cT_{g,1}$; it follows that
$$\Imm(\xi^p_{\tau})=\Imm(\xi^p_{\tau}|_{\cT_{g,1}})=\Imm(\xi_{\tau}\otimes \Fp).$$
Using $\xi_{\tau}=2\tau_{g,1}$ and the assumption that $p$ is odd, we have that the maps
\[
\xi_{\tau}\otimes\Fp,\tau_{g,1}\otimes\Fp\colon\cT_{g,1}\to \Hom(H,\Lambda^2H)\otimes\Fp
\]
have the same image and the same kernel.

Recall from Theorem \ref{thm:Johnsonimage} that $\Imm(\tau_{g,1})\subset\Hom(H,\Lambda^2H)\cong H\otimes\Lambda^2H$ is the split subgroup $\Lambda^3H$; thus we can identify $\Imm(\xi_{\tau}\otimes\Fp)=\Imm(\tau_{g,1}\otimes\Fp)\subset H\otimes\Lambda^2H\otimes\Fp$ with $\Imm(\tau_{g,1})\otimes\Fp\cong\Lambda^3H\otimes\Fp$; this proves the first claim.

For the second claim, we first observe that $\ker(\xi_{\tau}^p)$ is generated by $p$\sth powers of Dehn twists together with $\ker(\xi_{\tau}\otimes\Fp)=\ker(\tau_{g,1}\otimes\Fp)\subset\cT_{g,1}$. From the first isomorphism theorem, there is a short sequence of groups 
$$\ker(\tau_{g,1})\to \ker(\tau_{g,1}\otimes\Fp)\to p\Hom(H,\Lambda^2H)\cap \Imm(\tau_{g,1});$$
the third term can be identified with $p\Lambda^3H=p\Imm(\tau_{g,1})$, using again that the inclusion $\Lambda^3H\subset H\otimes\Lambda^2H$ is split. 
In particular $\ker(\tau_{g,1}\otimes\Fp)$ can be generated by $\ker(\tau_{g,1})$ and lifts in $\cT_{g,1}$ of generators of the group $p\Lambda^3H$; for the latter we may take the $p$\sth powers of a generating set of $\cT_{g,1}$. We now observe that $\cT_{g,1}$ is generated by separating twists and \emph{bounding pairs} $D_aD_{a'}^{-1}$, i.e. products of a Dehn twist $D_a$ with the inverse of a Dehn twist $D_{a'}$, where $a,a'\subset\Sigma_{g,1}$ are disjoint simple closed curves cobounding a subsurface of $\Sigma_{g,1}$. We conclude that the following elements generate $\ker(\xi^p_{\tau})$:
\begin{enumerate}
 \item all $p$\sth powers of all Dehn twists;
 \item all separating Dehn twists;
 \item all $p$\sth powers of bounding pairs.
\end{enumerate}
We may drop the third kind of generators from the generating set: for a bounding pair $D_aD_{a'}^{-1}$, we have that the Dehn twists $D_a$ and $D_{a'}$ commute, hence $(D_aD_{a'}^{-1})^p=D_a^pD_{a'}^{-p}$, which is already an element in the subgroup of $\Gamma_{g,1}$ generated by $p$\sth powers of Dehn twists.
\end{proof}
Lemma \ref{lem:pthpowersarekilled} and
Proposition \ref{prop:Fpimageofxi} justify the following definition.
\begin{defn}
 For an odd prime $p$, the mod-$p$ Johnson kernel is  $\cJ_{g,1}(p):=\ker(\xi^p_{\tau})\subseteq\cT_{g,1}(p)$, i.e. the subgroup generated by separating Dehn twists and $p$\sth powers of Dehn twists.
\end{defn}

\subsection{The \texorpdfstring{$\Gamma_{g,1}$}{Gammag1}-action on integral homology}\label{subsec:Zhomologyaction}

\begin{prop}\label{prop:ZactiononUMor}
The kernel of the $\Gamma_{g,1}$-action on $\UMor_\bullet(2g)$ is the Johnson kernel $\cJ_{g,1}$.
\end{prop}
\begin{proof}
Let $\phi\in\Gamma_{g,1}$ act trivially on $\UMor_{-2}(2g)\cong H\oplus\Lambda^2H$, where we use the splitting from Notation \ref{nota:UMor2splitting}. Then $\phi$ acts trivially on the subrepresentation $\Lambda^2H\subset\UMor_{-2}(2g)$ and on the quotient $\UMor_{-2}(2g)/\Lambda^2H\cong H$, so that $\phi\in\cT_{g,1}$.
Furthermore, by Proposition \ref{prop:actiononUMorviaxi} $\phi\in \cT_{g,1}$ acts trivially on $H\subset\UMor_{-2}(2g)$ precisely if $\xi_{\tau}(\phi)=0$. Since $\xi_{\tau}=2\tau_{g,1}$ and the target of $\xi_{\tau}$ is torsion-free, it follows that $\tau_{g,1}(\phi)=0$ as well, and thus $\phi\in \cJ_{g,1}$.

Conversely, $\cJ_{g,1}\subseteq\cT_{g,1}$ is generated by separating Dehn twists; hence $\cJ_{g,1}$ acts trivially on $\UMor_{-1}(2g)$, and combining Lemma \ref{lem:pthpowersarekilled} and Proposition \ref{prop:actiononUMorviaxi} we obtain that $\cJ_{g,1}$ also acts trivially on $\UMor_{-2}(2g)$. The action of $\cJ_{g,1}$ on $\UMor_\bullet(2g)$ is by ring automorphisms, and the torsion-free ring $\UMor_\bullet(2g)$ is generated, at least after tensoring with $\Q$, by $\UMor_{-1}(2g)$ and $\UMor_{-2}(2g)$. This proves that $\cJ_{g,1}$ acts trivially on $\UMor_\bullet(2g)\otimes\Q$, and hence on $\UMor_\bullet(2g)$.
\end{proof}
The element $\Omega_2=2\omega\in\Lambda^2H\subset\UMor_{-2}(2g)$ from Definition \ref{defn:omega} is $\Gamma_{g,1}$-invariant, so the quotient $\UMor_{-2}(2g)/\Z\Omega_2$ is also a $\Gamma_{g,1}$-representation.

\begin{lem}\label{lem:ZkernelUMor2/2omega}
The kernel of the $\Gamma_{g,1}$-action on $\UMor_{-2}(2g)\otimes R/R\Omega_2$ is $\cJ_{g,1}$ for $R=\Z$ and $\Q$.
\end{lem}
\begin{proof}
We assume $R=\Z$; the case $R=\Q$ is identical. The fact that $\cJ_{g,1}$ acts trivially on $\UMor_{-2}(2g)/\Z\Omega_2$ is a consequence of Proposition \ref{prop:ZactiononUMor}.
Conversely, if a class $\phi\in\Gamma_{g,1}$ acts trivially on $\UMor_{-2}(2g)/\Z\Omega_2$, then it also acts trivially on $\UMor_{-2}(2g)/\Lambda^2H\cong H$, hence $\phi$ must lie in $\cT_{g,1}$. Furthermore, $\phi\in \cT_{g,1}$ acts trivially on $\UMor_{-2}(2g)/\Z\Omega_2$ if and only if $\xi_{\tau}(\phi)(a)\in \Z\Omega_2\subset\Lambda^2H$ for all $a\in H$.
Since $\xi_{\tau}=2\tau_{g,1}$ and $\Omega_2=2\omega$, and since multiplication by $2$ is injective on $\Lambda^2H$, the condition on $\phi$ can be rewritten as
$\tau_{g,1}(\phi)\in \Hom(H,\Z\omega)\subset \Hom(H,\Lambda^2H)$.
We will prove that
\[
\Imm(\tau_{g,1})\cap \Hom(H,\Z\Omega_2)=0\subset\Hom(H,\Lambda^2H);
\]
from this the assertion $\ker(\tau_{g,1})=\cJ_{g,1}$ will follow.
    
We use the identification
$\Hom(H,\Lambda^2H)\cong  H\otimes \Lambda^2H$
already considered in the discussion before Theorem \ref{thm:Johnsonimage}, and embed
$\Lambda^2H\hookrightarrow H\otimes H$
via $a\wedge b\mapsto a\otimes b-b\otimes a$; we thus obtain an embedding of $H\otimes\Lambda^2H\hookrightarrow H\otimes H\otimes H$.
Under this embedding, $\Hom(H,\Z\omega)$
is the subgroup of $H^{\otimes 3}$ consisting of the elements
$$
a\otimes \omega:=\sum_{i=1}^g a\otimes x_{2i-1}\otimes x_{2i}-a\otimes x_{2i}\otimes x_{2i-1},
$$
for varying $a\in H$, whereas $\Imm(\tau_{g,1})\cong\Lambda^3H$ is the subgroup of $H^{\otimes 3}$ generated by all elements of the form 
$$
\sum_{\sigma \in \mathfrak{S}_3}\sgn(\sigma) a_{\sigma(1)}\otimes a_{\sigma(2)}\otimes a_{\sigma(3)},
$$
for varying $a_1,a_2,a_3\in H$.
In particular, each element of $\Imm(\tau_{g,1})$ is  invariant under the cyclic permutation of tensors $C\colon H^{\otimes3}\to H^{\otimes3}$ acting as $a_1\otimes a_2\otimes a_3\mapsto a_2\otimes a_3\otimes a_1$.

Let now $a\in H$ be such that $a\otimes\omega\in\Imm(\tau_{g,1})$, and write $a\otimes\omega$ as a linear combination $\sum_{1\le j,k,l\le 2g}c_{j,k,l}x_j\otimes x_k\otimes x_l$ of the basis elements $x_j\otimes x_k\otimes x_l$ of $H^{\otimes 3}$, with $c_{j,k,l}\in\Z$. Using the definition of $a\otimes\omega$ we get $c_{j,k,l}=0$ for all $j,k,l$ such that $\set{k,l}$ is not a pair of the form $\set{2i+1,2i+2}$ for $i\in\{0,1,...,g-1\}$; using invariance of $a\otimes\omega$ under $C$, we obtain that $c_{j,k,l}=0$ for every $j,k,l$ such that at least one of the pairs $\set{j,k}$, $\set{k,l}$ and $\set{l,j}$ is not a pair of the form $\set{2i+1,2i+2}$; this clearly implies $c_{j,k,l}=0$ for all $j,k,l$, so that $a\otimes\omega=0$.
\end{proof}

\begin{prop}\label{thm:Zkernelhomology}
    For $n\ge 2$ and $R=\Z$ and $ \Q$, the kernel of the $\Gamma_{g,1}$-action on $H_{2n-2}(\conf[n]{\Sigma_{g,1}};R)$ is the Johnson kernel $\cJ_{g,1}$.
\end{prop}
\begin{proof}
    Propositions \ref{prop:UMorissubrepandCHbistrivial} and \ref{prop:ZactiononUMor} together imply that $\cJ_{g,1}$ acts trivially on the cellular chains $\redchains{\conf{\cM}}$, thus also on the homology $\widetilde{H}_*(\conf[n]{\cM};\Z)$ for all $n\ge 0$. 
    To prove that the kernel is precisely $\cJ_{g,1}$, it suffices by Lemma \ref{lem:ZkernelUMor2/2omega} to embed $\UMor(2)/\langle 2\omega\rangle $ as a $\Gamma_{g,1}$-subrepresentation of $H_{2n-2}(\conf[n]{\cM};\Z)$. 
    
For $l\ge0$ denote by $1^l$ the $l$-record $(l,(1,\dots,1),\uzero)$ 
and by $1^l2$ the $(l+2)$-record $(l,(1,\dots,1,2),\uzero)$.
By Proposition \ref{prop:differential}, the differential $\del(e_{1^l})$ vanishes for all $l\ge0$; in particular the differential vanishes on the abelian groups
\[
\begin{split}
&\Z e_{1^{n-2}}\otimes \UMor_{-2}(2g)\subset \redchains[n-2]{\conf[n]{\cM}},\\
&\Z e_{1^{n-1}}\otimes\UMor_{-1}(2g)\subset\redchains[n-1]{\conf[n]{\cM}}.
\end{split}
\]
The first vanishing implies that $\Z e_{\bar1}\otimes \UMor_{-2}(2g)$ only consists of
cycles. The second vanishing, together with the direct sum decomposition
\[
 \redchains[n-1]{\conf[n]{\cM}}\cong \Z e_{1^{n-1}}\otimes\UMor_{-1}(2g)\ \oplus\ \Z e_{1^{n-2}2}\otimes\UMor_0(2g)
\]
implies that the boundaries in $\redchains[n-2]{\conf[n]{\cM}}$ are generated by the image of the restriction of 
the differential on $\Z e_{1^{n-2}2}\otimes \UMor_0(2g)$,
which is generated by the element $e_{1^{n-2}2}=e_{1^{n-2}2}\otimes 1$, and again by Proposition \ref{prop:differential} we have
\[
\del(e_{1^{n-2}2})=e_{1^{n-2}}\otimes \Omega_2.
\]
Thus $e_{1^{n-2}}$ gives an injection of $\UMor_{-2}(2g)/\Z \Omega_2$ into $\widetilde{H}_{2n-2}(\conf[n]{\cM};\Z)$, which is $\Gamma_{g,1}$-equivariant.

\end{proof}

\subsection{The \texorpdfstring{$\Gamma_{g,1}$}{Gammag1}-action on \texorpdfstring{$\Fp$}{Fp}-homology}
Let $p$ be an odd prime.
The homology $\widetilde{H}_*(\conf{\cM};\Fp)$ is computed by the complex
$\redchains{\conf{\cM}}\otimes \Fp$ with the induced differential and $\Gamma_{g,1}$-action. The arguments of this subsection are analogous to the ones of Subsection \ref{subsec:Zhomologyaction}, so we emphasise only the differences.

\begin{prop}\label{prop:FpactiononUMor}
The kernel of the $\Gamma_{g,1}$-action on $\UMor_\bullet(2g)\otimes \Fp$ is $\cJ_{g,1}(p)$. 
\end{prop}
\begin{proof}
The first part of the proof of Proposition \ref{prop:ZactiononUMor} can be adapted word by word modulo $p$ to show that $\phi\in\Gamma_{g,1}$ acts trivially on $\UMor_{-2}(2g)\otimes\Fp$ if and only if $\phi\in \ker(\xi^p_{\tau})$: we use that the short exact sequence of $\Gamma_{g,1}$-representations $\Lambda^2H\to\UMor_{-2}(2g)\to H$ is split as a sequence of abelian groups, hence after tensoring with $\Fp$ we obtain a short exact sequence $\Lambda^2H\otimes\Fp\to\UMor_{-2}(2g)\otimes\Fp\to H\otimes\Fp$. This shows that the kernel of the $\Gamma_{g,1}$-action on $\UMor_\bullet(2g)\otimes \Fp$ is contained in $\cJ_{g,1}(p)$.

By Proposition \ref{prop:UMoralgebra},
the ring $\UMor_\bullet(2g)$ is generated (as a $\Z$-algebra) by the elements $x_i$, for $1\le i\le 2g$, and by the divided powers $y_i^{[m]}$, for $1\le i\le2g$ and $m\ge1$; it follows that the elements $x_i\otimes 1$ and $y_i^{[m]}\otimes 1$ generate the $\Fp$-algebra $\UMor_\bullet(2g)\otimes\Fp$.
Therefore, to prove that $\cJ_{g,1}(p)$ acts trivially on $\UMor_\bullet(2g)\otimes\Fp$, it suffices to show that for $\phi \in \cJ_{g,1}(p)=\ker(\xi^p_{\tau})$ we have $\UMor(\phi)(y_i^{[m]})\equiv y_i^{[m]} \mod p$, i.e. the difference $\UMor(\phi)(y_i^{[m]})- y_i^{[m]}$ is a multiple of $p$ in the abelian group $\UMor_{2m}(2g)$. We know that $\UMor(\phi)(y_i)\equiv y_i\mod p$, i.e. there is $a\in \UMor_{-2}(2g)$ such that $\UMor(\phi)(y_i)=y_i+pa$; for $m\ge1$ we then have
\begin{align*}
\UMor(\phi)\left(y_i^{[m]}\right)&= \frac{(y_i+pa)^m}{m!}
\equiv y_i^{[m]} \mod p, 
\end{align*}
where the first equality follows from the fact that $\UMor(\phi)$ is a ring homomorphism and $\UMor_\bullet(2g)$ is torsion-free.
\end{proof}

\begin{lem}\label{lem:FpkernelUMor2/2omega}
    The kernel of the $\Gamma_{g,1}$-actions on $\UMor_{-2}(2g)\otimes \Fp/ \Fp \Omega_2$ is $\cJ_{g,1}(p)$.
\end{lem}
\begin{proof}
The argument is the same as in the proof of Lemma \ref{lem:ZkernelUMor2/2omega}. By Proposition \ref{prop:FpactiononUMor} we know that $\cJ_{g,1}(p)$ acts trivially on $\UMor_{-2}(2g)\otimes \Fp$, hence on its quotient $\UMor_{-2}(2g)\otimes \Fp/ \Fp \Omega_2$.

Conversely, if $\phi\in\Gamma_{g,1}$ acts trivially on $\UMor_{-2}(2g)\otimes \Fp/ \Fp \Omega_2$ then it also acts trivially on $\UMor_{-2}(2g)\otimes \Fp/\Lambda^2H\otimes\Fp\cong H\otimes\Fp$, hence $\phi\in\cT_{g,1}(p)$; the condition that $\phi$ acts trivially on $\UMor_{-2}(2g)\otimes \Fp/ \Fp \Omega_2$ can then be rephrased as
\[
\xi^p_{\tau}(\phi)\in\Hom(H,\Z\Omega_2)\otimes\Fp=\Hom(H,\Z\omega)\otimes\Fp\subset\Hom(H,\Lambda^2H)\otimes\Fp,
\]
where in the last step we use that $p$ is odd, and that $\Z\omega\subset\Lambda^2H$ is a split subgroup.

We can now identify $\Hom(H,\Lambda^2H)\cong H\otimes\Lambda^2H\otimes\Fp$ with a subgroup of $H^{\otimes 3}\otimes\Fp$ as in the proof of Lemma \ref{lem:ZkernelUMor2/2omega}; by Proposition \ref{prop:Fpimageofxi}, the image of $\xi^p_{\tau}$ can then be identified with $\Lambda^3H\otimes\Fp$, and all we have to prove is that the subgroups $\Lambda^3H\otimes\Fp$ and $H\otimes\Z\omega\otimes\Fp$ of $H^{\otimes3}\otimes\Fp$ intersect trivially: this can be done by repeating the final argument used in the proof of Lemma \ref{lem:ZkernelUMor2/2omega}.
\end{proof}

\begin{prop}\label{thm:Fpkernelhomology}
For $n\ge 2$, the kernel of the $\Gamma_{g,1}$-action on the homology group $H_{2n-2}(C_n(\Sigma_{g,1})^{\infty};\Fp)$  is the mod-$p$ Johnson kernel $\cJ_{g,1}(p)$.
\end{prop}
\begin{proof}
By Proposition \ref{prop:FpactiononUMor}, the group $\cJ_{g,1}(p)$ acts trivially on the $\Gamma_{g,1}$-equivariant chain complex $\redchains{\conf{\cM},\Fp}\cong\widetilde{\Ch}^{B}(\conf{\cM})\otimes(\UMor_\bullet(2g)\otimes\Fp)$, hence on its homology. To show that the kernel of the $\Gamma_{g,1}$-action on $H_{2n-2}(C_n(\Sigma_{g,1})^{\infty};\Fp)$ is no larger than $\cJ_{g,1}(p)$, for a fixed $n\ge2$, it suffices by Lemma \ref{lem:FpkernelUMor2/2omega} to embed a copy of $\UMor_{-2}(2g)\otimes\Fp/\Fp\Omega_2$
inside $H_{2n-2}(C_n(\Sigma_{g,1})^{\infty};\Fp)$, and this can be done by the same argument used in the proof of Proposition \ref{thm:Zkernelhomology}, tensoring with $\Fp$.
\end{proof}

\begin{proof}[Proof of Theorem \ref{thm:A}]
    Propositions \ref{prop:UMorissubrepandCHbistrivial} and \ref{prop:ZactiononUMor} together imply that $\cJ_{g,1}$ acts trivially on the cellular chains $\redchains{\conf{\cM}}$ which, after tensoring with $R$ and dualising, computes $\widetilde{H}^*(\conf{\Sigma_{g,1}};R)$. It follows that $\cJ_{g,1}$ acts trivially on this cohomology and, by Proposition \ref{prop:equivariantPD}, the same holds for $H_*(C_n(\Sigma_{g,1});R)$.

    In particular, we have an isomorphism $\widetilde{H}_{2n-2}(\conf[n]{\Sigma_{g,1}};\Q)\cong H^{2}(C_n(\Sigma_{g,1});\Q)$ of $\Gamma_{g,1}$-representations. It follows by Proposition \ref{thm:Zkernelhomology} that the kernel of the $\Gamma_{g,1}$-action on $H^{2}(C_n(\Sigma_{g,1});\Q)$, and thus also on its linear dual $H_{2}(C_n(\Sigma_{g,1});\Q)$, is $\cJ_{g,1}$ for $n\ge 2$. Since $H_{2}(C_n(\Sigma_{g,1});\Q)\cong H_{2}(C_n(\Sigma_{g,1});\Z) \otimes \Q$, then $H_{2}(C_n(\Sigma_{g,1});\Z)$ has kernel at least $\cJ_{g,1}$; by (1) the kernel is precisely $\cJ_{g,1}$, and (2) follows.
    
    The case $R=\Fp$ follows analogously using Propositions \ref{prop:FpactiononUMor} and \ref{thm:Fpkernelhomology}.
\end{proof}

\part{Homology computation up to an Ext problem}\label{part:2}
In this part of the article we prove Theorem \ref{thm:B} and its corollaries. We fix $g\ge0$ throughout the part and denote by $\cM$ the model of $\Sigma_{g,1}$ from Definition \ref{defn:cM}.
\section{Proof of Theorem \ref{thm:B} and its corollaries}
\subsection{\texorpdfstring{$\redchains{\conf{\cM}}$}{redchains(conf(M))} as a bar complex}
\label{sec:proofthmB}
The starting point towards the proof of Theorem \ref{thm:B} is a reinterpretation of $\redchains{\conf{\cM}}$. Recall from Definition \ref{defn:zetag} the element $\zeta\in\pi_1(V_{2g},*)\cong\Z^{*2g}$, and the corresponding pointed map $f_{\zeta_g}\colon(V_1,*)\to(V_{2g},*)$ from the proof of Proposition \ref{prop:differential}.
The ring homomorphism
\[
 \UMor_\bullet(f_{\zeta_g})\colon \UMor_\bullet(1)\to\UMor_\bullet(2g)
\]
makes $\UMor_\bullet(2g)$ into a weighted $\UMor_\bullet(1)$-left module. Following Definition \ref{defn:UMor}, observe that $\UMor_\bullet(2g)$ and $\UMor_\bullet(1)$ are concentrated in non-positive weights.

Recall from Proposition \ref{prop:UMoralgebra} the isomorphisms $\UMor_\bullet(1)\cong\Lambda_\Z(x)\otimes\Gamma_\Z(y)$ and
$\UMor_\bullet(2g)\cong \Lambda_\Z(x_1,\dots,x_{2g})\otimes\Gamma_\Z(y_1,\dots,y_{2g})$. By Theorem \ref{thm:UMormaps} we have that
$\UMor_\bullet(f_{\zeta_g})$ sends $x\mapsto0$ and $y^{[m]}\mapsto\Omega_{2m}$, for $m\ge1$ (see Definition \ref{defn:omega}).
We regard $\Z$ as a $\UMor_\bullet(1)$-$\UMor_\bullet(1)$-bimodule concentrated in weight $0$, via the ring homomorphism $\varepsilon\colon\UMor_\bullet(1)\to\Z$ sending $x,y^{[m]}\mapsto0$.
We can therefore consider the reduced bar complex
\[
 \reB_\star(\Z,\UMor_\bullet(1),\UMor_\bullet(2g)),
\]
i.e. the weighted chain complex having $\Z\otimes\UMor_\bullet(1)_+^{\otimes b}\otimes\UMor_\bullet(2g)$ in bar-degree $\star=b\ge0$, and vanishing in negative bar-degrees. Here $\UMor_\bullet(1)_+:=\ker(\varepsilon)$ denotes the augmentation ideal.
For $b\ge 1$ the differential
\[
\del\colon \reB_b(\Z,\UMor_\bullet(1),\UMor_\bullet(2g))\to\reB_{b-1}(\Z,\UMor_\bullet(1),\UMor_\bullet(2g)),
\]
is given by the alternating sum $\del=\sum_{i=0}^b(-1)^i\del_i$, with $\del_i=\id^{\otimes i}\otimes\mu\otimes\id^{\otimes b-i}$.
Here we denote for simplicity all multiplication maps
by $\mu$, and all identities by $\id$.

\begin{prop}
\label{prop:redchainsasbarcomplex}
There is an isomorphism of weighted chain complexes
\[
  \reB_\star(\Z,\UMor_\bullet(1),\UMor_\bullet(2g))\cong\redchains{\conf{\cM}}.
\]
\end{prop}
\begin{proof}
As bigraded abelian group $\redchains{\conf{\cM}}$ is freely generated by the elements $e_\ft$ for varying record $\ft$ of the form $(b,\uP,\fv)$, where $b\ge0$, $\uP=(P_1,\dots,P_b)$ with $P_i\ge1$, and $\fv=(v_1,\dots,v_{2g})$ with $v_j\ge0$.
Similarly, 
$\reB_\star(\Z,\UMor_\bullet(1),\UMor_\bullet(2g))$ is freely generated by the elements $1\otimes a_1\otimes\dots\otimes a_b\otimes c$, where $b\ge0$, each $a_i$ has the form $x^{\epsilon_i}y^{[m_i]}$ and is different from $1$, and $c$ has the form $\prod_{j=1}^{2g}x_j^{\epsilon'_j}y_j^{[m'_j]}$, for some $\epsilon_i,\epsilon'_j\in\set{0,1}$ and $m_i,m'_j\ge0$.

Given $\ft=(b,\uP,\fv)$, there are unique decompositions $P_i=2m_i+\epsilon_i$ and $v_j=2m'_j+\epsilon_j'$ with $\epsilon_i,\epsilon'_j,m_i,m'_j$ as above; setting $a_i=x^{\epsilon_i}y^{[m_i]}$ and  $c=\prod_{j=1}^{2g}x_j^{\epsilon'_j}y_j^{[m'_j]}$, we obtain a basis element $1\otimes a_1\otimes\dots\otimes a_b\otimes c$ of $\reB_\star(\Z,\UMor_\bullet(1),\UMor_\bullet(2g))$ in the same bigrading $(\bullet,\star)=(-\sum_{i=1}^b P_i-\sum_{j=1}^{2g} v_j,b)$ as $e_\ft$. Thus the assignment $e_\ft\mapsto 1\otimes a_1\otimes\dots\otimes a_b\otimes c$
gives a bigraded bijection of abelian groups
\[
\reB_\star(\Z,\UMor_\bullet(1),\UMor_\bullet(2g))\cong\redchains{\conf{\cM}}.
\]
Finally, Proposition \ref{prop:differential} ensures that the above is a chain map.
\end{proof}

The two chain complexes from Proposition \ref{prop:redchainsasbarcomplex} are concentrated in non-negative bar-degree and non-positive weight, and are of finite rank in each bidegree.

\begin{nota}
For a commutative ring $R$ we denote by $\UMor_\bullet^R(k)$ the weighted $R$-algebra $\UMor_\bullet(k)\otimes R$.
The maps $\UMor(f_{\zeta_g})\otimes R$ and $\epsilon\otimes R$ make
$\UMor_\bullet^R(2g)$ into a left $\UMor_\bullet^R(1)$-module,
and $R$ into a $\UMor_\bullet^R(1)$-$\UMor_\bullet^R(1)$-bimodule.

\end{nota}

\begin{prop}
\label{prop:isoextUMor}
Let $R$ be a commutative ring; then there is an isomorphism of bigraded $R$-modules
\[
  \Ext^\star_{\UMor^R_\bullet(1)}(\UMor^R_\bullet(2g);R)\cong H_{\bullet+\star}(C_\bullet(\sgone);R).
\]
\end{prop}
\begin{proof}
For $R=\Z$, taking the bigraded dual of the chain complexes from Proposition \ref{prop:redchainsasbarcomplex}, we obtain an isomorphism of cochain complexes
\begin{equation}\label{eq:cochainiso}
\Hom_{\Z}\pa{\reB_\star(\Z,\UMor_\bullet(1),\UMor_\bullet(2g));\Z}\cong\redcochains{\conf{\cM}},
\end{equation}
where $\redcochains{\conf{\cM}}$ is defined as $\Hom_\Z(\redchains{\conf{\cM}},\Z)$. Here $\Z$ is concentrated in bidegree $(\bullet,\star)=(0,0)$, and hence both of the previous cochain complexes are concentrated in non-positive bar-degree and non-negative weight.

On the one hand we identify the cohomology of the first cochain complex
\[
\begin{split}
\Hom_{\Z}\pa{\reB_\star(\Z,\UMor_\bullet(1),\UMor_\bullet(2g)),\Z}\cong\\
\Hom_{\UMor_\bullet(1)}\pa{\reB_\star(\UMor_\bullet(1),\UMor_\bullet(1),\UMor_\bullet(2g)),\Z}
\end{split}
\]
with $\Ext^\star_{\UMor_\bullet(1)}(\UMor_\bullet(2g)),\Z)$.
On the other hand, by Proposition \ref{prop:equivariantPD} the cohomology of $\redcochains{\conf{\cM}}$ is isomorphic to $H_{\bullet+\star}(C_\bullet(\sgone);\Z)$. The convention on bigradings from Notation \ref{nota:bulletstarofft} ensures that the isomorphism is bigraded.

For generic $R$, we tensor the isomorphism \eqref{eq:cochainiso} with $R$: the left hand side becomes a cochain complex computing  $\Ext^\star_{\UMor^R_\bullet(1)}(\UMor^R_\bullet(2g);R)$, and the cohomology of the right hand side is isomorphic by Proposition \ref{prop:equivariantPD} to $H_{\bullet+\star}(C_\bullet(\sgone);R)$.
\end{proof}

\subsection{Yoneda product}
In the case $g=0$, we have $\UMor^R_\bullet(2g)\cong R$, so that Proposition \ref{prop:redchainsasbarcomplex}, up to identifying $\Sigma_{0,1}$ with the disc $D:=[0,1]^2$, gives us an isomorphism of bigraded $R$-modules
\begin{equation}\label{eq:isoring}
\Ext^\star_{\UMor^R_\bullet(1)}(R,R)\cong H_{\bullet+\star}(C_\bullet(D);R).
\end{equation}
Both sides admit a structure of bigraded $R$-algebra: the left hand side has a Yoneda product, making it into an associative ring; the right hand side has a Pontryagin product, coming from the $E_2$-algebra structure on $C_\bullet(\mathring{D})$, and it is thus a graded-commutative ring. Moreover, for any $g\ge0$, each of the two terms in the isomorphism of Proposition \ref{prop:isoextUMor} is naturally a bigraded left module over one of these rings: $\Ext^\star_{\UMor^R_\bullet(1)}(\UMor^R_\bullet(2g),R)$ is a left module over $\Ext^\star_{\UMor^R_\bullet(1)}(R,R)$ via the Yoneda product, whereas the $E_1$-module structure of $C_\bullet(\mathring{\Sigma}_{g,1})$ over the $E_2$-algebra $C_\bullet(\mathring{D})$ makes $H_*(C_\bullet(\mathring{\Sigma}_{g,1});R)$ into a left module over $H_*(C_\bullet(\mathring{D});R)$. The main result of this subsection is the following proposition.
\begin{prop}
\label{prop:isoYonedamodule}
 Isomorphism \eqref{eq:isoring} is an isomorphism of bigraded rings; after identifying these two rings, the isomorphism from Proposition \ref{prop:isoextUMor} is an isomorphism of bigraded (left) modules.
\end{prop}
\begin{defn}\label{defn:ttau}
Recall Definition \ref{defn:cM}. We let $\ttau\colon D\sqcup\cM\to\cM$ be the map sending
\[
 \begin{array}{cl}
(x,y)\mapsto\fp(x,y) &\mbox{for }(x,y)\in D,\\
\fp(x,y)\mapsto\fp(1+x/2,y)&\mbox{for }(x,y)\in\bfR.  
 \end{array}
\]

\end{defn}
The map $\ttau$ can be used to describe conveniently the $E_1$-algebra and module structures of $C_\bullet(\mathring{D})$ and $C_\bullet(\mathring{\cM})$, as we discuss in the following.
Since $\ttau$ restricts to an open embedding $\mathring{D}\sqcup\mathring{\cM}\hookrightarrow\mathring{\cM}$, it gives rise to a map
\[
 C_\bullet(\ttau)\colon C_\bullet(\mathring{D})\times C_\bullet(\mathring{\cM})\to C_{\bullet}(\mathring{\cM});
\]
the map $C_\bullet(\ttau)$ is an open embedding and for $n\ge0$ it restricts to an open embedding
\[
C_n(\ttau)\colon \coprod_{m=0}^n C_m(\mathring{D})\times C_{n-m}(\mathring{\cM})\hookrightarrow C_n(\mathring{\cM}).
\]
For $g=0$, the map $C_\bullet(\ttau)\colon C_\bullet(\mathring{D})\times C_\bullet(\mathring{D})\to C_\bullet(\mathring{D})$ is part of the $E_2$-algebra structure on $C_\bullet(\mathring{D})$, and
for $g\ge0$ the map $C_\bullet(\ttau)$ is part of the $E_1$-module structure of $C_{\bullet}(\mathring{\cM})$ over $C_\bullet(\mathring{D})$.
Taking $R$-homology, we obtain that $C_\bullet(\ttau)$ induces in the two cases the Pontryagin ring multiplication $H_*(C_\bullet(\mathring{D});R)$ and the Pontryagin module multiplication $H_*(C_\bullet(\mathring{D});R)\otimes_R H_*(C_\bullet(\mathring{\cM});R)\to H_*(C_\bullet(\mathring{\cM});R)$.

For $n\ge0$, the open embedding $C_n(\ttau)$ induces a collapse map
\[
 \conf[n]{\ttau}\colon \conf[n]{\cM}\to\bigvee_{m=0}^n \conf[m]{D}\wedge\conf[n-m]{\cM}.
\]
For $0\le\bar m\le n$ let
$\pi_{\bar m,n}\colon\bigvee_{m=0}^n \conf[m]{D}\wedge\conf[n-m]{\cM}\to \conf[{\bar m}]{D}\wedge\conf[{n-\bar m}]{\cM}$
be the map collapsing all wedge summands except the $\bar m$\sth.
We can now take reduced cohomology and obtain a description of the Pontryagin module multiplication as the vertical composite on the right of the following diagram: all homology and cohomology groups are with understood coefficients in $R$, tensor products are over $R$, and $-\times-$ denotes the cross product in $R$-homology and cohomology; the horizontal isomorphisms are given by Proposition \ref{prop:equivariantPD}:
\[
 \begin{tikzcd}[column sep=6pt]
  H_i(C_m(\mathring{D}))\otimes H_j(C_{n-m}(\mathring{\cM}))\ar[d,"-\times-"]\ar[r,phantom,"\cong"] &
  \tilde H^{2m-i}(\conf[m]{D})\otimes \tilde H^{2n-2m-j}(\conf[n-m]{\cM})\ar[d,"-\times-"]\\
  H_{i+j}(C_m(\mathring{D})\times C_{n-m}(\mathring{\cM}))\ar[r,phantom,"\cong"]\ar[d,"C_{n}(\ttau)_*"] &\tilde H^{2n-i-j}(\conf[m]{D}\wedge\conf[n-m]{\cM})\ar[d,"\pa{\pi_{m,n}\circ\conf[n]{\ttau}}^*"]\\
  H_{i+j}(C_n(\mathring{\cM}))\ar[r,phantom,"\cong"]
  &
  \tilde H^{2n-i-j}(\conf[n]{\cM}).
 \end{tikzcd}
\]
For $0\le m\le n$ we consider on $\conf[m]{D}\wedge\conf[n-m]{\cM}$ the smash product cell decomposition, and on $\bigvee_{m=0}^n \conf[m]{D}\wedge\conf[n-m]{\cM}$ the wedge cell decomposition.
\begin{lem}
\label{lem:confntaucellular}
Let $n\ge0$; then the map $\conf[n]{\ttau}$ is cellular. For an $n$-record $\ft=(b,\uP,\fv)$ as in Definition \ref{defn:record}, the induced map on reduced cellular chain complexes
\[
\conf[n]{\ttau}_*\colon\redchains{\conf[n]{\cM}}\to\bigoplus_{m=0}^n\redchains{\conf[m]{D}}\otimes\redchains{\conf[n-m]{\cM}}
\]
sends the generator $e_\ft$ to the sum $\sum_{i=0}^b \epsilon_i e_{\ft_i^l}\otimes e_{\ft_i^r}$, where $\ft_i^l=(i,(P_1,\dots,P_i))$, $\ft_i^r=(b-i,(P_{i+1},\dots,P_b),\fv)$, and $\epsilon_i=(-1)^{(b-i)\sum_{j=0}^iP_j}\in\set{\pm1}$.
\end{lem}
\begin{proof}
Observe first that $\conf[n]{\ttau}(s)=\infty$. Let $\infty\neq s\in\conf[n]{\cM}$ be a configuration; then $s$ is in the image of the map $\Phi^\ft\colon\mathring{\Delta}^\ft\hookrightarrow\conf[n]{\cM}$, for an $n$-record $\ft$ as above. Let $(x_1,\dots,x_b)\in\mathring\Delta^b$ be the image of $(\Phi^\ft)^{-1}(s)\in\mathring{\Delta}^\ft$ along the projection $\Delta^\ft\to\Delta^b$. If $x_i=\frac 12$ for some $1\le i\le b$, then $\conf[n]{\ttau}(s)=\infty$. Otherwise there is $0\le i\le b$ such that $x_i<\frac 12$ and either $i=b$ or $x_{i+1}>\frac 12$; then $\conf[n]{\ttau}(s)$ lies in the cell $e_{\ft_i^l}\times e_{\ft_i^r}$, whose dimension is $d(\ft_i^l)+d(\ft_i^r)=d(\ft)$. Hence $\conf[n]{\ttau}$ is cellular.

The open embedding $C_n(\ttau)$ restricts to an open embedding of $d(\ft)$-dimensional manifolds $\coprod_{i=0}^b e_{\ft_i^l}\times e_{\ft_i^r}\hookrightarrow e_{\ft}$, so for $0\le i\le b$ the degree $[\conf[n]{\ttau}_*e_{\ft}:e_{\ft_i^l}\times e_{\ft_i^r}]$ is $\pm1$.
To determine the sign, we parametrise the cells $e_{\ft_i^l}$, $e_{\ft_i^r}$ and $e_{\ft}$ by the open polysimplices $\mathring{\Delta}^{\ft_i^l}$, $\mathring{\Delta}^{\ft_i^r}$ and $\mathring{\Delta}^\ft$: then we can identify the restricted embedding $C_n(\ttau)\colon e_{\ft_i^l}\times e_{\ft_i^r}\to e_{\ft}$ with the composition
\[
\begin{split}
  \mathring{\Delta}^{\ft_i^l}\times\mathring{\Delta}^{\ft_i^r}&=\mathring{\Delta}^{i}\times\prod_{j=1}^i\mathring{\Delta}^{P_j}\times\mathring{\Delta}^{n-i}\times\prod_{j=i+1}^b\mathring{\Delta}^{P_j}\times\prod_{j=1}^{2g}\mathring{\Delta}^{v_j}\\
  &\cong \mathring{\Delta}^{i}\times\mathring{\Delta}^{b-i}\times\prod_{j=1}^b\mathring{\Delta}^{P_j}\times\prod_{j=1}^{2g}\mathring{\Delta}^{v_j}\\
  &\overset{\psi_i\times\id}{\longrightarrow}\mathring{\Delta}^{b}\times\prod_{j=1}^b\mathring{\Delta}^{P_j}\times\prod_{j=1}^{2g}\mathring{\Delta}^{v_j}=\mathring{\Delta}^\ft,
\end{split}
\]
where the first homeomorphism comes from rearranging the factors and has degree $\epsilon_i$, and the map $\psi_i\colon\mathring{\Delta}^i\times\mathring{\Delta}^{b-i}\to\mathring{\Delta}^b$, defined by
\[
((x_1,\dots,x_i),(x_{i+1},\dots,x_b))\mapsto(\frac 12 x_1,\dots,\frac 12 x_i,\frac{1+x_{i+1}}2,\dots,\frac{1+x_b}2),
\]
is an orientation-preserving open embedding.
\end{proof}
By Lemma \ref{lem:confntaucellular} the map induced by $\conf{\ttau}$ on reduced cellular cochains is the following: given records $\ft_1=(i,(P_1,\dots,P_i))$ and $\ft_2=(b-i,(P_{i+1},\dots,P_b),\fv)$ representing cells $e_{\ft_1}\subset\conf[m]{D}$ and $e_{\ft_2}\subset\conf[n-m]{\cM}$, we can consider the dual generators
\[
e_{\ft_1}^\vee\in\redcochains{\conf[m]{D}}\ \mbox{ and }\ e_{\ft_2}^\vee\in\redcochains{\conf[n-m]{\cM}}
\]
and their tensor product
\[
e_{\ft_1}^\vee\otimes e_{\ft_2}^\vee:=(e_{\ft_1}\times e_{\ft_2})^\vee\in\redcochains{\conf[m]{D}\wedge\conf[n-m]{\cM}};
\]
then $\pa{\conf{\ttau}}^*$ sends $e_{\ft_1}^\vee\otimes e_{\ft_2}^\vee\mapsto\epsilon e_{\ft}^\vee$, where
\[
\ft=(b,(P_1,\dots,P_b),\fv)\ \mbox{ and }\ \epsilon=(-1)^{(b-i)\sum_{j=1}^iP_j}.
\]

\begin{proof}[Proof of Proposition \ref{prop:isoYonedamodule}]
 The isomorphism given in the proof of Proposition \ref{prop:redchainsasbarcomplex} can be dualised to an isomorphism of cochain complexes
 \[
  \Hom_R\pa{\reB_\star(R,\UMor^R_\bullet(1),\UMor^R_\bullet(2g),R}\cong\redcochains{\conf{\cM};R};
 \]
under this isomorphism, for a record $\ft=(b,\uP,\fv)$, the dual generator $e_{\ft}^\vee$ on the right-hand side corresponds to the dual generator $(1\otimes a_1\otimes\dots\otimes a_b\otimes c)^\vee$, where the monomial $a_i\in\UMor^R_\bullet(1)$ depends on $P_i$, and the monomial $c\in \UMor^R_\bullet(2g)$ depends on $\fv$.
The Yoneda product is induced by the cochain map
\[
\begin{tikzcd}
  \Hom_R\pa{\reB_\star(R,\UMor^R_\bullet(1),R),R}\otimes \Hom_R\pa{\reB_\star(R,\UMor^R_\bullet(1),\UMor^R_\bullet(2g)),R}\ar[d,"Y"]\\
  \Hom_R\pa{\reB_\star(R,\UMor^R_\bullet(1),\UMor^R_\bullet(2g)),R},\\
(1\!\otimes a_1\otimes\!\dots\!\otimes a_i\otimes 1)^\vee\!\otimes\! (1\!\otimes a_{i+1}\otimes\!\dots\!\otimes a_b\otimes c)^\vee\!\mapsto
(-1)^\epsilon(1\!\otimes a_1\otimes\!\dots\!\otimes a_b\otimes c)^\vee,
\end{tikzcd}
\]
where $\epsilon=(-1)^{(b-i)\sum_{j=1}^i P_j}$ is a sign that makes $Y$ into a cochain map: here it is important to recall that Koszul signs are computed with respect to the \emph{homological degree}, which is the sum of the weight and the bar-degree. Combining this observation with the discussion after Lemma \ref{lem:confntaucellular}, we obtain the desired statement.
\end{proof}

\subsection{The K\"unneth isomorphism}
By Proposition \ref{prop:UMoralgebra}, the weighted $R$-algebra $\UMor_\bullet^R(1)$ is isomorphic to the tensor product over $R$ of the algebras $\Lambda_R(x)$ and $\Gamma_R(y)$; we can also consider $R$ as an $R$-algebra concentrated in weight 0, and decompose $\UMor_\bullet^R(1)$ as the tensor product of the algebras $\Lambda_R(x)$, $\Gamma_R(y)$, and $R$.
Similarly, the $\UMor_\bullet^R(1)$-module $R$ is isomorphic to the tensor product of the $\Lambda_R(x)$-module $R$, the $\Gamma_R(y)$-module $R$, and the $R$-module $R$. 
Finally, also the $\UMor_\bullet^R(1)$-module $\UMor_\bullet^R(2g)$ is isomorphic to the tensor product of the $\Lambda_R(x)$-module $R$, the $\Gamma_R(y)$-module $\Lambda_R(x_1,\dots,x_{2g})$, and the $R$-module $\Gamma_R(y_1,\dots,y_{2g})$: indeed by Theorem \ref{thm:UMormaps} the ring homomorphism $\UMor_\bullet(f_{\zeta_g})$ sends $x\mapsto0$ and sends $y^{[m]}\mapsto\Omega_{2m}\in\Lambda_R(x_1,\dots,x_{2g})$ (see Definition \ref{defn:omega}).

\begin{defn}
 \label{defn:bMg}
 For $g\ge0$ we denote by $\bM_g^R$ the $\Gamma_R(y)$-module $\Lambda_R(x_1,\dots,x_{2g})$, with module structure induced by the homomorphism of $R$-algebras $\Gamma_R(y)\to\Lambda_R(x_1,\dots,x_{2g})$ sending $y^{[m]}\mapsto\Omega_{2m}$.
\end{defn}

Since all algebras and relative modules are free and of finite type when considered as weighted $R$-modules, (i.e. in each weight they coincide with a finitely generated free $R$-module), and since the module $R\cong R\otimes_RR\otimes_RR$ is supported in a single weight, we obtain a K\"unneth decomposition of $\Ext^\star_{\UMor^R_\bullet(1)}(\UMor^R_\bullet(2g),R)$ as
\[
\Ext^\star_{\Lambda_R(x)}(R,R)\otimes_R\Ext^\star_{\Gamma_R(y)}(\bM_g^R,R)\otimes_R\Ext^\star_R(\Gamma_R(y_1,\dots,y_{2g}),R).
\]

\begin{proof}[Proof of Theorem \ref{thm:B}]
The K\"unneth isomorphism is compatible with the Yoneda product; in particular the $\Ext^\star_{\UMor^R_\bullet(1)}(R,R)$-module $\Ext^\star_{\UMor^R_\bullet(1)}(\UMor^R_\bullet(2g),R)$ is the tensor product of $\Ext^\star_{\Lambda_R(x)}(R,R)$ as a module over itself, the $\Ext^\star_{\Gamma_R(y)}(R,R)$-module
$\Ext^\star_{\Gamma_R(y)}(\bM_g^R,R)$, and the $R$-module $\Ext^\star_R(\Gamma_R(y_1,\dots,y_{2g}),R)$.

We compute explicitly all factors in the previous decomposition and use the grading convention from Remark \ref{rem:Extnegative}. The ring isomorphism $\Ext^\star_{\Lambda_R(x)}(R,R)\cong R[\epsilon]$ is standard, so we just comment on bigradings. The class $\epsilon\in\Ext^{-1}_{\Lambda_R(x)}(R,R)$, represented by $(1\otimes x\otimes 1)^\vee$, has bar-degree $\star=-1$. Since $x$ has weight $-1$, $\epsilon$ has weight $\bullet=-(-1)=1$. Hence the homological degree of $\epsilon$ is $*=\bullet+\star=1-1=0$.

The second factor $\Ext^\star_{\Gamma_R(y)}(\bM_g^R,R)$ is the same presented in the statement of Theorem \ref{thm:B}, and we will analyse it more closely later; for the moment, notice that for $g=0$ we have $\bM_g^R=\Lambda_R()\cong R$, as in (1) of Theorem \ref{thm:B}.

The third factor $\Ext^\star_R(\Gamma_R(y_1,\dots,y_{2g}),R)$ is concentrated in bar-degree $\star=0$, as $\Gamma_R(y_1,\dots,y_{2g})$ is free, hence projective over $R$. We are thus interested in computing
\[
\Ext^0_R(\Gamma_R(y_1,\dots,y_{2g}),R)=\Hom_R(\Gamma_R(y_1,\dots,y_{2g}),R).
\]
Let $\cH_g^R$ denote the $R$-linear dual of $\bigoplus_{i=1}^{2g}Ry_i$; since each $y_i$ has weight $-2$, $\cH_g^R$ is concentrated in weight $\bullet=2$; we put it in bar-degree $\star=0$. The free \emph{divided power $R$-algebra} $\Gamma_R(y_1,\dots,y_g)$ has a natural coproduct given on generators by $\Delta(y_i^{[m]})=\sum_{j=1}^my_i^{[j]}\otimes y_i^{[m-j]}$, and making $\Gamma_R(y_1,\dots,y_g)$ into a $R$-Hopf algebra. The $R$-linear dual
$\Hom_R(\Gamma_R(y_1,\dots,y_g),R)$ is also an $R$-Hopf algebra, and as an algebra it agrees with the free \emph{symmetric $R$-algebra} on $\cH_g^R$. We thus identify $\Ext^\star_R(\Gamma_R(y_1,\dots,y_{2g}),R)$ with $R[\cH_g]$ as $R$-modules.
\end{proof}

\begin{rem}
 The fact that the $R$-module $R[\cH_g]$ has a natural algebra structure suggests that there should be an action of $R[\cH_g]$ on $H_*(C_\bullet(\sgone);R)$. Such an action
 is sketched in \cite[Subsection 1.1]{KMT}.
\end{rem}

\subsection{Direct sum splitting of \texorpdfstring{$\bM_g^R$}{bMgR}}
In this subsection we decompose the $\Gamma_R(y)$-module $\bM_g^R$ as a direct sum of smaller modules, thus reducing the computation of the term $\Ext^\star_{\Gamma_R(y)}(\bM_g^R,R)$ from Theorem \ref{thm:B} to smaller computations.
\begin{defn}\label{defn:twoB}
Let $u\ge0$. We consider $B^{\Z}_u:=\Z[z_1,\dots,z_u]/(z_1^2,\dots,z_u^2)$ as a weighted ring, by putting each $z_i$ in weight $\bullet=-2$. We let $\phi_1\colon \Gamma_\Z(y)\to B^{\Z}_u$ be the unique ring homomorphism sending
$y\mapsto\sum_{i=1}^u z_i$; we have $\phi_1\colon y^{[m]}\mapsto\sum_{1\le i_1<\dots<i_m\le u}z_{i_1}\dots z_{i_m}$.
This makes $B^\Z_u$ into a $\Gamma_\Z(y)$-algebra, and in particular into a $\Gamma_\Z(y)$-module.

We will also consider the unique ring homomorphism $\phi_2\colon\Gamma_\Z(y)\to B^\Z_u$, sending $\phi_2\colon y\mapsto 2\sum_{i=1}^u z_i$; we have $\phi_2(y^{[m]})= 2^m\phi_1(y^{[m]})$. We denote by $\twoB^\Z_u$ the corresponding $\Gamma_\Z(y)$-module structure on the weighted abelian group $B^\Z_u$, to distinguish it from the module structure induced by $\phi_1$.

For a generic ring $R$ we let $B^R_u$ and $\twoB^R_u$ be the corresponding $\Gamma_R(y)$-modules, obtained from the previous by tensoring with $R$.
\end{defn}

\begin{nota}\label{nota:K}
 Let $\bK=\set{\bd,\bu,\be}$ be the set with three elements $\bd,\bu,\be$; for a function $\xi\colon \set{1,\dots,g}\to\bK$ we denote by $\xi^{-1}(\be)\subseteq\set{1,\dots,g}$ the preimage of $\be$ and we let $0\le|\xi|\le g$ be the cardinality of $\xi^{-1}(\be)$.
\end{nota}

\begin{prop}
\label{prop:bMgsplitting}
 For $g\ge0$, we have an isomorphism of weighted $\Gamma_R(y)$-modules
 \[
  \bM^R_g\cong\bigoplus_{\mathclap{\xi\colon\set{1,\dots,g}\to\bK}}\twoB^R_{|\xi|}\ [|\xi|-g],
 \]
 where $\twoB^R_r[-m]$ is the weighted $\Gamma_R(y)$-module $\twoB^R_r$ shifted down by $m$.
\end{prop}
\begin{proof}
The ring homomorphism $\Gamma_R(y)\to\Lambda_R(x_1,\dots,x_{2g})$ sending $y^{[m]}\mapsto\Omega_{2m}$ (and making
$\Lambda_R(x_1,\dots,x_{2g})$ into a $\Gamma_R(y)$-module) can be written as the composite of the ring homomorphism $\phi_2\colon\Gamma_R(y)\to B^R_g$ (see Definition \ref{defn:twoB}) and the ring homomorphism $B^R_g\to\Lambda_R(x_1,\dots,x_{2g})$ sending $z_i\mapsto x_{2i-1}\wedge x_{2i}$. Thus we can lift the $\Gamma_R(y)$-module structure on $\bM^R_g$ to a $B^R_g$-module structure.

For $\xi\colon \set{1,\dots,g}\to\bK$ and $1\le i\le g$ let $x_{\xi,i}\in\Lambda_R(x_1,\dots,x_{2g})$ be:
\[
x_{\xi,i}:=\left\{
\begin{array}{cl}
 x_{2i-1} & \mbox{ if } \xi(i)=\bd;\\[.2em]
 x_{2i} & \mbox{ if } \xi(i)=\bu;\\[.2em]
 1 & \mbox{ if }\xi(i)=\be.
\end{array}
\right.
\]

Let $x_\xi=x_{\xi,1}\wedge\dots\wedge x_{\xi,g}\in\Lambda^{g-|\xi|}(x_1,\dots,x_{2g})$, and note that $x_\xi$ has weight $|\xi|-g$. For $\emptyset\subseteq S\subseteq\set{1,\dots,g}$ let $z_S=\prod_{i\in S}z_i\in B^R_g$. Note that $z_S\cdot x_\xi=0$ if $S\nsubseteq \xi^{-1}(\be)$. Instead, the elements $z_S\cdot x_\xi$, for varying $\xi\colon \set{1,\dots,g}\to\bK$ and $\emptyset\subseteq S\subseteq\xi^{-1}(\be)$, give a basis of $\Lambda_R(x_1,\dots,x_{2g})$ as a free $R$-module. Hence $\bM^R_g$ decomposes as a $B^R_g$-module (and hence as a $\Gamma_R(y)$-module) as a direct sum
\[
 \bM^R_g\cong\bigoplus_{\mathclap{\xi\colon\set{1,\dots,g}\to\bK}} B^R_g\cdot x_\xi.
\]
The summand $B^R_g\cdot x_\xi$ is isomorphic, as a $B^R_g$-module, to the shifted quotient $B^R_g/(z_i\ |\ i\notin\xi^{-1}(\be))\ [|\xi|-g]$, hence it is isomorphic over $\Gamma_R(y)$ to $\twoB^R_{|\xi|}\ [|\xi|-g]$.
\end{proof}
As a consequence of Proposition \ref{prop:bMgsplitting}, we have an isomorphism
 \[
  \Ext^\star_{\Gamma_R(y)}(\bM_g,R)\cong \bigoplus_{\mathclap{\xi\colon\set{1,\dots,g}\to\bK}} \Ext^\star_{\Gamma_R(y)}( \twoB^R_{|\xi|},R)\ [g-|\xi|],
 \]
where again $[m]$ shifts up the weight $\bullet$ by $m$, without changing the bar-degree $\star$.

\subsection{Proof of the corollaries of Theorem \ref{thm:B}}
\begin{proof}[Proof of Corollary \ref{cor:splitstabextremalstab}]
For point (1), split injectivity of $\epsilon\cdot-$ follows from the presence of the tensor factor $R[\epsilon]$ in part (2) of Theorem \ref{thm:B}; homological stability follows from the following facts:
\begin{itemize}
 \item $R[\cH_g^R]$ is supported in weights $\bullet\ge0$ and bar-degree $\star=0$;
 \item $\bM_g$ and $\Gamma_R(y)$ are supported in non-positive weights, $\bM_g$ is free over $R$ and $\Gamma_R(y)$ is $R$ in weight $0$ and is otherwise concentrated in weights at most $-2$; as a consequence $\Ext^\star_{\Gamma_R(y)}(\bM_g,R)$ is supported in bidegrees $(\bullet,\star)$ satisfying $\bullet\ge-2\star\ge0$, i.e. $2(\bullet+\star)\ge\bullet$.
\end{itemize}

Point (2) follows from part (3) of Theorem \ref{thm:B} after noticing the following. First, we have $\Gamma_{\Q}(y)\cong\Q[y]$; since $\bM_g$ is finite-dimensional over $\Q$, it splits over $\Q[y]$ as a finite direct sum of torsion $\Q[y]$-modules; it follows that 
$\Ext^\star_{\Gamma_\Q(y)}(\bM_g,\Q)\cong \Ext^\star_{\Q[y]}(\bM_g,\Q)$
is finite dimensional. Second, $\epsilon$ has bigrading $(-1,1)$, so no power of $\epsilon$ higher than $\epsilon^i$ can occur in a class in $H_{k-i}(C_k(\sgone);\Q)$: therefore the weighted $\Q[\cH_g^\Q]$-module $H_{\bullet-i}(C_\bullet(\sgone);\Q)$ is finitely generated. We apply the theory of the Hilbert function for finitely generated weighted modules over a weighted polynomial ring.
\end{proof}
\begin{proof}[Proof of Corollary \ref{cor:filledgenus}]
The filled genus stabilisation corresponds in the light of Theorem \ref{thm:B} to the tensor product over $R$ of the following split injective maps of $R$-modules (whence it is itself split injective):
\begin{itemize}
 \item the identity of $R[\epsilon]$;
 \item the $R$-split inclusion $R[\cH_g^R]\hookrightarrow R[\cH_{g+1}^R]$ given by the $R$-split inclusion $\cH_g^R\hookrightarrow \cH_{g+1}^R$, which is the $R$-dual of the $R$-split surjection
 $\bigoplus_{i=1}^{2g+2}Ry_i\twoheadrightarrow\bigoplus_{i=1}^{2g}Ry_i$ killing $y_{2g+1}$ and $y_{2g+2}$;
 \item the $R$-split inclusion on Ext groups induced by the $\Gamma_R(y)$-split surjective map of modules $\bM_{g+1}^R\to\bM_g\otimes_R\pa{Rx_{2g+1}\oplus Rx_{2g+2}}$ which, in the light of Proposition \ref{prop:bMgsplitting}, is the $\Gamma_R(y)$-split surjective map
 \[
  \bigoplus_{\mathclap{\xi\colon\set{1,\dots,g+1}\to\bK}} \hspace{6pt} \twoB^R_{|\xi|}\ [|\xi|-g]\longrightarrow  \hspace{6pt}\bigoplus_{\mathclap{\substack{\xi\colon\set{1,\dots,g+1}\to\bK,\\ |\xi^{-1}(\be)\cap\set{2g+1,2g+2}|=1}}} \hspace{6pt}\twoB^R_{|\xi|}\ [|\xi|-g]
 \]
killing all summands corresponding to maps $\xi\colon\set{1,\dots,g+1}\to\bK$ such that either both or neither of $2g+1,2g+2$ is sent to $\be$.
\end{itemize}
\end{proof}

\part{Solution of the Ext problem}\label{part:3}
In this part of the article we conclude the computation of $H_*(C_\bullet(\sgone);\Fp)$ as a module over $H_*(C_\bullet(D);\Fp)$, for $p$ an odd prime. We observe that Theorem \ref{thm:B} and Proposition \ref{prop:bMgsplitting} reduce this to computing $\Ext^\star_{\Gamma_{\Fp}(y)}(\twoB^{\Fp}_u,\Fp)$ as a module over $\Ext^\star_{\Gamma_{\Fp}(y)}(\Fp,\Fp)$. We observe that for $p$ an odd prime we have an isomorphism of $\Gamma_{\Fp}(y)$-modules $\twoB^{\Fp}_u\cong B^{\Fp}_u$, so we might compute $\Ext^\star_{\Gamma_{\Fp}(y)}(B^{\Fp}_u,\Fp)$ as well.

Recall that for any prime $p$ (also for $p=2$) we have an identification of rings
\[
 \Ext^\star_{\Gamma_{\Fp}(y)}(\Fp,\Fp)\cong \bigotimes_{i=0}^\infty \Ext^\star_{\Fp[\fy_i]/(\fy_i^p)}(\Fp,\Fp).
\]
\addtocounter{athm}{-1}
\begin{athm}\label{thm:C}
Let $p$ be a prime, let $u\ge0$, and let $h\ge0$ be such that
\[
p^h-1\le u\le p^{h+1}-2.
\]
Then for every $0\le i\le h$ there is a finitely generated, weighted $\Fp[\fy_i]/(\fy_i^p)$-module $N_{u,i}$, such that 
$\Ext^\star_{\Gamma_{\Fp}(y)}(B^{\Fp}_u,\Fp)$ is isomorphic, as a bigraded $\Ext^\star_{\Gamma_{\Fp}(y)}(\Fp,\Fp)$-module, to the direct sum
\[
 \bigoplus_{i=0}^h \pa{\Fp\otimes_{\Fp}\Ext^\star_{\Fp[\fy_i]/(\fy_i^p)}(N_{u,i},\Fp)\otimes_{\Fp}\bigotimes_{j=i+1}^\infty \Ext^\star_{\Fp[\fy_j]/(\fy_j^p)}(\Fp,\Fp)}.
\]
Here the $i$\textsuperscript{th} summand is the tensor product of the following modules:
\begin{itemize}
 \item $\Fp$ as the augmentation module over the ring $\bigotimes_{j=0}^{i-1} \Ext^\star_{\Fp[\fy_j]/(\fy_j^p)}(\Fp,\Fp)$;
 \item $\Ext^\star_{\Fp[\fy_i]/(\fy_i^p)}(N_{u,i},\Fp)$ as a module over $\Ext^\star_{\Fp[\fy_i]/(\fy_i^p)}(\Fp,\Fp)$;
 \item the ring $\bigotimes_{j=i+1}^{\infty}\Ext^\star_{\Fp[\fy_j]/(\fy_j^p)}(\Fp,\Fp)$ as a module over itself.
\end{itemize}
\end{athm}
A characterisation of $N_{u,i}$ as a finite direct sum of cyclic $\Fp[\fy_i]/(\fy_i^p)$-modules will be given in Lemma \ref{lem:weighteddim}. For a cyclic $\Fp[\fy_i]/(\fy_i^p)$-module $N$ we consider the explicit computation of $\Ext^\star_{\Fp[\fy_i]/(\fy_i^p)}(N,\Fp)$ as a module over $\Ext^\star_{\Fp[\fy_i]/(\fy_i^p)}(\Fp,\Fp)$ to be a standard exercise in homological algebra, which we leave to the reader.

In Section \ref{sec:Kaehler} we will analyse more closely the $\Gamma_R(y)$-module $B^R_u$, for a generic commutative ring $R$. For $R=\Fp$, these results will be used in Section \ref{sec:tamemodules} to check that $B^{\Fp}_u$ is a \emph{tame} module over $\Gamma_{\Fp}(y)\cong\Fp[\fy_0,\fy_1,\dots]/(\fy_0^p,\fy_1^p,\dots)$ (see Definition \ref{defn:tame}).
We will then prove a structure result for tame modules, allowing in particular a computation of $\Ext^\star_{\Gamma_{\Fp}(y)}(B^{\Fp}_u,\Fp)$ as a $\Ext^\star_{\Gamma_{\Fp}(y)}(\Fp,\Fp)$-module.

\section{Analysis of \texorpdfstring{$B^R_u$}{BRu}}
\label{sec:Kaehler}
In this section we fix $u\ge0$ and a commutative ring $R$, and work in the category $\mathrm{Mod}_R^\Z$ of weighted $R$-modules.
Abstractly, $B^R_u$ is isomorphic to the cohomology ring of $(\C P^1)^u$ with coefficients in $R$.
The results that we obtain in the following correspond, in the case $R=\Q$, to (almost all of) the ``K\"ahler package'' of the rational cohomology of $(\C P^1)^u$, considered as a projective subvariety of $\C P^{2^u-1}$ via the standard Segre embedding (so that $y\in B^\Q_u$ corresponds to the class of a hyperplane in $\C P^{2^u-1}$). We shall use a corresponding terminology.
\begin{nota}
 For $S\subseteq\set{1,\dots,u}$ we denote by $z_S$ the element $\prod_{i\in S}z_i\in B^R_u$.
\end{nota}
Recall Definition \ref{defn:twoB}, and note that the ring homomorphism $\phi_1\colon \Gamma_R(y)\to B^R_u$ sends $y^{[k]}=y^k/k!\mapsto \sum_{S\subseteq\set{1,\dots,u},|S|=k}z_S$, for all $k\ge0$. In particular, for $S'\subset\set{1,\dots,u}$, we have $y^{[k]}\cdot z_{S'}=\sum_{S'\subseteq S\subseteq\set{1,\dots,u},|S|=|S'|+k}z_S$.

\begin{nota}\label{nota:vee}
Let $R$ be a ring (concentrated in weight 0) and let $M$ be a weighted $R$-module. We denote by $M^\vee$ the weighted $R$-module $\Hom_R(M,R)$; if $M(n)$ is the part of $M$ of weight $n\in\Z$, then $M^\vee$ has
$\Hom_R(M(n),R)$ in weight $-n$. We leave $R$ implicit in the notation.
For a weighted $R$-algebra $A$ and a weighted $A$-module $M$, we can consider $M$ as an $A$-$R$-bimodule, and thus consider $M^\vee$ as an $A$-module.
\end{nota}

\begin{lem}[Poincar\'e duality]
\label{lem:poincareduality}
 There is an isomorphism of $\Gamma_R(y)$-modules
 \[
(B^R_u)^\vee\cong B^R_u [2u].
 \]
\end{lem}
\begin{proof}
 A weighted $R$-linear basis for $B^R_u$ is given by the monomials $z_S$. Denote by $\bar S$ the complement of $S$, and by $z_S^\vee\in(B^R_u)^\vee=\Hom_R(B^R_u,R)$ the map sending $z_S\mapsto 1$ and  $z_{S'}\mapsto 0$ for $S'\neq S$. Note that $z_S^\vee\in(B^R_u)^\vee$ has weight $2|S|$. The $R$-linear map $B^R_u [2u]\to (B^R_u)^\vee$ sending $z_S[2u]\mapsto z^\vee_{\bar S}$ is a $\Gamma_R(y)$-linear isomorphism.
\end{proof}
\begin{nota}\label{nota:Mlek}
Let $A$ be a weighted algebra and $M$ a weighted $A$-module. For $k\in\Z$ we denote by $M_{\ge k}\subseteq M$ the sub-$A$-module by the elements of weight at least $k$, and by $M_{\le k}$ the sub-$A$-module generated by the elements of weight at most $k$.
\end{nota}
Note that if $A$ is concentrated in non-negative weights, then $M_{\le k}$ is just the additive span of all elements of $M$ of weight at most $k$.
\begin{ex}
For $k\ge0$, we have that $B^R_{u,\ge -2k}\subset B^R_u$ is the sub-$\Gamma_R(y)$-module generated by all elements $z_S$ with $|S|\le k$ (i.e. $z_S$ has weight at least $-2k$), whereas $B^R_{u,\le -2k}\subset B^R_u$ is the sub-$\Gamma_R(y)$-module generated by all elements $z_S$ with $|S|\ge k$ (i.e. $z_S$ has weight at most $-2k$).
Note that $B^R_{u,\le -2k}$ coincides with the $R$-linear span of the elements $z_S$ with $|S|\ge k$.
For instance, $B^R_{u,\le -2u}$ is spanned by the single monomial $z_1\dots z_u$, and $B^R_{u,\ge 0}$ is spanned by the elements $\sum_{|S|=k}z_S$ for varying $0\le k\le u$. On the other hand, we have
$B^R_{u,\le0}=B^R_{u,\ge -2u}=B^R_u$. 
\end{ex}

\begin{lem}[Hard Lefschetz]
\label{lem:hardlefschetz}
Let $0\le k\le u$. Then $B^R_{u,\ge -u+k}$ contains $B^R_{u,\le -u-k}$.\end{lem}
\begin{proof}
Let $T\subseteq\set{1,\dots,u}$ have cardinality at least $(u+k)/2$, so that $z_T$ is a generator of  $B^R_{u,\le -u-k}$, and let $\bar T$ be the complement of $T$. Then the inclusion-exclusion principle implies that $z_T$ can be expressed as
 \[
  z_T=\sum_{S\subseteq \bar T}(-1)^{|S|}\ y^{[|T|-|S|]}\cdot z_S,
 \]
 and in particular $z_T\in B^R_{u,\ge -u+k}$.
\end{proof}
\begin{defn}
 A finite subset $S\subset\Z_{\ge1}$ is \emph{sparse} if for all $i\in\Z_{\ge1}$ we have $|S\cap\set{1,\dots, i}|\le|S\setminus\set{1,\dots, i}|$, or equivalently $|S\cap\set{1,\dots, i}|\le \frac{i}{2}$.
 We let
 \[
 \Theta_u\colon\bigoplus_{\mathclap{\substack{S\subseteq\set{1,\dots,u}\\[1px] S\ \mathrm{sparse}}}} \Gamma_R(y) \sca{X_S} \to B^R_u
 \]
 be the map of weighted $\Gamma_R(y)$-modules sending $X_S\mapsto z_S$, where $X_S$ is a free generator in weight $-2|S|$.
\end{defn}
We note that if $S\subseteq\set{1,\dots,u}$ is sparse, then in particular $|S|\le\floor{u/2}$.
\begin{nota}\label{nota:ellui}
 For $u\ge0$ and $0\le k\le \floor{u/2}$, we denote $\ell_{u,k}=\binom{u}{k}-\binom{u}{k-1}>0$, where we set $\binom{u}{-1}=0$.
\end{nota}
\begin{lem}\label{lem:l_u,k-cardinality}
 For $u\ge0$ and $0\le k\le \floor{u/2}$ there are precisely $\ell_{u,k}$ sparse subsets $|S|\subseteq\set{1,\dots,u}$ of cardinality $k$.
\end{lem}
\begin{proof}
We leave the proof to the reader.
\end{proof}

\begin{lem}
\label{lem:sparsegenerate}
The map $\Theta_u$ is surjective.
\end{lem}
\begin{proof}
We induct on $u\ge0$. For $r=0$ the map $\Theta_0\colon \Gamma_R(y)\sca{X_\emptyset}\to B^R_0\cong R$ sends $X_\emptyset\mapsto z_\emptyset=1$, so it is surjective.
For the inductive step, we assume $u\ge1$ and note that there is a short exact sequence of $\Gamma_R(y)$-modules
\begin{equation}\label{eq:SESBuR}
    z_u B^R_u\hookrightarrow B^R_u\twoheadrightarrow B^R_u/z_u,
\end{equation}
and there are $\Gamma_R(y)$-isomorphisms $B^R_{u-1}[-2]\cong z_u B^R_u$, sending $z_S[-2]\mapsto z_uz_S$, for $S\subseteq\set{1,\dots,r-1}$, and $B^R_{u-1}\cong B^R_u/z_u$, sending $z_S\mapsto [z_S]_{/z_u}$.

By the inductive hypothesis and \eqref{eq:SESBuR}, the following elements suffice, together, to generate the entire $B^R_u$ over $\Gamma_R(y)$:
\begin{itemize}
 \item the elements $z_S$, lifting $[z_S]_{/z_u}$, for $S\subseteq\set{1,\dots,u-1}$ sparse;
 \item the elements $z_{S\cup\set{u}}=z_uz_S$, for $S\subseteq\set{1,\dots,u-1}$ sparse.
\end{itemize}
Suppose now that $S\subseteq\set{1,\dots,u-1}$ is sparse but $S\cup\set{u}$ is not sparse. Then the failure of $S\cup\set{u}$ from being sparse must be witnessed by the inequality $|S|>|\set{1,\dots,u}\setminus S|$, that is $|S|>\floor{u/2}$; by Lemma \ref{lem:hardlefschetz} we can safely remove the element $z_{S\cup\set{u}}=z_uz_S$ from the above list and still be left with a list of generators of $B^R_u$.
\end{proof}
\begin{prop}\label{prop:Psiukinj}
 The map $\Theta_u$ is injective in weights at least $-2(u-\floor{u/2})$. More generally, for all $0\le k\le\floor{u/2}$, the map $\Theta_u$ restricts to a map
 \[
 \Theta_{u,k}\colon\bigoplus_{\mathclap{\substack{S\subseteq\set{1,\dots,u}\\[1px]|S|\le k,\  S\ \mathrm{sparse}}}} \Gamma_R(y) \sca{X_S} \to B^R_u
 \]
that is split injective (as a map of $R$-modules) in every weight $-2m\ge -2(u-k)$. In particular $B^R_{u,\ge -2k}$, that is the image of $\Theta_{u,k}$, agrees with a free $\Gamma_R(y)$-module in weights at least $-2(u-k)$.
\end{prop}
\begin{proof}
It is enough to prove the result for the case $R=\Z$: for general $R$, we notice that all modules and maps are basechanged from $\Z$, and basechange preserves split injectivity.

Let $k\le\floor{u/2}$, and let $m$ be either at most $k$ or equal to $u-k$. By Lemma \ref{lem:l_u,k-cardinality}, in weight $-2m$ both the source and the target of $\Theta_{u,k}$ are free $\Z$-modules of same rank $\binom{r}{m}$. Moreover, by Lemmas \ref{lem:hardlefschetz} and \ref{lem:sparsegenerate}, the map 
$\Theta_{u,k}$ is surjective in weight $-2m$. It follows that $\Theta_{u,k}$ is an isomorphism in this weight, and in particular it is split injective.

Consider now a weight $-2m$ for $k<m<u-k$. We can consider the surjective $\Gamma_{\Z}(y)$-map $\mathfrak{q}_{u,m+k}\colon B^R_u\twoheadrightarrow B^{\Z}_{m+k}$ obtained by quotienting the variables $z_{m+k+1},\dots,z_u$ in the ring $B^{\Z}_u$. The composition $\mathfrak{q}_{u,m+k}\circ \Theta_{u,k}$ coincides with $\Theta_{m+k,k}$, and in particular it restricts to an isomorphism, and in particular a split injection, in weight $-2m$ by the previous argument. It follows that $\Theta_{u,k}$ is split injective in weight $-2m$.
\end{proof}
\begin{rem}
\label{rem:2andQ}
 Putting together the results of this section, and taking $R=\Q$, one obtains that $B^\Q_u$ is isomorphic, as a module over the ring $\Gamma_\Q(y)\cong\Q[t]$, to the direct sum
 \[
\bigoplus_{k=0}^{\floor{u/2}}\pa{\Q[y]/(y^{u+1-2k})[-2k]}^{\bigoplus \ell_{u,k}}.
\]
In particular, we obtain the analogue of Theorem \ref{thm:C} over $\Q$, and together with Theorem \ref{thm:B}, we obtain a computation of $H_*(C_\bullet(\sgone);\Q)$ as a $H_*(C_\bullet(D);\Q)$-module.
\end{rem}

\section{Tame modules}
\label{sec:tamemodules}
In the entire section we fix a field $\bF$ and work in the category $\mathrm{Mod}_{\bF}^\Z$ of weighted $\bF$-vector spaces; all tensor products are assumed to be over $\bF$.
\begin{nota}\label{nota:intervals}
For $m,m'\in\Z\cup\set{\pm\infty}$ with $m<m'$, we denote by ${]m,m'[}\subseteq\Z$ the set of all $i\in\Z$ with $m<i<m'$. Similarly, if $m,m'\in\Z$, we let $[m,m']=\set{i\in\Z\,|\, m\le i\le m'}$, and we define analogously $[m,m'[$ and $]m,m']$.
\end{nota}

\subsection{Definition of tame modules and examples}
\begin{defn}
 \label{defn:cAud}
Let $\ud=(d_0,d_1,d_2,\dots)$ be a sequence of positive integers with $d_0\ge1$ and $d_i\ge2$ for $i\ge1$.
We denote by $\cA(\ud)$ the truncated polynomial algebra
\[
\cA(\ud)=\bF[\fy_0,\fy_1,\fy_2,\dots]/(\fy_0^{d_1},\fy_1^{d_2},\fy_2^{d_3},\dots)=\bigotimes_{i=0}^\infty\bF[\fy_i]/(\fy_i^{d_{i+1}}),
\]
where $\fy_i$ is a variable of weight $-2\prod_{j=0}^id_j$.
\end{defn}
\begin{nota}\label{nota:cAcR}
We fix the sequence $\ud$ throughout the section, and abbreviate $\cA=\cA(\ud)$. We also set $D_i=\prod_{j=0}^id_j$ for all $i\ge-1$, so that $\fy_i$ has weight $-2D_i$ for all $i\ge0$; in particular $D_0=d_0$ and $D_{-1}=1$. We denote by $\cR=\bF[\fy_0]/(\fy_0^{d_1})$ the subalgebra of $\cA$ generated by the first variable $\fy_0$, and by $\cA'$ the ``complementary'' subalgebra $\bigotimes_{i=1}^\infty\bF[\fy_i]/(\fy_i^{d_{i+1}})$. Note that $\cA'=\cA(d_1,d_2,d_3,\dots)$.
\end{nota}
Note that $\cA$ is a weighted $\bF$-algebra concentrated in non-positive weights: more precisely, $\cA$ has dimension 1 over $\bF$ in each non-positive weight multiple of $2d_0$, and dimension 0 in every different weight. Moreover $\cA\cong\cR\otimes\cA'$.

\begin{ex}\label{ex:cAGammay}
If $p$ is a prime number, we can set $\bF=\Fp$, $d_0=1$ and $d_i=p$ for all $i\ge1$: then $\cA=\cA(\ud)$ is isomorphic to $\Gamma_{\Fp}(y)$, with $\fy_i$ corresponding to $y^{[i]}$.
\end{ex}

\begin{defn}
 \label{defn:tame}
 Let $A$ be a weighted algebra concentrated in non-positive weights. A weighted $A$-module $M$ is \emph{tame} if it satisfies the following conditions:
 \begin{itemize}
  \item $M$ is finite dimensional as an $\bF$-vector space;
  \item $M\cong M^\vee$ as $A$-modules;
  \item for all $k\ge0$, $M_{\ge k}$ (see Notation \ref{nota:Mlek}) agrees in weights at least $-k$ with a free $A$-module. In other words, there exists a free $A$-module $\cF_k$, generated in weights at least $k$, and an $A$-module homomorphism $\psi_k\colon\cF_k\to M$ which is bijective in weights $[k,\infty[$ and injective in weights $[-k, k]$.
 \end{itemize}
For $u\in\Z$, $M$ is \emph{$u$-tame} over $A$ if the shifted module $M[u]$ is tame over $A$.
\end{defn}
We note that if $M$ is tame over $A$ and $A'\subset A$ is a sub-$\bF$-algebra such that $A$ is free as an $A'$-module, then $M$ is also tame over $A'$. 
\begin{ex}\label{ex:augmentation}
If $A$ is concentrated in non-positive weights and is $\bF$ in weight 0 (generated by the unit), then the augmentation module $\bF$, i.e. the quotient of $A$ by its elements of negative weight, is a tame $A$-module.
\end{ex}
\begin{nota}
In the setting of Example \ref{ex:augmentation}, we denote by $A_+\subset A$ the augmentation ideal of $A$, spanned by all elements of negative weight.
\end{nota}

\begin{ex}
 For all $i\ge 0$, the shifted quotient $\cA/(\fy_i,\fy_{i+1},\dots)[D_i-d_0]$ is a tame $\cA$-module: it is generated by the single element $1$ (in weight $D_i-d_0$), and it coincides with the quotient of $\cA[D_i-d_0]$ by all elements of weight at most $-D_i+d_0$.
\end{ex}
\begin{ex}\label{ex:Bputame}
 Let $p$ be a prime number, let $\bF=\Fp$, and let $u\ge0$. Then Lemma \ref{lem:poincareduality} and Proposition \ref{prop:Psiukinj} imply that $B^p_u[u]$ is tame as a module over the algebra $\cA\cong\Gamma_{\Fp}(y)$ from Example \ref{ex:cAGammay}; in other words, $B^p_u$ is $u$-tame.
\end{ex}

\begin{rem}
\label{rem:choicecF}
Let $M$ be a tame $A$-module, let $k\ge0$, and let $\cF_k$ be as in Definition \ref{defn:tame}. Then $\cF_k$ must have in each weight $m\ge k$ as many free $A$-generators as the dimension of $M/A_+M$ in weight $m$: this characterises $\cF_k$ up to isomorphism.
Moreover, let $\psi_k,\tilde\psi_k\colon\cF_k\to M$ be two $A$-linear maps inducing a bijection in weights at least $ k$; then both $\psi_k$ and $\tilde\psi_k$ induce a bijection $\cF_k/A_+\cF_k\overset{\cong}{\to} M/A_+M$ in weights at least $k$, and the composite bijection $\cF_k/A_+\cF_k\overset{\cong}{\to} M/A_+M\overset{\cong}{\leftarrow}\cF_k/A_+\cF_k$ can be lifted to an $A$-linear automorphism of $\cF_k$ intertwining $\psi_k$ and $\tilde\psi_k$. In particular $\psi_k$ is injective in weights in $[-k,k]$ if and only if $\tilde \psi_k$  is injective in this range of weights.

\end{rem}

\subsection{Barcodes}
Recall Notation \ref{nota:cAcR}. If $M$ is a $\cA$-module, it is in particular a $\cR$-module; since $\cA$ is free over $\cR$, if $M$ is tame over $\cA$ then it is also tame over $\cR$. The ring $\cR$ is a quotient of the PID $\bF[\fy_0]$, and a finitely generated, weighted $\cR$-module $M$ is isomorphic to a finite direct sum of shifts of cyclic modules, of the form $\bF[\fy_0]/(\fy_0^c)[m]$ for some $1\le c\le d_1$ and $m\in\Z$.

\begin{defn}
The $\fy_0$-\emph{barcode} of a finitely generated, weighted $\cR$-module $M$ is the finite multiset $\cB_{\fy_0}(M)$ of pairs $(m,c)\in\Z\times\set{1,\dots,d_1}$ indexing the direct sum $\bigoplus_{(m,c)}\cR/\fy_0^c[m]$ classifying the isomorphism type of $M$ as an $\cR$-module.

Each pair $(m,c)$ is a $\fy_0$-\emph{bar}; the number $c$ is the \emph{size} of the bar, and it coincides with the dimension over $\bF$ of $\cR/(\fy_0^c)[m]$;
the support of $(m,c)$ is the finite subset $\set{m-2id_0\,|\,0\le i\le c-1}\subset\Z$. The \emph{barycentre} of $(m,c)$ is the integer $m-(c-1)d_0$.
\end{defn}
\begin{lem}\label{lem:barycentre}
Let $u\in\Z$ and let $M$ be a $u$-tame $\cR$-module;
then the barcode $\cB_{\fy_0}(M)$
satisfies the following:
\begin{enumerate}
 \item for each $i\in\Z$, and each $1\le c\le d_1$, there are as many bars of size $c$ with barycentre $-u-i$ as there are with barycentre $-u+i$;
 \item each bar $(m,c)$ of size $1\le c<d_1$ has barycentre
inside $]{-u-d_0},{-u+d_0}[$.
\end{enumerate}
\end{lem}
\begin{proof}
We only consider the case $u=0$.
Part (1) is a direct consequence of the $\cR$-linear isomorphism $M\cong M^\vee$.
For part (2), let $(m,c)$ be a bar in $\cB_{\fy_0}(M)$ with $1\le c<d_1$;
by (1) we can assume
$m\ge (c-1)d_0$ that is, 
$(m,c)$ has barycentre at least $0$. We now assume
$m\ge cd_0$, that is, the barycentre is at least $d_0$,
and look for a contradiction.
Decompose of $M$ as a direct sum of cyclic $\cR$-modules, and let $\nu\in M$ be a generator of a summand corresponding to the bar $(m,c)$;
then $\nu\in M_{\ge m}$, which agrees with a free $\cR$-module in weights at least $-m$; since $\fy_0^c\nu=0$ is an equality in $M_{\ge m}$ in weight $m-2cd_0\ge -m$, we must have that $\nu$ is divisible by $\fy_0^{d_1-c}$, hence by $\fy_0$, in $M_{\ge m}$ and hence also in $M$. But $\nu$ generates a cyclic summand of $M$.
\end{proof}

\subsection{The free-narrow decomposition}
\begin{defn}
 A graded vector space is \emph{narrow} if it is concentrated in weights in $]{-D_1+d_0-u},{D_1-d_0-u}[$.
\end{defn}

\begin{prop}\label{prop:freenarrow}
Let $u\in\Z$ and let $M$ be a $u$-tame $\cA$-module (see Notation \ref{nota:cAcR}).
Then $M$ is isomorphic, as an $\cA$-module, to the direct sum $F\oplus N$ of two $u$-tame $\cA$-modules $F$ and $N$, such that $F$ is free over the subring $\cR\subset\cA$, and $N$ is narrow.
The isomorphism types of $F$ and $N$ as $\cA$-modules are uniquely determined by $M$.
\end{prop}
The rest of the subsection is devoted to the proof of Proposition \ref{prop:freenarrow}; we only consider the case $u=0$, and fix a tame $\cA$-module $M$ throughout the subsection.

By Lemma \ref{lem:barycentre}, $M$ decomposes as an $\cR$-module as a direct sum $\hat F\oplus \hat N$, where $\hat F$  and $\hat N$ satisfy the requirements \emph{as modules over $\cR$}; however it is not clear a priori whether one can achieve that $\hat F$ and $\hat N$ are tame sub-$\cA$-modules of $M$.
For $\hat N$, which is narrow, being a sub-$\cA$-module of $M$ is equivalent to the requirement that for all $i\ge1$, the map $\fy_i\cdot-\colon M\to M$ restricts to the zero map on $\hat N$.

We fix a decomposition $M\cong \hat F\oplus \hat N$ over $\cR$ as above throughout the subsection.

\begin{lem}\label{lem:kerinImm}
 Let $\ua=\fy_0^{i_0}\fy_1^{i_1}\dots \fy_h^{i_h}$ be a monomial in $\cA$, with $h\ge0$ and $0\le i_j< d_j$ for $0\le j\le h$, and denote $k=\sum_{j=0}^h D_ji_j$, so that $\ua$ has weight $-2k$.
 Let $\nu\in M$ be an element of weight at most $-k$, and assume that $\fy_j^{d_j-i_j}\nu=0$ for all $0\le j\le h$. Then $\nu$ is in the image of the map $\ua\cdot-\colon M\to M$.
\end{lem}
\begin{proof}
 The statement is a dualisation of the third condition for $M$ to be tame from Definition \ref{defn:tame}, and we spell out the details.
 
 By the definition of $\cA$, we have an inclusion of weighted sub-$\bF$-vector spaces of $M$ of the form $\Imm(\ua\cdot-)\subseteq\ker (\fy_j^{d_j-i_j}\cdot-)$, for every $0\le j\le h$. The statement of the lemma is equivalent to this inclusion being an equality in weights  at most $-k$. By dualising and a rank-nullity argument, it suffices to prove the inclusion $\ker(a\cdot-)\subseteq\sum_{j=0}^h\Imm(\fy_j^{d_j-i_j}\cdot-)$ sub-$\bF$-vector spaces of $M^{\vee}$ in weights at least $k$.
  
 Since $M^\vee$ is isomorphic to $M$, it is also tame, and in particular $M^\vee_{\ge k}$ agrees with a free $\cA$-module in weights at least $-k$. Thus, if $z\in\ker(\ua\cdot-)\subseteq M^\vee$ has weight at least $k$, then in particular $z\in M^\vee_{\ge k}$; lifting $z$ to a (unique) element $\tilde z\in\cF_k$ along a map $\psi_k\colon\cF_k\to M^\vee$ as in Definition \ref{defn:tame}, we have that $\ua\tilde z=0\in\cF_k$, and since $\cF_k$ is free, this implies that $\tilde z$ is in $\sum_{j=0}^h\Imm(\fy_j^{d_j-i_j}\cdot-)\subseteq\cF_k$; projecting along $\psi_k$ to $M^\vee$ we obtain $z\in \sum_{j=0}^h\Imm(\fy_j^{d_j-i_j}\cdot-)\subseteq M^\vee$.
\end{proof}

\begin{nota}\label{nota:hatNdecomposition}
We fix a decomposition of $\hat N$ as a direct sum of cyclic modules $\bigoplus_{j=1}^r\pa{\bF[\fy_0]/(\fy_0^{c_j})}[m_j]$, where $r=|\cB_{\fy_0}(\hat N)|$. We denote by $\nu_j\in \hat N$ the generator of the $j$\sth cyclic summand; by Lemma \ref{lem:barycentre}, $\nu_j$ has weight $m_j>-d_0$.
\end{nota}

\begin{lem}
\label{lem:perturbation}
 Let $\nu\in \hat N$ be an element of weight $m\in]{-d_0},{D_1-d_0}[$ generating a $\fy_0$-bar of $\hat N$ of size $1\le c<d_1$ (i.e. $\nu=\nu_j$ for some $j$). Then there exists an element $\bar w\in M$ of weight $m$, such that $\bar w$ is in the image of $\fy_0\cdot-\colon M\to M$, and moreover $\fy_0^c\bar w=0$ and $\fy_i\bar w=\fy_i\nu$ for all $i\ge1$.
\end{lem}
\begin{proof}
By Lemma \ref{lem:barycentre} we have $m\in ]{(c-2)d_0},{cd_0}[$.
First, we prove by induction on $h\ge0$ that there exists an element $w_h\in\Imm(\fy_0\cdot-)$ of weight $m$ and such that $\fy_0^cw_h=0$ and $\fy_iw_h=\fy_i\nu$ for all $1\le i\le h$. For $h=0$ we just take $w_0=0$. Let now $h>0$ and suppose that we have proved the existence of $w_{h-1}$.

Let $z=\fy_h(\nu-w_{h-1})$; we observe that $z$ has weight $m-2D_h$ and is in the kernel of any of the maps $\fy_0^c\cdot-$, $\fy_1\cdot-$, $\dots$, $\fy_{h-1}\cdot-$ and $\fy_h^{d_{h+1}-1}\cdot-$ (the last one because $z$ is a multiple of $\fy_h$ in $M$); we want to argue that $z$ has the form $\ua z'$ for some $z'\in M$, where $\ua$ is the monomial $\fy_0^{d_1-c}\fy_1^{d_2-1}\dots \fy_{h-1}^{d_h-1}\fy_h$, of weight
\[
-2(D_1-cd_0+\sum_{j=1}^{h-1}(D_{i+1}-D_i)+D_h)=-2(2D_h-cd_0).
\]
For this we apply Lemma \ref{lem:kerinImm}, checking the inequality $m-2D_h\le cd_0-2D_h$, which holds (even strictly) since $m<cd_0$.
We now set
$w_h=w_{h-1}+\fy_0^{d_1-c}\fy_1^{d_2-1}\dots \fy_{h-1}^{d_h-1}z'$, and readily check that it satisfies the required properties.

Since $M$ is tame, it is supported in $]{-D_h},D_h[$ for some $h$: we then set $\bar w=w_{h-1}$, which just as $\nu$ is killed by all variables $\fy_i$ with $i\ge h$.
\end{proof}

\begin{nota}
 \label{nota:torsion}
 Let $A$ be an algebra and let $a\in A$. For an $A$-module $M$ we denote by $M\set{a}\subset M$ the submodule of $a$-torsion elements. Both assignments $M\mapsto M\set{a}$ and $M\mapsto M/M\set{a}$ extend naturally to functors $\mathrm{Mod}_A\to\mathrm{Mod}_A$.
\end{nota}

\begin{proof}[Proof of Proposition \ref{prop:freenarrow}]
Recall Notation \ref{nota:hatNdecomposition}. For each $1\le j\le r$ fix an element $\bar w_j\in M$ as in Lemma \ref{lem:perturbation}, corresponding to the element $\nu_j\in \hat N$.
Let $\bar N$ be the sub-$\cR$-module of $M$ generated by the elements $\nu_j-\bar w_j$ for $1\le j\le r$. Then, for every $i\ge1$, the map $\fy_i\cdot-\colon M\to M$ vanishes identically on $\bar N$, as it is a map of $\cR$-modules and it vanishes on the $\cR$-generators of $\bar N$. It follows that $\bar N$ is a sub-$\cA$-module of $M$. Moreover, since by the assertion of Lemma \ref{lem:perturbation} each generator of $\bar N$ differs from the corresponding generator of $\hat N$ by a $\fy_0$-multiple, we have that $M$ splits also as the $\cR$-module direct sum $\hat F\oplus \bar N$.

The above reasoning also implies that $\bar N$ and $\hat N$ are isomorphic $\cR$-modules; as $\hat N$ is $\cR$-self-dual, then so is $\bar N$. Indeed, the $\fy_0$-barcodes of $\bar N$ and $\bar N^\vee$ are obtained from the $\fy_0$-barcodes of $M$ and $M^\vee$, respectively, by forgetting all bars of size $d_1$. Since for all $i\ge1$ the element $\fy_i\in\cA$ acts trivially on both $\bar N$ and $\bar N^\vee$ (both modules are ``narrow'') we also have that $\bar N\cong \bar N^\vee$ as $\cA$-modules.

The inclusion of $\cA$-modules $\bar N\hookrightarrow M$ dualises to a surjection of $\cA$-modules $M^\vee\twoheadrightarrow \bar N^\vee$; we precompose this with the $\cA$-linear isomorphism $M\cong M^\vee$ and define $F\subset M$ as the kernel of the composition: by construction $F$ is a sub-$\cA$-module of $M$. We notice that $F$ is isomorphic to $(\hat F)^\vee$, and hence to $\hat F$, as an $\cR$-module; in particular $F$ is free over $\cR$.

We have obtained a short exact sequence of $\cA$-modules $F\hookrightarrow M\twoheadrightarrow \bar N^\vee$, and we will prove it is split by checking that $M$ retracts $\cA$-linearly onto $F$. This will follow from proving that the $\cA$-linear composition $\iota\colon F\hookrightarrow M\cong M^\vee \twoheadrightarrow F^\vee$ is an isomorphism of $\bF$-vector spaces; by finite-dimensionality, it suffices to prove that $\iota$ is a surjection. Furthermore, since $\iota$ is $\cR$-linear, it suffices to prove that $\iota$ induces a surjection $F/\fy_0 F\twoheadrightarrow F^\vee/\fy_0F^\vee$. Now notice that, since both $F$ and $F^\vee$ are free over $\cR$, an element of either $F$ or $F^\vee$ is a multiple of $\fy_0$ if and only if it is $\fy_0^{d_1-1}$-torsion; thus it suffices to check that $\iota/\iota\set{\fy_0^{d_1-1}}$ is surjective, but the latter is the composition
\[
F/F\set{\fy_0^{d_1-1}}\to M/M\set{\fy_0^{d_1-1}}\to M^\vee/M^\vee\set{\fy_0^{d_1-1}}\to F^\vee/F^\vee\set{\fy_0^{d_1-1}},
\]
where all maps are isomorphisms because $\bar N$ and $\bar N^\vee$ are $\fy_0^{d_1-1}$-torsion. Thus we conclude that $\iota$ is an $\cA$-isomorphism, proving that $F$ is a split $\cA$-submodule of $M$. Now let $N$ be a direct complement of $F$, e.g. the image of a section of the $\cA$-linear split surjection $M\twoheadrightarrow\bar N^\vee$.

It remains to check that $F$ and $N$ are tame over $\cA$. We already have checked that $N\cong N^\vee$ and $F\cong F^\vee$ as $\cA$-modules.
Let $k\ge0$ and let $\psi_k\colon\cF_k\to M$ be an $\cA$-linear map as in Definition \ref{defn:tame}, with $\cF_k$ free over $\cA$. 
By Remark \ref{rem:choicecF}, we can assume that $\cF_k$ splits as a direct sum of free $\cA$-modules $\cF_k^F$ and $\cF_k^N$, so that $\psi_k$ restricts to maps $\cF_k^F\to F$ and $\cF_k^N\to N$: here we use that the weighted $\bF$-vector space $M/\cA_+M$ is the direct sum of the weighted $\cF$-vector spaces $F/\cA_+F$ and $N/\cA_+N$. If $\psi_k$ is injective in weights at least $-k$, so must be its restrictions to $\cF_k^F$ and $\cF_k^N$. This proves that $F$ and $N$ are tame over $\cA$.
\end{proof}

\subsection{Propagation of tameness}
Let $F$ be an $\cA$-module. Considering both $F$ and $\bF$ as $\cR$-modules, we compute
$F/\fy_0F\cong \Tor^{\cR}_0(F,\bF)$.
Recall Notation \ref{nota:cAcR}: we have $\cA\cong\cR\otimes\cA'$, and in particular 
$F/\fy_0F$ inherits a $\cA'$-module structure.
\begin{prop}\label{prop:propagation}
Let $u\in\Z$ and let $F$ be a $u$-tame $\cA$-module which is free over the ring $\cR$. Then the module $F/\fy_0F$ is $(u-D_1+d_0)$-tame over $\cA'$.
\end{prop}
\begin{proof}
It suffices to consider the case $u=0$, i.e. $F$ is tame over $\cA$, and prove that the shifted module $(F/\fy_0F)[-D_1+d_0]$ is tame over $\cA'$.

Since $F$ is free over $\cR$, we have that the map $\fy_0^{d_1-1}\cdot-$ induces a weight-preserving isomorphism of $\cA$-modules $F/\fy_0F\cong F\set{\fy_0}[2(D_1-d_0)]$. On the other hand, it is easy to see that for \emph{any} $\cA$-module $M$ we have an isomorphism of $\cA$-modules $M\set{\fy_0}^\vee\cong M^\vee/\fy_0M^\vee$; using that $F$ is tame over $\cA$, we obtain a chain of $\cA$-linear isomorphisms
\[
\begin{split}
 (F/\fy_0F[-D_1+d_0])^\vee & \cong (F\set{\fy_0}[D_1-d_0])^\vee\cong F\set{\fy_0}^\vee[-D_1+d_0]\\
 &\cong (F^\vee/\fy_0 F^\vee)[-D_1+d_0]\cong (F/\fy_0F)[-D_1+d_0].
\end{split}
\]
Next, let $k\ge D_1-d_0$, and let $\psi_k\colon \cF_k\to F$ be as in Definition \ref{defn:tame}, with $\cF_k$ free over $\cA$ and generated in weights at least $k$. Then $\cF'_k:=\cF_k/\fy_0\cF_k$ is free over $\cA'$ and is generated in weights at least $k$, and $\psi'_k:=\psi_k/\fy_0\psi_k\colon\cF'_k \to F/\fy_0F$ is a map of $\cA'$-modules and is surjective in weights at least $k$. We claim that $\psi'_k$ is injective in weights at least $-k+2(D_1-d_0)$.

To prove this, let $\nu_1,\dots,\nu_r$ be generators of $\cF_k$ over $\cA$, and let $\bar \nu_1,\dots,\bar \nu_r$ be their projections on $\cF'_k$, which are generators of $\cF'_k$ over $\cA'$. Let $m\ge -k+2(D_1-d_0)$, and suppose that there are elements $\lambda_1,\dots,\lambda_r\in\cA'$, not all zero, such that the sum $\sum_{j=1}^r\lambda_j\psi'_k(\bar \nu_j)$ is a homogeneous sum of weight $m$ and vanishes in $F/\fy_0F$. Then the element $\sum_{j=1}^r\lambda_j \nu_j\in \cF_k$ is sent along $\psi_k$ to an element divisible by $\fy_0$, hence this element must itself be divisible by $\fy_0$ in $\cF_k$. It follows that $\sum_{j=1}^r\lambda_j \nu_j$ lies in the kernel of the map $\fy_0^{d_1-1}\cdot-\colon \cF_k\to \cF_k$, i.e. $\sum_{j=1}^r\fy_0^{d_1-1}\lambda_j \nu_j$ is a homogeneous sum of weight $m-2(D_1-d_0)\ge -k$, and it vanishes: since at least one coefficient $\fy_0^{d_1-1}\lambda_j\in\cA=\cR\otimes\cA'$ is non-zero, we get a contradiction.

The maps $\psi'_k[-D_1+d_0]$ witness that $F/\fy_0F[-D_1+d_0]$ is tame over $\cA'$.
\end{proof}

\section{Proof of Theorem \ref{thm:C} and its corollaries}
\label{sec:proofthmC}
In this section we prove Theorem \ref{thm:C}, together with its corollaries.
Throughout the section we let $\cA=\cA(1,p,p,p,\dots)\cong\Gamma_{\Fp}(y)$ as in Example \ref{ex:cAGammay}.

\begin{defn}[Modules \texorpdfstring{$N_{u,i}$}{Nui}]
Fix $u\ge0$. We define recursively, for $i\ge0$, an $\Fp[\fy_i]/(\fy_i^p)$-module $N_{u,i}$. 

Recall from Example \ref{ex:Bputame} that the $\cA$-module $B^{\Fp}_u$ is $u$-tame. By Proposition \ref{prop:freenarrow}, we can decompose $B^{\Fp}_u$ as a direct sum of $\cA$-modules $F_{u,0}\oplus N_{u,0}$, with both modules being $u$-tame, $F_{u,0}$ being free over $\Fp[\fy_0]/(\fy_0^p)$ and $N_{u,0}$ being narrow. We consider $N_{u,0}$ as an $\Fp[\fy_0]/(\fy_0^p)$-module, and recall that all variables $\fy_1,\fy_2,\dots$ act trivially on $N_{u,0}$. This is the base case.

Recursively, we assume to have defined, for some $i\ge0$, a module $F_{u,i}$ which is $(u-p^i+1)$-tame over the algebra 
\[
\cA(p^i,p,p,\dots)=\Fp[\fy_i,\fy_{i+1},\dots]/(\fy_i^p,\fy_{i+1}^p,\dots),
\]
and such that $F_{u,i}$ is free as a module over $\Fp[\fy_i]/(\fy_i^p)$; by Proposition \ref{prop:propagation} the module $F_{u,i} /\fy_iF_{u,i}$ is $(u-p^{i+1}+1)$-tame over the algebra
\[
\cA(p^{i+1},p,p,\dots)=\Fp[\fy_{i+1},\fy_{i+2}\dots]/(\fy_{i+1}^p,\fy_{i+2}^p,\dots),
\]
and by Proposition \ref{prop:freenarrow} we have a decomposition of the module $F_{u,i} /\fy_iF_{u,i}$ as a direct sum $F_{u,i+1}\oplus N_{u,i+1}$, with both modules being $(u-p^{i+1}+1)$-tame over $\cA(p^{i+1},p,p,\dots)$, $F_{u,i+1}$ being free over $\Fp[\fy_{i+1}]/(\fy_{i+1}^p)$ and $N_{u,i+1}$ being \emph{narrow}: in particular all variables $\fy_{i+2},\fy_{i+3},\dots$ act trivially on $N_{u,i+1}$. We consider $N_{u,i+1}$ as an $\Fp[\fy_{i+1}]$-module.
\end{defn}

We remark that the isomorphism type of each $\Fp[\fy_i]/(\fy_i^p)$-module $N_{u,i}$ is unique, by Proposition \ref{prop:freenarrow}. Note also that the module $N_{u,i}$ is concentrated in even weights contained in the interval $]-u-p^{i+1}+2p^i-1,-u+p^{i+1}-1[$.
\begin{defn}
 For a finite-dimensional weighted $\Fp$-vector space $M$, we denote by $P_M(s)\in\Z[s^{\pm1}]$ the Poincar\'e polynomial of $M$, i.e. the Laurent polynomial
 \[
  P_M(s)=\sum_{j\in\Z}(\dim_{\Fp}M(j))s^j
 \]
\end{defn}
We notice that, for all $i\ge0$, the Laurent polynomial $P_{N_{u,i}}(s)$ has vanishing coefficients of all monomials $s^j$ for $j\notin]-u-p^{i+1}+2p^i-1,-u+p^{i+1}-1[$. This, together with the following Lemma, characterises the polynomials $P_{N_{u,i}}(s)$. 

\begin{lem}
 \label{lem:weighteddim}
 For $u\ge0$, the Poincar\'e polynomials of the modules $N_{u,i}$
 satisfy the following equality:
 \[
  (1+s^{-2})^u=\sum_{i\ge0}P_{N_{u,i}}(s)\cdot\frac{s^{-2p^i}-1}{s^{-2}-1}=\sum_{i\ge0}P_{N_{u,i}}(s)\cdot(1+s^{-2}+s^{-4}+\dots+s^{-2p^i+2}).
 \]
\end{lem}
\begin{proof}
We leave the proof to the reader.
\end{proof}

\begin{rem}\label{rem:NuiEventuallyVanishes}
The sum in the right-hand side of the formula of Lemma \ref{lem:weighteddim} is a sum of Laurent polynomials with non-negative coefficients. Since the left hand side is a polynomial supported on monomials $s^j$ for $j\in[-2u,0]$, which is an interval of width $2u$, there can be no summand in the right hand side contributing to the coefficients of two monomials $s^j$ and $s^{j'}$ with $|j-j'|>2u$. In particular, if $h$ is such that $p^h-1\le u\le p^{h+1}-2$, then $N_{u,i}=0$ for all $i> h$.
\end{rem}

\begin{proof}[Proof of Theorem \ref{thm:C}]
We keep denoting $\cA=\Gamma_{\Fp}(y)$.
The aim is to compute $\Ext^\star_{\cA}(B^{\Fp}_u,\Fp)$ as a module over $\Ext^\star_{\cA}(\Fp,\Fp)$, and we do so by a recursion on increasingly deep submodules and the number of variables $\fy_i$ involved.

We start by decomposing $B^{\Fp}_u$ to $F_{u,0}\oplus N_{u,0}$ as an $\cA$-module. We have an isomorphism of
$\Ext^\star_{\cA}(\Fp,\Fp)$-modules
\[
 \Ext^\star_{\cA}(B^{\Fp}_u,\Fp)\cong \Ext^\star_{\cA}(F_{u,0},\Fp)\oplus \Ext^\star_{\cA}(N_{u,0},\Fp),
\]
so we focus on the computation of the two summands.

We compute $\Ext^\star_{\cA}(N_{u,0},\Fp)$ using the K\"unneth isomorphism. The algebra $\cA$ is the tensor product of the algebras $\Fp[\fy_0]/(\fy_0^p)$ and $\Fp[\fy_1,\dots]/(\fy_1^p,\dots)$. Correspondingly, the $\cA$-module $N_{u,0}$ (respectively, $\Fp$) is the tensor product of the $\Fp[\fy_0]/(\fy_0^p)$-module $N_{u,0}$ (respectively, $\Fp$) and the $\Fp[\fy_1,\dots]/(\fy_1^p,\dots)$-module $\Fp$. We thus factor $\Ext^{\star}_{\cA}(\Fp,\Fp)$ as the tensor product of the rings
$\Ext^\star_{\Fp[\fy_0]/(\fy_0^p)}(\Fp,\Fp)$ and $\Ext^\star_{\Fp[\fy_1,\dots]/(\fy_1^p,\dots)}(\Fp,\Fp)$ and, correspondingly, $\Ext^{\star}_{\cA}(N_{u,0},\Fp)$ as the tensor product of the modules 
$\Ext^\star_{\Fp[\fy_0]/(\fy_0^p)}(N_{u,0},\Fp)$ and $\Ext^\star_{\Fp[\fy_1,\dots]/(\fy_1^p,\dots)}(\Fp,\Fp)$. This accounts for the $0$\sth direct summand in the statement of Theorem \ref{thm:C}.

Turning to the other summand $\Ext^\star_{\cA}(F_{u,0},\Fp)$, it follows from a reason similar to the previous paragraph that $\cA$ and the second module $\Fp$ factor are tensor products of two rings and two modules, respectively. We may thus invoke the K\"unneth spectral sequence, which is supported on the quadrant $s,t\le0$ and reads:
\begin{equation}\label{eq:KunnethSSFu0}
     E^2_{s,t}=\Ext^{s}_{\Fp[\fy_1,\dots]/(\fy_1^p,\dots)}(\Tor_{-t}^{\Fp[\fy_0]/(\fy_0^p)}(F_{u,0},\Fp),\Fp)\Rightarrow \Ext^{s+t}_{\cA}(F_{u,0},\Fp).
\end{equation}
Since $F_{u,0}$ is a free $\Fp[\fy_0]/(\fy_0^p)$-module, only the row $E^2_{s,0}$ contains non-trivial entries, and in particular the spectral sequence collapses on the second page. The spectral sequence \eqref{eq:KunnethSSFu0} is a module over the K\"unneth spectral sequence for $\Ext^\star_{\cA}(\Fp,\Fp)$, which also collapses on its second page.
By identifying $\Tor_0^{\Fp[\fy_0]/(\fy_0^p)}(F_{u,0},\Fp)$ with $F_{u,0} /\fy_0F_{u,0}$ as $\Fp[\fy_1,\dots]/(\fy_1^p,\dots)$-modules, and using the ring decomposition 
\[\Ext^\star_{\cA}(\Fp,\Fp)\cong \Ext^\star_{\Fp[\fy_0]/(\fy_0^p)}(\Fp,\Fp)\otimes \Ext^\star_{\Fp[\fy_1,\dots]/(\fy_1^p,\dots)}(\Fp,\Fp),\] 
we obtain an identification of $\Ext^\star_{\cA}(\Fp,\Fp)$-modules
\[
\Ext^\star_{\cA}(F_{u,0},\Fp)\cong \Fp\otimes \Ext^\star_{\Fp[\fy_1,\dots]/(\fy_1^p,\dots)}(F_{u,0} /\fy_0F_{u,0},\Fp).
\]
In particular, all classes in $\Ext^\star_{\Fp[\fy_0]/(\fy_0^p)}(\Fp,\Fp)$ for $\star$ strictly negative act trivially on $\Ext^\star_{\cA}(F_{u,0},\Fp)$.

We then continue recursively as follows. Suppose that, for some $i\ge1$, we have reduced ourselves to computing $\Ext^\star_{\Fp[\fy_i,\dots]/(\fy_i^p,\dots)}(F_{u,i-1} /\fy_{i-1}F_{u,i-1},\Fp)$ as an $\Ext^\star_{\Fp[\fy_i,\dots]/(\fy_i^p,\dots)}(\Fp,\Fp)$-module. We first decompose $F_{u,i-1} /\fy_{i-1}F_{u,i-1}$ as a direct sum of $\Fp[\fy_i,\dots]/(\fy_i^p,\dots)$-modules $F_{u,i}\oplus N_{u,i}$. By K\"unneth, we obtain  
an isomorphism of $\Ext^\star_{\Fp[\fy_i]/(\fy_i^p)}(\Fp,\Fp)\otimes \Ext^\star_{\Fp[\fy_{i+1},\dots]/(\fy_{i+1}^p,\dots)}(\Fp,\Fp)$-modules
\[
\begin{split}
 &\Ext^\star_{\Fp[\fy_i,\dots]/(\fy_i^p,\dots)}(N_{u,i},\Fp)\cong\\
 &\Ext^\star_{\Fp[\fy_i]/(\fy_i^p)}(N_{u,i},\Fp)\otimes \Ext^\star_{\Fp[\fy_{i+1},\dots]/(\fy_{i+1}^p,\dots)}(\Fp,\Fp),
\end{split}
\]
which accounts for the $i$\sth direct summand in the statement of Theorem \ref{thm:C}. The remaining summand is $\Ext^\star_{\Fp[\fy_i,\dots]/(\fy_i^p,\dots)}(F_{u,i},\Fp)$, to which we apply the K\"unneth spectral sequence; this is concentrated on the $0$\sth row and thus collapses, so we identify
\[
 \Ext^\star_{\Fp[\fy_i,\dots]/(\fy_i^p,\dots)}(F_{u,i},\Fp)\cong\Fp\otimes
 \Ext^\star_{\Fp[\fy_{i+1},\dots]/(\fy_{i+1}^p,\dots)}(F_{u,i} /\fy_iF_{u,i},\Fp).
\]
The problem is reduced to computing $\Ext^\star_{\Fp[\fy_{i+1},\dots]/(\fy_{i+1}^p,\dots)}(F_{u,i} /\fy_iF_{u,i},\Fp)$ as an $\Ext^\star_{\Fp[\fy_{i+1},\dots]/(\fy_{i+1}^p,\dots)}(\Fp,\Fp)$-module, which is left to the recursion.

The process eventually terminates by Remark \ref{rem:NuiEventuallyVanishes}, completing the proof.
\end{proof}

In the rest of the section, we prove the Corollaries of Theorem \ref{thm:C}.
\begin{proof}[Proof of Corollary \ref{acor:qualitative}]
 The statement is a direct consequence of Theorems \ref{thm:B}, Proposition \ref{prop:bMgsplitting}, the isomorphism $\twoB^{\Fp}_u\cong B^{\Fp}_u$ and the fact that, for any finitely generated $\Fp[\fy_i]/\fy_i^p$-module $N$, $\Ext^\star_{\Fp[\fy_i]/\fy_i^p}(N,\Fp)$ is finitely generated and free as an $\Fp[\beta_i]$-module.
\end{proof}

\begin{proof}[Proof of Corollary \ref{acor:neartop}]
By a standard analysis from homological algebra, for $i\ge0$ and for each finitely generated $\Fp[\fy_i]/(\fy_i^p)$-module $M$, the $\Ext^\star_{\Fp[\fy_i]/(\fy_i^p)}(\Fp,\Fp)$-mo\-du\-le $\Ext^\star_{\Fp[\fy_i]/(\fy_i^p)}(M,\Fp)$ is generated by the vector space
\[
\Ext^{-1\le\star\le 0}_{\Fp[\fy_i]/(\fy_i^p)}(M,\Fp)=\Ext^0_{\Fp[\fy_i]/(\fy_i^p)}(M,\Fp)\oplus\Ext^{-1}_{\Fp[\fy_i]/(\fy_i^p)}(M,\Fp).
\]
We tensor both the ring $\Ext^\star_{\Fp[\fy_i]/(\fy_i^p)}(\Fp,\Fp)$ and the module $\Ext^\star_{\Fp[\fy_i]/(\fy_i^p)}(M,\Fp)$ with the ring $\Ext^\star_{\Fp[\fy_{i+1},\dots]/(\fy_{i+1}^p,\dots)}(\Fp,\Fp)$, and obtain that $\Ext^{-1\le\star\le0}_{\Fp[\fy_i]/(\fy_i^p)}(M,\Fp)$ generates the $\Ext^\star_{\Fp[\fy_i,\dots]/(\fy_i^p,\dots)}(\Fp,\Fp)$-mo\-du\-le $\Ext^\star_{\Fp[\fy_i,\dots]/(\fy_i^p,\dots)}(M,\Fp)$.
We can even consider $\Ext^\star_{\Fp[\fy_i,\dots]/(\fy_i^p,\dots)}(M,\Fp)$ as an $\Ext^\star_\cA(\Fp,\Fp)$-module by letting the ideal $\Ext^{\star<0}_{\Fp[\fy_0,\dots,\fy_{i-1}]/(\fy_0^p,\dots,\fy_{i-1}^p)}(\Fp,\Fp)$ act trivially.
This holds in particular for $M=N_{u,i}$:
therefore, for all $u\ge0$ and $i\ge0$, the vector space $\Ext^{-1\le\star\le0}_{\Fp[\fy_i]/(\fy_i^p)}(N_{u,i},\Fp)$ generates $\Ext^\star_{\Fp[\fy_i,\dots]/(\fy_i^p,\dots)}(N_{u,i},\Fp)$ over the ring $\Ext^\star_{\cA}(\Fp,\Fp)$.

It follows now from Theorem \ref{thm:C} that the vector space $\Ext^{-1\le\star\le0}_{\cA}(B^{\Fp}_u,\Fp)$ generates $\Ext^\star_{\cA}(B^{\Fp}_u,\Fp)$ over $\Ext^\star_{\cA}(\Fp,\Fp)$. For all $g\ge0$ we can now recall Definition \ref{defn:bMg}, and by taking a suitable direct sum we obtain that the vector space $\Ext^{-1\le\star\le0}_{\cA}(\bM_g,\Fp)$ generates $\Ext^\star_{\cA}(\bM_g,\Fp)$ over $\Ext^\star_{\cA}(\Fp,\Fp)$.
Basechanging to $\Fp[\epsilon]$ over $\Fp$, we obtain that $\Ext^{-1\le\star\le0}_{\cA}(\bM_g,\Fp)$ generates $\Fp[\epsilon]\otimes\Ext^\star_{\cA}(\bM_g,\Fp)$ over $\Fp[\epsilon]\otimes\Ext^\star_{\cA}(\Fp,\Fp)$.

Finally, tensoring both the modules and the generating vector space with the vector space $\Fp[\cH_g]=\Fp[H_1(\sgone)]$ over $\Fp$, we obtain that 
$\Ext^{-1\le\star\le0}_{\cA}(\bM_g,\Fp)\otimes\Fp[\cH_g]$ generates $\Fp[\epsilon]\otimes\Ext^\star_{\cA}(\bM_g,\Fp)\otimes\Fp[\cH_g]$ over the ring $\Fp[\epsilon]\otimes\Ext^\star_{\cA}(\Fp,\Fp)$, which by Theorem \ref{thm:B} is identified with $H_*(C_\bullet(D);\Fp)$. We now observe that, according to the decomposition of Theorem \ref{thm:B}, the vector space 
$\Ext^{-1\le\star\le0}_{\cA}(\bM_g,\Fp)\otimes\Fp[\cH_g]$ is concentrated in bigradings $(\bullet,\star)$ satisfying $-1\le\star\le0$; it follows that
\[
 \Ext^{-1\le\star\le0}_{\cA}(\bM_g,\Fp)\otimes\Fp[\cH_g]\subseteq
\bigoplus_{n\ge0}H_{n-1}(C_n(\sgone);\Fp)\oplus\bigoplus_{n\ge0}H_{n}(C_n(\sgone);\Fp).
\]
The statement of the corollary is now a direct consequence of Theorem \ref{thm:B}.
\end{proof}

\begin{proof}[Proof of Corollary \ref{acor:genusvsprime}]
Theorem \ref{thm:B} ensures that $H_*(C_\bullet(\sgone);\Fp)$ is a basechange to $\Fp[\epsilon]\otimes\Fp[\cH_g]$ over $\Fp$ of the $\Ext^\star_{\cA}(\Fp,\Fp)$-module $\Ext^\star_{\cA}(\bM_g,\Fp)$, so it suffices to study the latter. By Definition \ref{defn:bMg}, $\bM_g$ splits as a direct sum of shifted copies of $\cA$-modules of the form $B^{\Fp}_u$, for values of $u$ that are at most $g$.
We then have two cases.
\begin{itemize}
 \item If $g<p^\lambda$, then for all $u\le g$ we have that $N_{u,i}=0$ for $i\ge\lambda+1$, whereas $N_{u,\lambda}$ is a $\cA$-module with trivial action of $\fy_\lambda$, or in other words, $\cB_{\fy_\lambda}(N_{u,\lambda})$ consists solely of bars of size 1: this is because $N_{u,\lambda}$ is supported in weights in the interval $[-2u,0]$, which is strictly narrower than $2p^\lambda$. It follows that $\Ext^\star_{\Fp[\fy_\lambda]/(\fy_\lambda^p)}(N_{u,\lambda},\Fp)$ is a free $\Ext^\star_{\Fp[\fy_\lambda]/(\fy_\lambda^p)}(\Fp,\Fp)$-module, and the claim now follows from Theorem \ref{thm:C}.
 \item If $g<(p-1)p^{\lambda-1}$, then for all $u\le g$ we have that $N_{u,i}=0$ for $i\ge\lambda$, and $N_{u,\lambda-1}$ is an $\Fp[\fy_{\lambda-1}]/(\fy_{\lambda-1}^p)$-module with no free summand, or in other words, $\cB_{\fy_{\lambda-1}}(N_{u,\lambda-1})$ consists solely of bars of size strictly less than $p$: this is because $N_{u,\lambda-1}$ is supported in weights in the interval $[-2u,0]$, which is strictly narrower than $2(p-1)p^{\lambda-1}$. We now use the general fact that for a finitely generated $\Fp[\fy_{\lambda-1}]/(\fy_{\lambda-1}^p)$-module $M$ with no free summand, the module $\Ext^\star_{\Fp[\fy_{\lambda-1}]/(\fy_{\lambda-1}^p)}(M,\Fp)$ is free over the polynomial subalgebra $\Fp[\beta_{\lambda-1}]$ of
 $\Ext^\star_{\Fp[\fy_{\lambda-1}]/(\fy_{\lambda-1}^p)}(\Fp,\Fp)$.
 The claim follows again from Theorem \ref{thm:C}.
\end{itemize}
\end{proof}

\begin{proof}[Proof of Corollary \ref{acor:topminusone}]
By Theorem \ref{thm:B}, $H_*(C_\bullet(\sgone))=H_{\bullet+\star}(C_\bullet(\sgone))$ is the tensor product of $\Ext^\star_{\Gamma_\Z(y)}(\bM_g,\Z)$ with a bigraded abelian group that is free abelian in each
bigrading, namely $\Z[\epsilon]\otimes\Z[\cH_g]$. Both bigraded abelian groups are concentrated in non-negative bar-degrees $\star\ge0$. The part of
$H_*(C_\bullet(\sgone))$ in bar-degree $\star=0$ is therefore a direct sum of shifted copies of $\Ext^0_{\Gamma_\Z(y)}(\bM_g,\Z)$. Similarly, the part of $H_*(C_\bullet(\sgone))$ in bar-degree $\star=-1$ is a direct sum of shifted copies of $\Ext^0_{\Gamma_\Z(y)}(\bM_g,\Z)$ and $\Ext^{-1}_{\Gamma_\Z(y)}(\bM_g,\Z)$. We will therefore prove that $\Ext^0_{\Gamma_\Z(y)}(\bM_g,\Z)$ is a free abelian group, and that $\Ext^{-1}_{\Gamma_\Z(y)}(\bM_g,\Z)$ only has 2-power torsion. 

For the first statement, we note that $\Ext^0_{\Gamma_\Z(y)}(\bM_g,\Z)$ can be computed using the cobar complex
\[
\Hom_{\Gamma_\Z(y)}(\Gamma_\Z(y)^{\otimes\star+1}\otimes\bM_g,\Z),
\]
and in particular $\Ext^0_{\Gamma_\Z(y)}(\bM_g,\Z)$ can be identified with the kernel of the differential
\[
\delta^0\colon\Hom_{\Gamma_\Z(y)}(\Gamma_\Z(y)\otimes\bM_g,\Z)\to\Hom_{\Gamma_\Z(y)}(\Gamma_\Z(y)^{\otimes 2}\otimes\bM_g,\Z),
\]
whose source is a free abelian group.

For the second statement, by Definition \ref{defn:bMg} it suffices to prove that the abelian group $\Ext^{-1}_{\Gamma_\Z(y)}(\twoB^\Z_u,\Z)$ has only $2$-power torsion for $u\ge0$; after inverting 2, we can replace $\twoB^{\Z[\frac12]}_u$ by $B^{\Z[\frac12]}_u$ and aim to prove that $\Ext^{-1}_{\Gamma_{\Z[\frac12]}(y)}(B^{\Z[\frac12]}_u,\Z[\frac12])$ is a free $\Z[\frac12]$-module. We will directly prove that $\Ext^{-1}_{\Gamma_\Z(y)}(B^\Z_u,\Z)$ is a free abelian group.

For this, we use again that $\Ext^\star_{\Gamma_\Z(y)}(B^\Z_u,\Z)$ can be computed as the cohomology of the cobar complex $\Hom_{\Gamma_\Z(y)}(\Gamma_\Z(y)^{\otimes \star+1}\otimes B^\Z_u,\Z)$. We want to argue that the differential
\[
\delta^0\colon\Hom_{\Gamma_\Z(y)}(\Gamma_\Z(y)\otimes B^\Z_u,\Z)\to
 \Hom_{\Gamma_\Z(y)}(\Gamma_\Z(y)^{\otimes 2}\otimes B^\Z_u,\Z)
 \]
has a split image, i.e. $\Imm(\delta^0)$ is a direct summand of the target; once this is proved, we can find a direct complement $C\subset \Hom_{\Gamma_\Z(y)}(\Gamma_\Z(y)^{\otimes 2}\otimes B^\Z_u,\Z)$ to $\Imm(\delta^0)$, and identify $\Ext^{-1}_{\Gamma_\Z(y)}(B^\Z_u,\Z)$ with the kernel of the restriction to $C$ of the differential
\[
\delta^{-1}\colon \Hom_{\Gamma_\Z(y)}(\Gamma_\Z(y)^{\otimes 2}\otimes B^\Z_u,\Z)\to \Hom_{\Gamma_\Z(y)}(\Gamma_\Z(y)^{\otimes 3}\otimes B^\Z_u,\Z),
\]
thus proving that it is a free abelian group.

To prove that $\Imm(\delta^0)$ is a summand, we can consider the $\Z$-dual chain complex, i.e. the bar complex $\Z\otimes_{\Gamma_\Z(y)}\Gamma_\Z(y)^{\otimes\star+1}\otimes B^\Z_u$, and in particular the differential
\[
 d_1\colon\Z\otimes_{\Gamma_\Z(y)}\Gamma_\Z(y)^{\otimes 2}\otimes B^\Z_u\to
 \Z\otimes_{\Gamma_\Z(y)}\Gamma_\Z(y)\otimes B^\Z_u;
\]
it suffices to prove that $\Imm(d_1)$ is a summand. Since the source and the target are free abelian groups, this is equivalent to proving that $\Tor_0^{\Gamma_\Z(y)}(B^\Z_u,\Z)$ is a free abelian group; this last statement follows from Proposition \ref{prop:Psiukinj}, identifying the latter with the free abelian group spanned by all elements $z_S$ for $S\subset\set{1,\dots,u}$ sparse.
This completes the proof of the second statement.

The equality $\dim_\Q(H_n(C_n(\sgone);\Q)) = 
  \dim_{\Fp}(H_n(C_n(\sgone);\Fp))$ is now a direct consequence of the universal coefficient theorem for homology.
\end{proof}

\begin{proof}[Proof of Corollary \ref{acor:ppowertorsion}]
For $n\ge0$, the homology $H_*(C_n(\sgone);\Z)$ can be computed as the homology of a finitely generated chain complex of free abelian groups; such a chain complex can be decomposed as a direct sum of shifted copies of atomic chain complexes $0\to\Z\to0$
and $0\to\Z\overset{\cdot m}{\to}\Z\to0$ with $m\ge1$.
A standard analysis of the behaviour of each type of atomic chain complex after tensoring with the short exact sequence $\Z\overset{\cdot p}{\to}\Z\twoheadrightarrow\Fp$ leads to the equivalence of the following two statements:
\begin{enumerate}
 \item for $n\ge0$, all $p$-power torsion classes in $H_*(C_n(\sgone);\Z)$ are $p$-torsion;
 \item for $n\ge0$, all $p$-power torsion classes in $H_*(C_n(\sgone);\Z)$ are in the image of the Bockstein homomorphism $H_{*+1}(C_n(\sgone);\Fp)\to H_*(C_n(\sgone);\Z)$.
\end{enumerate}
Thus, in order to prove Corollary \ref{acor:ppowertorsion}, it suffices to show (1). By Theorem \ref{thm:B} (1) is equivalent to the following statement (restricted to $0\le u\le g$):
\begin{enumerate}
\setcounter{enumi}{2}
 \item for $u\ge0$, all $p$-power torsion classes in $\Ext^\star_{\Gamma_\Z(y)}(B^\Z_u,\Z)$ are $p$-torsion.
\end{enumerate}
We next prove that (3) follows from the following statement:
\begin{enumerate}
\setcounter{enumi}{3}
 \item for $u\ge0$, all summands in $\Ext^\star_{\Gamma_{\Z/p^2}(y)}(B^{\Z/p^2}_u,\Z/p^2)$ of the form $\Z/p^2$ are in bar-degrees $-1\le\star\le0$.
\end{enumerate}
We use as in the proof of Corollary \ref{acor:topminusone} that $\Ext^\star_{\Gamma_\Z(y)}(B^\Z_u,\Z)$ is the cohomology of the cochain complex of free abelian groups $\Hom_{\Gamma_\Z(y)}(\Gamma_\Z(y)^{\otimes \star+1}\otimes B^\Z_u,\Z)$. Tensoring this cochain complex with $\Z/p^2$, we obtain the analogue cochain complex computing $\Ext^\star_{\Gamma_{\Z/p^2}(y)}(B^{\Z/p^2}_u,\Z/p^2)$. A standard analysis involving the universal coefficients theorem implies that for each of the following triples $(A,B,C)$, a summand of type $A$ in $\Ext^i_{\Gamma_\Z(y)}(B^\Z_u,\Z)$ gives rise to a summand of type $B$ in $\Ext^i_{\Gamma_{\Z/p^2}(y)}(B^{\Z/p^2}_u,\Z/p^2)$ and a summand of type $C$ in $\Ext^{i+1}_{\Gamma_{\Z/p^2}(y)}(B^{\Z/p^2}_u,\Z/p^2)$:
\[
(A,B,C) =\quad (\Z,\Z/p^2,0),\quad (\Z/p,\Z/p,\Z/p),\quad (\Z/p^l,\Z/p^2,\Z/p^2)\mbox{ with }l\ge2.
\]
Combining with the observation that 
$\Ext^{-1\le\star\le0}_{\Gamma_\Z(y)}(B^\Z_u,\Z)$ contains no $p$-power torsion, we deduce that (3) follows from (4).

We next take the Bockstein long exact sequence obtained by tensoring the cochain complex $\Hom_{\Gamma_\Z(y)}(\Gamma_\Z(y)^{\otimes \star+1}\otimes B^\Z_u,\Z)$ with the short exact sequence $\Fp\to\Z/p^2\to\Fp$. Let $\fb\colon\Ext^\star_{\Gamma_{\Fp}(y)}(B^{\Fp}_u,\Fp)\to\Ext^{\star+1}_{\Gamma_{\Fp}(y)}(B^{\Fp}_u,\Fp)$ denote the Bockstein homomorphism; we regard $\fb$ a differential on $\Ext^\star_{\Gamma_{\Fp}(y)}(B^{\Fp}_u,\Fp)$ and consider the $\fb$-cohomology of the latter.
As before, we can decompose the cochain complex
$\Hom_{\Gamma_{\Z/p^2}(y)}(\Gamma_{\Z/p^2}(y)^{\otimes \star+1}\otimes_{\Z/p^2} B^{\Z/p^2}_u,\Z/p^2)$ as a direct sum of shifts of the atomic complexes $0\to\Z/p^2\to0$, $0\to\Z/p^2\overset{\cdot p}{\to}\Z/p^2\to0$ and $0\to\Z/p^2\overset{\cdot 1}{\to}\Z/p^2\to0$. By a standard analysis, (4) is equivalent to the following statement:
\begin{enumerate}
\setcounter{enumi}{4}
 \item for $u\ge0$, the $\fb$-cohomology of $\Ext^\star_{\Gamma_{\Fp}(y)}(B^{\Fp}_u,\Fp)$ is concentrated in bar-degrees $-1\le\star\le0$.
\end{enumerate}

Denote by $-\cdot-$ the Yoneda product: then for all $\gamma\in \Ext^i_{\Gamma_{\Fp}(y)}(\Fp,\Fp)$ and $\eta\in\Ext^j_{\Gamma_{\Fp}(y)}(\Fp,\Fp)$ we have the equality $\fb(\gamma\cdot\eta)=\fb(\gamma)\cdot\eta+(-1)^i\gamma\cdot \fb(\eta)$, which we call the ``Leibniz rule'' for the Bockstein homomorphism; a similar rule holds when $\eta\in\Ext^j_{\Gamma_{\Fp}(y)}(B^{\Fp}_u,\Fp)$.
A standard computation gives $\fb(\alpha_{i+1})=\beta_i$ and $\fb(\beta_i)=0$
for all $i\ge0$; by the Leibniz rule, we get that the Bockstein homomorphism on $\Ext^\star_{\Gamma_{\Fp}(y)}(B^{\Fp}_u,\Fp)$ is a map of $\Fp[\beta_0,\beta_1,\dots]$-modules.

Let $S_i$ be the $i$\sth summand of $\Ext^\star_{\Gamma_{\Fp}(y)}(B^{\Fp}_u;\Fp)$ as in Theorem \ref{thm:C}:
\[ 
S_i=\Fp\otimes_{\Fp}\Ext^\star_{\Fp[\fy_i]/(\fy_i^p)}(N_{u,i},\Fp)\otimes_{\Fp}\bigotimes_{j=i+1}^\infty \Ext^\star_{\Fp[\fy_j]/(\fy_j^p)}(\Fp,\Fp).
\]
Then $\beta_j$ annihilates $S_i$ for all $0\le j\le i-1$, whereas $\beta_j$ acts injectively on $S_i$ for $i\le j\le \infty$. It follows that the Bockstein homomorphism $\fb$ on $\Ext^\star_{\Gamma_{\Fp}(y)}(B^{\Fp}_u,\Fp)$ sends $S_i$ to the direct sum $\bigoplus_{j=i}^\infty S_j$. Therefore the direct sums $\bigoplus_{j=i}^\infty S_j$, for varying $i\ge0$, give a descending filtration of the Bockstein cochain complex $\Ext^\star_{\Gamma_{\Fp}(y)}(B^{\Fp}_u,\Fp)$.
We denote by $\bar S_i=\pa{\bigoplus_{j=i}^\infty S_j}/\pa{\bigoplus_{j=i+1}^\infty S_j}$ the $i$\sth filtration quotient, and by $\bar\fb$ its differential. Then (5) is implied by the following statement:
\begin{enumerate}
\setcounter{enumi}{5}
 \item for $u\ge0$ and $i\ge0$, the $\fb$-cohomology of $\bar S_i$ is concentrated in bar-degrees $-1\le\star\le0$.
\end{enumerate}
We can regard $\bar S_i$ as a differential bigraded module over the differential bigraded algebra $\Fp[\beta_i,\alpha_{i+1},\beta_{i+1},\dots]$, which we endow with the differential $\bar\fb$ defined by $\bar\fb(\alpha_{j+1})=\beta_j$ and $\bar\fb(\beta_j)=0$ for $j\ge i$. Note that $\Fp[\beta_i,\alpha_{i+1},\beta_{i+1},\dots]$ is $\Fp$ in bigrading (0,0), it is concentrated in bigradings $\bullet\ge0$ and $\star\le 0$, and its module $\bar S_i$ is freely generated by the bigraded vector space $\Ext^{-1\le\star\le0}_{\Fp[\fy_i]/(\fy_i^p)}(N_{u,i},\Fp)$.

We define a monomial filtration on $\bar S_i$: for all $j\ge0$, we put in filtration at least $j$ all elements that can be written as the product of a class in $\Ext^{-1\le\star\le0}_{\Fp[\fy_i]/(\fy_i^p)}(N_{u,i},\Fp)$ and a monomial of at least $j$ variables (counted with repetition) chosen among $\beta_i,\alpha_{i+1},\beta_{i+1},\dots$. For all $j\ge0$, the part of $\bar S_i$ of filtration at least $j$ is a sub-cochain complex, and (6) is implied by the following statement
\begin{enumerate}
\setcounter{enumi}{6}
 \item for $u\ge0$, $i\ge0$, the associated graded of $(\bar S_i,\bar\fb)$ with respect to the monomial filtration has cohomology only in bar-degrees $-1\le\star\le0$.
\end{enumerate}
For the last statement, we notice that the associated graded of $\bar S_i$ with respect to the monomial filtration is isomorphic, as a bigraded cochain complex, to the tensor product of the vector space $\Ext^{-1\le\star\le0}_{\Fp[\fy_i]/(\fy_i^p)}(N_{u,i},\Fp)$, with trivial differential, and the algebra $\Fp[\beta_i,\alpha_{i+1},\beta_{i+1},\dots]$, with  differential $\bar\fb$. The $\bar\fb$-cohomology of $\Fp[\beta_i,\alpha_{i+1},\beta_{i+1},\dots]$ is $\Fp$, sitting in bar-degree 0, and (7) follows.
\end{proof}

\begin{proof}[Proof of Corollary \ref{acor:basechange}]
By Theorem \ref{thm:B} it suffices to prove that for any $u\le p-2$ there
is an isomorphism of bigraded $\Fp$-vector spaces
\[
\Ext^\star_{\Gamma_{\Fp}(y)}(B^{\Fp}_u,\Fp)\cong\pa{\Ext^\star_{\Gamma_\Q(y)}(B^\Q_u,\Q)}^\Q_{\Fp}\otimes_{\Fp}\Fp[\beta_0,\alpha_1,\beta_1,\alpha_2,\beta_2\dots];
\]
applying part (2) of Corollary \ref{acor:genusvsprime} with $\lambda=1$, it suffices to prove an isomorphism
\[
\Ext^\star_{\Fp[\fy_0]/(\fy_0^p)}(B^{\Fp}_u,\Fp)\cong\pa{\Ext^\star_{\Gamma_\Q(y)}(B^\Q_u,\Q)}^\Q_{\Fp}\otimes_{\Fp}\Fp[\beta_0].
\]
The barcode of $B^\Q_u$ as a module over $\Gamma_\Q(y)\cong\Q[y]$ has the following description: for all $1\le k\le \floor{u/2}$ there are $\ell_{u,k}$ bars of the form $(-2k,u+1-2k)$. The same description holds for the barcode of $B^{\Fp}_u$ over $\Fp[\fy_0]/(\fy_0^p)$: since $u\le p-2$, the free-narrow decomposition of $B^{\Fp}_u$ has only a narrow part.
It is then standard to prove, for $1\le c\le p-1$, that there is an isomorphism of bigraded vector spaces
\[
 \Ext^\star_{\Fp[\fy_0]/(\fy_0^p)}(\Fp[\fy_0]/(\fy_0^c),\Fp)\cong\pa{\Ext^\star_{\Q[y]}(\Q[y]/(y^c),\Q)}^\Q_{\Fp}\otimes_{\Fp}\Fp[\beta_0].
\]
\end{proof}

\begin{proof}[Proof of Corollary \ref{acor:lowdimbetti}]
Using Theorem \ref{thm:B} and Proposition \ref{prop:bMgsplitting}, the claim is implied by the following statement: $\Ext^\star_{\Gamma_\Z(y)}(B^\Z_u,\Z)$ has no $p$-torsion in homological degree $*=\bullet+\star<\max\set{2p-2,u+p}$.
As in the proof of Corollary \ref{acor:ppowertorsion}, the $p$-torsion in $\Ext^\star_{\Gamma_\Z(y)}(B^\Z_u,\Z)$ coincides with the image of the Bockstein homomorphism $\fb^\Z\colon \Ext^{\star+1}_{\Gamma_\Z(y)}(B^\Z_u,\Fp)\to\Ext^{\star}_{\Gamma_\Z(y)}(B^\Z_u,\Z)$, and the natural map
$\Ext^\star_{\Gamma_\Z(y)}(B^\Z_u,\Z)\to \Ext^\star_{\Gamma_\Z(y)}(B^\Z_u,\Fp)$ is injective when restricted to the image of $\fb^\Z$. Denoting by
$$\fb\colon \Ext^{\star+1}_{\Gamma_\Z(y)}(B^\Z_u,\Fp)\to\Ext^{\star}_{\Gamma_\Z(y)}(B^\Z_u,\Fp)$$ 
the composite of $\fb^\Z$ and the aforementioned natural map (i.e., the Bockstein homomorphism from the proof of Corollary \ref{acor:ppowertorsion}), we have to prove that the image of $\fb$ is concentrated in homological degrees $*=\bullet+\star\ge\max\set{2p-2,u+p}$.

We observe that $\Ext^{\star}_{\Gamma_\Z(y)}(B^\Z_u,\Fp)$ is canonically isomorphic to $\Ext^{\star}_{\Gamma_{\Fp}(y)}(B^{\Fp}_u,\Fp)$ and so we replace the former by the latter.
By Theorem \ref{thm:C}, the vector space $\Ext^{-1}_{\Gamma_{\Fp}(y)}(B^{\Fp}_u,\Fp)$ can be identified with the direct sum $E\oplus E'$, where
\[
\begin{split}
 E=&\bigoplus_{i=0}^h\Ext^{-1}_{\Fp[\fy_i]/(\fy_i^p)}(N_{u,i},\Fp);\\
 E'=&\bigoplus_{i=0}^h\Ext^0_{\Fp[\fy_i]/(\fy_i^p)}(N_{u,i},\Fp)\otimes_{\Fp}\pa{\bigoplus_{j=i}^{\infty}\Fp\cdot\alpha_{j+1}}.
\end{split}
\]
According to Corollary \ref{acor:neartop}, $\Ext^{\star}_{\Gamma_{\Fp}(y)}(B^{\Fp}_u,\Fp)$ is generated by the direct sum $\Ext^{0}_{\Gamma_{\Fp}(y)}(B^{\Fp}_u,\Fp)\oplus E\subset\Ext^{-1\le\star\le0}_{\Gamma_{\Fp}(y)}(B^{\Fp}_u,\Fp)$ as a module over $\Ext^{\star}_{\Gamma_{\Fp}(y)}(\Fp,\Fp)$.

We proceed to analyse the behaviour of $\fb$ on $\Ext^{-1\le \star \le 0}_{\Gamma_{\Fp}(y)}(B^{\Fp}_u,\Fp)$. By Corollary \ref{acor:topminusone}, $\fb$ vanishes on $\Ext^0_{\Gamma_{\Fp}(y)}(B^{\Fp}_u,\Fp)$, since the natural map $\Ext^0_{\Gamma_\Z(y)}(B^\Z_u,\Z)\to\Ext^0_{\Gamma_{\Fp}(y)}(B^{\Fp}_u,\Fp)$ is surjective. Instead $\fb$ is injective on $E'$: since $\fb(\alpha_{j+1})=\beta_j$, we may identify the restriction of $\fb$ on $E'$ with the injective composition of maps
\[
\begin{split}
 E'&=\bigoplus_{i=1}^h\Ext^0_{\Fp[\fy_i]/(\fy_i^p)}(N_{u,i},\Fp)\otimes_{\Fp}\pa{\bigoplus_{j=i}^{\infty}\Fp\cdot\alpha_{j+1}}\cong\\
 &\cong\bigoplus_{i=1}^h\Ext^0_{\Fp[\fy_i]/(\fy_i^p)}(N_{u,i},\Fp)\otimes_{\Fp}\pa{\bigoplus_{j=i}^{\infty}\Fp\cdot\beta_{j}}\subseteq \Ext^{-2}_{\Gamma_{\Fp}(y)}(B^{\Fp}_u,\Fp).
\end{split}
\]
Finally, let $\fb_{-1}$ denote the map $\fb\colon\Ext^{-1}_{\Gamma_{\Fp}(y)}(B^{\Fp}_u,\Fp)\to\Ext^{-2}_{\Gamma_{\Fp}(y)}(B^{\Fp}_u,\Fp)$; then $\ker(\fb_{-1})$ has the same total dimension as $E$. For this, first note that
\[
\dim_{\Fp}E=\dim_{\Fp}\Ext^0_{\Gamma_{\Fp}(y)}(B^{\Fp}_u,\Fp)=\binom{u}{\floor{u/2}},
\]
as we have ungraded isomorphisms
\[
E=\bigoplus_{i=0}^h\Ext^{-1}_{\Fp[\fy_i]/(\fy_i^p)}(N_{u,i},\Fp)
\cong \bigoplus_{i=0}^h\Ext^0_{\Fp[\fy_i]/(\fy_i^p)}(N_{u,i},\Fp)\cong \Ext^0_{\Gamma_{\Fp}(y)}(B^{\Fp}_u,\Fp).
\]
Similarly, the total rank of $\Ext^{-1}_{\Gamma_\Z(y)}(B^\Z_u,\Z)$ is both equal to
\[
\dim_\Q\Ext^{-1}_{\Gamma_\Q(y)}(B^\Q_u,\Q)=\dim_\Q\Ext^0_{\Gamma_\Q(y)}(B^\Q_u,\Q)=\binom{u}{\floor{u/2}},
\]
and to $\dim_{\Fp}\ker(\fb_{-1})$, by Corollary \ref{acor:topminusone}).

Combining the previous remarks, we conclude that $\Ext^{-1}_{\Gamma_{\Fp}(y)}(B^{\Fp}_u,\Fp)$ also decomposes as the direct sum $E'\oplus\ker(\fb_{-1})$. In particular, every element in $E$ can be written as a sum of an element in $E'$ and an element in $\ker(\fb)$. 

We will now estimate the homological degrees in which a non-zero element in the image of $\fb\colon\Ext^{\star}_{\Gamma_{\Fp}(y)}(B^{\Fp}_u,\Fp)\to \Ext^{\star-1}_{\Gamma_{\Fp}(y)}(B^{\Fp}_u,\Fp)$ may occur. We will consider two types of additive generators of $\Ext^{\star}_{\Gamma_{\Fp}(y)}(B^{\Fp}_u,\Fp)$: the first type comes from $\Ext^{0}_{\Gamma_{\Fp}(y)}(B^{\Fp}_u,\Fp)$, and the second type from $E$.

Fix $0\le i\le h$, let $x\in\Ext^0_{\Fp[\fy_i]/(\fy_i^p)}(N_{u,i},\Fp)$ and let $s$ be a bihomogeneous element of bidegree $(\bullet_s,\star_s)$ in $\Fp[\beta_i,\alpha_{i+1},\beta_{i+1},\dots]\subset\Ext^\star_{\Gamma_{\Fp}(y)}(\Fp,\Fp)$; then $s\cdot x$ is a generator of first type of $\Ext^\star_{\Gamma_{\Fp}(y)}(B^{\Fp}_u,\Fp)$, and
$\fb(s\cdot x)=\fb(s)\cdot x$. Assuming that $\fb(s)$ is non-zero, we have that $s$ belongs to the ideal generated by the variables $\alpha_j$ for $j\ge i+1$; in particular $\star_s\le-1$ and $\bullet_s\ge p^{i+1}(-\star_s+1)$. Using that $N_{u,i}$ is supported in weights less than $-u+p^{i+1}-1$, we have that $x$ has weight $\bullet_x>u-p^{i+1}+1$. Thus $\fb(s)\cdot x$ has bar-degree $\star_s-1$ and weight $\bullet_x+\bullet_s> u-p^{i+1}+1+p^{i+1}(-\star_s+1)$. It follows that the homological degree of $\fb(s)\cdot x$ is strictly larger than 
$u+p-1$. Replacing the estimate  $\bullet_x>u-p^{i+1}+1$ with the estimate  $\bullet_x>-1$, we obtain that the homological degree of $\fb(s)\cdot x$ is also strictly larger than $2p-3$.

Fix now $0\le i\le h$, let $x\in\Ext^{-1}_{\Fp[\fy_i]/(\fy_i^p)}(N_{u,i},\Fp)$ and let $s$ be as above; then $s\cdot x$ is a generator of second type of $\Ext^\star_{\Gamma_{\Fp}(y)}(B^{\Fp}_u,\Fp)$. We can decompose $x$ inside $\Ext^{-1}_{\Gamma_{\Fp}(y)}(B^{\Fp}_u,\Fp)$ as a sum $x'+x''$, with $x'$ and $x''$ having the same weight as $x$, with $x'\in\ker(\fb)$, and with $x''\in E'$, in particular $x''$ lies in the image of the multiplication map
\[
\Ext^{-1}_{\Gamma_{\Fp}(y)}(\Fp,\Fp)\otimes_{\Fp}\Ext^0_{\Gamma_{\Fp}(y)}(B^{\Fp}_u,\Fp)\to\Ext^{-1}_{\Gamma_{\Fp}(y)}(B^{\Fp}_u,\Fp).
\]
We get a splitting $\fb(s\cdot x)=\fb(s)\cdot x'+\fb(s\cdot x'')$. The second term vanishes unless it lives in homological degree at least $\max\set{u+p,2p-2}$, by the previous argument. The first term can also be treated as in the previous argument: we now have $\star_x=-1$ instead of 0, but we also have the better estimate $\bullet_x>\max\set{u-p^{i+1}+1,-1}+2p^i$.
\end{proof}
\begin{rem}
The fact that $\Ext^{-1}_{\Gamma_{\Fp}(y)}(B^{\Fp}_u,\Fp)$ can be decomposed as the direct sum of $E'$ and $\ker(\fb)$ gives a slight improvement of the statement of Corollary \ref{acor:neartop}: $H_*(C_\bullet(\sgone);\Fp)$ is generated over $H_*(C_\bullet(D);\Fp)$ by the direct sum of 
\[
\bigoplus_{n\ge0}H_n(C_n(\sgone);\Fp)\ \mbox{ and }\ \ker(\fb)\subset \bigoplus_{n\ge0} H_{n-1}(C_n(\sgone);\Fp).
\]
\end{rem}

\begin{proof}[Proof of Corollary \ref{acor:extremalquasistab}]
We fix $g\ge1$ and $i\ge0$ throughout the proof. A direct consequence of Theorems \ref{thm:B} and \ref{thm:C} is that $H_{n-i}(C_n(\sgone);\Fp)\neq0$ as soon as $n\ge i$; hence it suffices to find constants $0<\bfc_{g,i}<\bfC_{g,i}$ that work for $n$ large enough.

We start by proving the existence of $\bfc_{g,i}$. Let $n=2m+\delta$ with $\delta\in\set{0,1}$, and let $\xi\colon \set{1,\dots,g}\to\bK$ (recall Notation \ref{nota:K}) be the map defined by
\[
\xi(j)=\left\{
\begin{array}{cl}
\be & \mbox{ if } j\le g-1 \mbox{ or }\delta=0;\\
\bu & \mbox{ else}.
\end{array}
\right.
\]

By Theorem \ref{thm:B} and Definition \ref{defn:bMg}, together with the isomorphism $B^p_{|\xi|}\cong\twoB^p_{|\xi|}$, we have that
$H_*(C_\bullet(\sgone);\Fp)$ contains a copy of $\Ext^\star_{\Gamma_{\Fp}(y)}(B^p_{g-\delta},\Fp)[\delta]\otimes_{\Fp}\Fp[\cH_g]$; we will estimate the dimension of the latter in bigrading $(\bullet,\star)=(n,-i)$; setting $u=g-\delta$, and using $\delta=n-2m$, we get that $\dim_{\Fp} H_{n-i}(C_n(\sgone);\Fp)$ is at least the dimension in bigrading $(2m,-i)$ of $\Ext^\star_{\Gamma_{\Fp}(y)}(B^{\Fp}_u,\Fp)\otimes_{\Fp}\Fp[\cH_g]$.

Let now $l:=\floor{\log_p(m/(i+1))}\asymp \log_p(n)$,
and let $h:=\floor{\log_p(u)}$.
Fix a non-trivial class $\gamma\in\Ext^0_{\Gamma_{\Fp}(y)}(B^{\Fp}_u,\Fp)$. For every subset $S\subset\set{h+1,\dots,l}$ of cardinality $i$ we can construct a non-trivial class $(\prod_{s\in S}\alpha_s)\cdot\gamma\in \Ext^{-i}_{\Gamma_{\Fp}(y)}(B^{\Fp}_u,\Fp)$ of weight $w(S):=\sum_{s\in S}2p^s\le 2ip^l\le 2m/(i+1)$. We can ``complete'' this class by a monomial in $\Fp[\cH_g]$ of weight $2m-w(S)\ge m/(i+1)$, to get a class in $\Ext^\star_{\Gamma_{\Fp}(y)}(B^{\Fp}_u,\Fp)\otimes_{\Fp}\Fp[\cH_g]$ of bigrading $(2m,-i)$: such monomial can be chosen in $\binom{m-w(S)/2+2g-1}{2g-1}\asymp n^{2g-1}$ ways.
Since there are $\binom{l-h}{i}\asymp \log_p(n)^i$ ways to choose the subset $S$, we have produced $\asymp n^{2g-1}\log_p(n)^i$ linearly independent classes in $H_{n-i}(C_n(\sgone);\Fp)$. This proves the existence of $\bfc_{g,i}$.

We next prove the existence of $\bfC_{g,i}$. By Theorem \ref{thm:B}, the weighted vector space $H_{\bullet-i}(C_\bullet(\sgone);\Fp)$ can be expressed as a finite direct sum of shifted copies of $\Ext^{-j}_{\Gamma_{\Fp}(y)}(B^{\Fp}_u,\Fp)\otimes_{\Fp}\Fp[\cH_g]$ for values of $j$ and $u$ satisfying $0\le j\le i$ and $u \le g$: multiplication by $\epsilon^{i-j}$ puts such a copy in bar-degree $-i$. The number of these summands is bounded above by $2^g\cdot i$, so we can focus on the growth of the weighted dimension of $\Ext^{-j}_{\Gamma_{\Fp}(y)}(B^{\Fp}_u,\Fp)\otimes_{\Fp}\Fp[\cH_g]$, for fixed $0\le j\le i$ and $u\ge0$. 

By Theorem \ref{thm:C}, $\Ext^{-j}_{\Gamma_{\Fp}(y)}(B^{\Fp}_u,\Fp)\otimes_{\Fp}\Fp[\cH_g]$ further decomposes as a finite direct sum of suitable bigraded shifts of quotients of $\Fp[\alpha_0,\beta_0,\alpha_1,\beta_1,\dots]\otimes_{\Fp}\Fp[\cH_g]$, so we can fix $j\le i$ and study the growth of the weighted dimension of the part of $\Fp[\alpha_0,\beta_0,\dots]\otimes_{\Fp}\Fp[\cH_g]$ of bar-degree $-j$. Since $\Fp[\alpha_0,\beta_0,\dots]\otimes_{\Fp}\Fp[\cH_g]$ is concentrated in even weights, and we are looking for an upper bound on the dimension of the part of bigrading $(n,-j)$, for fixed $j$ and varying $n$, of the form $\bfC_{g,i}\log_p(n)^in^{2g-1}$, we can focus on the even case $n=2m$.

In order to form a monomial of $\Fp[\alpha_0,\beta_0,\dots]\otimes_{\Fp}\Fp[\cH_g]$ of bigrading $(2m,-j)$, we need to
choose a monomial in $\Fp[\alpha_0,\beta_0,\dots]$ of some weight $2w\le 2m$ and of bar-degree $-j$, and complete it by a suitable monomial in $\Fp[\cH_g]$ of weight $2m-2w$ and bar-degree $0$. Letting $l=\floor{\log_p(m)}$, only the variables $\alpha_0,\beta_0,\dots,\alpha_l$ can be used for the first choice; an upper bound to the number of monomials is thus the number of \emph{ordered} sequences of $j$ letters in the alphabet $\alpha_0,\beta_0,\dots,\alpha_l$, which is $(2l+1)^j\asymp\log_p(n)^j=O(\log_p(n)^i)$. For each choice of a monomial of weight $2w\le 2m$, we can then choose a monomial of weight $2m-2w\le 2m$ in $\Fp[\cH_g]$ in at most $\binom{m+2g-1}{2g-1}\asymp n^{2g-1}$ ways. Thus the weighted dimension of $\Fp[\alpha_0,\beta_0,\dots]\otimes_{\Fp}\Fp[\cH_g]$ in bar-degree $-j$ is $O(n^{2g-1}\log_p (n)^i)$.
This proves the existence of $\bfC_{g,i}$.
\end{proof}

\bibliography{bibliography.bib}{}

\begin{thebibliography}{10}

\bibitem{Arnold}
Vladimir~I. Arnold.
\newblock On some topological invariants of algebraic functions.
\newblock {\em Trudy Moskovskogo matematicheskogo obshchestva}, 21:27–46,
  1970.

\bibitem{Bianchi}
Andrea Bianchi.
\newblock Splitting of the homology of the punctured mapping class group.
\newblock {\em Journal of Topology}, 13:1230--1260, 2020.

\bibitem{BMW}
Andrea Bianchi, Jeremy Miller, and Jennifer Wilson.
\newblock Mapping class group actions on configuration spaces and the {J}ohnson
  filtration.
\newblock {\em Transactions of the American Mathematical Society},
  375(8):5461--5489, 2022.

\bibitem{BC}
Carl-Friedrich B{\"o}digheimer and {Frederick R.} Cohen.
\newblock Rational cohomology of configuration spaces of surfaces.
\newblock In Tammo tom Dieck, editor, {\em Algebraic Topology and
  Transformation Groups}, pages 7--13. Springer Berlin Heidelberg, 1988.

\bibitem{BCT}
Carl-Friedrich B\"{o}digheimer, Frederick~R. Cohen, and Laurence~R. Taylor.
\newblock On the homology of configuration spaces.
\newblock {\em Topology}, 28:111--123, 1989.

\bibitem{BHK}
Lukas Brantner, Jeremy Hahn, and Ben Knudsen.
\newblock The {L}ubin-{T}ate theory of configuration spaces: {I}, 2019.
\newblock \href{https://arxiv.org/abs/1908.11321}{arXiv:1908.11321}.

\bibitem{Callegaro}
Filippo Callegaro.
\newblock The homology of the {M}ilnor fiber for classical braid groups.
\newblock {\em Algebraic \& Geometric Topology}, 6:1903--1923, 2006.

\bibitem{ChenZhang}
Matthew Chen and Adela~Y. Zhang.
\newblock Mod $p$ homology of unordered configuration spaces of $p$ points in
  parallelizable surfaces.
\newblock {\em Proceedings of the American Mathematical Society},
  152:2239--2248, 2024.

\bibitem{CLM}
Frederick~R. Cohen, Thomas~J. Lada, and J.~Peter May.
\newblock {\em The Homology of Iterated Loop Spaces}.
\newblock Springer, Lecture Notes in Mathematics 533, 1976.

\bibitem{Cooper}
James Cooper.
\newblock Two mod-p {J}ohnson filtrations.
\newblock {\em Journal of Topology and Analysis}, 7(2):309--343, 2015.

\bibitem{DCK}
{Gabriel C.} Drummond-Cole and Ben Knudsen.
\newblock Betti numbers of configuration spaces of surfaces.
\newblock {\em Journal of the London Mathematical Society}, 96:367--393, 2017.

\bibitem{ETW}
Jordan Ellenberg, TriThang Tran, and Craig Westerland.
\newblock {F}ox--{N}euwirth--{F}uks cells, quantum shuffle algebras, and
  {M}alle's conjecture for function fields, 2023.
\newblock \href{https://arxiv.org/abs/1701.04541}{arXiv:1701.04541}.

\bibitem{FelixThomas}
Yves F\'elix and Jean-Claude Thomas.
\newblock Rational {B}etti numbers of configuration spaces.
\newblock {\em Topology and its Applications}, 102(2):139--149, 2000.

\bibitem{Fuchs}
Dmitry~B. Fuchs.
\newblock Cohomologies of the braid group mod 2.
\newblock {\em Functional Analysis and its Applications}, 4:2:143--151, 1970.

\bibitem{Hoang}
Anh T.~N. Hoang.
\newblock Fox--{N}euwirth cells, quantum shuffle algebras, and the homology of
  type-{B} {A}rtin groups.
\newblock {\em Mathematische Zeitschrift}, 306:57, 2024.

\bibitem{JohnsonI}
Dennis Johnson.
\newblock An abelian quotient of the mapping class group $\mathscr{I}_g$.
\newblock {\em Mathematische Annalen}, 249:225--242, 1980.

\bibitem{Knudsen:BettiNumbers}
Ben Knudsen.
\newblock Betti numbers and stability for configuration spaces via
  factorization homology.
\newblock {\em Algebraic and Geometric Topology}, 17:3137--3187, 2017.

\bibitem{Knudsen:HEA}
Ben Knudsen.
\newblock Higher enveloping algebras.
\newblock {\em Geometry \& Topology}, 22(7):4013--4066, 2018.

\bibitem{KMT}
Ben Knudsen, Jeremy Miller, and Philip Tosteson.
\newblock Extremal stability for configuration spaces.
\newblock {\em Mathematische Annalen}, 386:1695--1716, 2023.

\bibitem{Looijenga}
Eduard Looijenga.
\newblock Torelli group action on the configuration space of a surface.
\newblock {\em Journal of Topology and Analysis}, 15(1):215--222, 2023.

\bibitem{Markaryan}
Nikita~S. Markaryan.
\newblock Cohomology of braid groups with nontrivial coefficients.
\newblock {\em Mathematical Notes}, 59:611--617, 1996.

\bibitem{McDuff}
Dusa McDuff.
\newblock Configuration spaces of positive and negative particles.
\newblock {\em Topology}, 14:91--107, 1975.

\bibitem{Moriyama}
Tetsuhiro Moriyama.
\newblock The mapping class group action on the homology of the configuration
  spaces of surfaces.
\newblock {\em Journal of the London Mathematical Society}, 76(2):451--466,
  2007.

\bibitem{Perron}
Bernard Perron.
\newblock Filtration de {J}ohnson et groupe de {T}orelli modulo p, p premier.
\newblock {\em Comptes Rendus Mathematique}, 346(11):667--670, 2008.

\bibitem{ORW}
Oscar Randal-Williams.
\newblock Homological stability for unordered configuration spaces.
\newblock {\em Quarterly Journal of Mathematics}, 64(1):303--326, 2013.

\bibitem{Segal}
Graeme Segal.
\newblock The topology of spaces of rational functions.
\newblock {\em Acta Mathematica}, 143:39--72, 1979.

\bibitem{Stavrou}
Andreas Stavrou.
\newblock Cohomology of configuration spaces of surfaces as mapping class group
  representations.
\newblock {\em Transactions of the American Mathematical Society},
  376:2821--2852, 2023.

\bibitem{Weinstein}
Felix~V. Weinstein.
\newblock Cohomology of the braid groups.
\newblock {\em Functional Analysis and its Applications}, 12:2:135--137, 1978.

\bibitem{Zassenhaus}
Hans Zassenhaus.
\newblock Ein {V}erfahren, jeder endlichen p-{G}ruppe einen {L}ie-{R}ing mit
  der {C}harakteristik p zuzuordnen.
\newblock {\em Abhandlungen aus dem Mathematischen Seminar der Universit\"at
  Hamburg}, 13:200--207, 1939.

\bibitem{Zhang}
Adela~Y. Zhang.
\newblock Quillen homology of spectral {L}ie algebras with application to mod
  $p$ homology of labeled configuration spaces, 2021.
\newblock To appear in Algebraic and Geometric Topology.
  \href{https://arxiv.org/abs/2110.08428}{arXiv:2110.08428}.

\end{thebibliography}
\bibliographystyle{plain}

\end{document}